\documentclass{amsart}
\usepackage[T1]{fontenc}
\usepackage[a4paper,lmargin={2.3cm},rmargin={2.3cm},tmargin={2.3cm},bmargin = {2.6cm}]{geometry}
\usepackage{xcolor}
\usepackage{amssymb}
\usepackage{amsthm}
\usepackage{graphicx}
\usepackage{enumerate}
\usepackage{amsmath}
\usepackage{dsfont}
\usepackage[mathscr]{euscript}
\usepackage{mathrsfs}
\usepackage{aligned-overset}
\usepackage{bbm}
\usepackage{shuffle}
\usepackage[colorlinks=true, 
linkcolor=OceanBlue, 
filecolor=magenta,       
urlcolor=MildBlue, 
citecolor=orange, 
]{hyperref} 
\definecolor{OceanBlue}{HTML}{014f86}  
\definecolor{rost}{HTML}{9d0208}    
\definecolor{MildBlue}{HTML}{457b9d}  
\definecolor{ForestGreen}{HTML}{386641}
\definecolor{orange}{HTML}{fb5607} 
\usepackage[backend=biber, style=alphabetic, natbib=true, hyperref=true, giveninits=true, maxbibnames=99]{biblatex}
\DeclareSourcemap{
  \maps[datatype=bibtex]{
    \map{
      \step[fieldsource=doi,final]
      \step[fieldset=url,null]
      \step[fieldset=urldate,null]
      \step[fieldset=issn,null]
      \step[fieldset=eprint,null]
      \step[fieldset=eprintclass,null]
      \step[fieldset=eprinttype,null]
      }  
  }
}
\usepackage{csquotes}

\theoremstyle{plain}
\newtheorem{lem}{Lemma}[section]
\newtheorem{theo}[lem]{Theorem}
\newtheorem{cor}[lem]{Corollary}
\newtheorem{defi}[lem]{Definition}

\theoremstyle{definition}
\newtheorem{rem}[lem]{Remark}
\newtheorem{exmp}[lem]{Example}

\usepackage{xargs}  
\usepackage[colorinlistoftodos,prependcaption,textsize=tiny]{todonotes}
\newcommandx{\add}[2][1=]{\todo[inline, disable, linecolor=ForestGreen,backgroundcolor=ForestGreen!25,bordercolor=ForestGreen,#1]{#2}}
\newcommandx{\unsure}[2][1=]{\todo[disable, linecolor=rost,backgroundcolor=rost!25,bordercolor=rost,#1]{#2}}
\todostyle{fine}{disable, color=OceanBlue!25,shadow}
\newcommandx{\todoP}[2][1=]{\todo[linecolor=rost,backgroundcolor=rost!25,bordercolor=rost,#1]{#2}}
\newcommandx{\todoPi}[2][1=]{\todo[inline, linecolor=rost,backgroundcolor=rost!25,bordercolor=rost,#1]{#2}}

\newcommand{\one}{\mathbf{1}}

\newcommand{\dd}{\mathop{}\!\mathrm{d}}
\newcommand{\prob}{\mathbb{P}}
\newcommand{\partition}{\mathcal{P}}

\newcommand{\id}{\mathrm{id}}
\newcommand{\eps}{\varepsilon}

\newcommand{\Lin}{\mathcal{L}}

\newcommand{\as}{\text{a.e}.}

\newcommand{\adm}{\mathcal{A}}

\newcommand{\Xb}{\mathrm{\mathbf{X}}}

\newcommand{\Bcan}{W^{can}}
\newcommand{\w}{\mathrm{w}}
\newcommand{\z}{\mathrm{z}}
\newcommand{\y}{\mathrm{y}}

\newcommand{\Wb}{\mathbf{W}}
\newcommand{\wt}{\mathbbm{w}}
\newcommand{\wb}{\mathbf{w}}
\newcommand{\lift}{\mathrm{Lift}}

\newcommand{\lb}{\left(}
\newcommand{\rb}{\right)}

\newcommand{\Q}{\mathbb{Q}}
\newcommand{\Yb}{\mathbf{Y}}

\newcommand{\vertiii}[1]{{\left\vert\kern-0.25ex\left\vert\kern-0.25ex\left\vert #1 \right\vert\kern-0.25ex\right\vert\kern-0.25ex\right\vert}}
\newcommand{\vb}{\bar{v}}

\newcommand{\sym}[1]{\mathrm{Sym}\left(#1\right)}
\newcommand{\Sym}[1]{\mathrm{Sym}_2(\R^{#1})}
\newcommand{\A}{\mathcal A}
\newcommand{\la}{\langle}
\newcommand{\ra}{\rangle}
\newcommand{\Wtau}{W^{\circ\tau}}
\newcommand{\wtau}{\w^{\circ\tau}}

\newcommand{\Wtaut}{W^{\circ\tau,t}}
\newcommand{\Wtaunull}{W^{\tau(0)}}
\newcommand{\Wtaunullb}{\Wb^{\tau(0)}}
\newcommand{\wtaunull}{\w^{\tau(0)}}

\newcommand{\Wtaunullt}{W^{\tau(t)}}
\newcommand{\Wtaunulltb}{\Wb^{\tau(t)}}

\newcommand{\diag}{\mathrm{diag}}

\newcommand{\cP}[2]{\prob_{#1\mid #2}}
\newcommand{\borelinfty}{\borel_{\vert\cdot\vert_\infty}}
\newcommand{\borelsup}[2]{\borel_{\vert\cdot\vert_\infty}(C([0,#1], \R^{#2}))}

\DeclareMathOperator{\tr}{tr}
\DeclareMathOperator{\negl}{\mathcal{N}}

\DeclareMathOperator{\borel}{\EuScript{B}}
\DeclareMathOperator{\expE}{\mathbb{E}}

\DeclareMathOperator{\F}{\mathcal{F}}
\DeclareMathOperator{\G}{\mathcal{G}}
\newcommand{\X}{\mathbb{X}}

\newcommand{\Xbt}{\tilde{\mathbf{X}}}

\DeclareMathOperator{\W}{\mathbb{W}}
\DeclareMathOperator{\R}{\mathbb{R}}
\DeclareMathOperator{\c12}{\mathbb{C}^{1,2}_b}
\DeclareMathOperator{\N}{\mathbb{N}}

\DeclareMathOperator{\Lt}{L^2}

\newcommand{\RP}{\mathscr{C}}
\DeclareMathOperator*{\esssup}{ess\,sup\ }
\DeclareMathOperator{\esssinf}{ess\,inf\ }

\newcommand\blfootnote[1]{
    \begingroup
    \renewcommand\thefootnote{}\footnote{#1}
    \addtocounter{footnote}{-1}
    \endgroup
}

\title{Causal Hamilton-Jacobi-Bellman Equations for Anticipative Stochastic Optimal Control}
\author{Peter Bank, Franziska Bielert}
\address{Institut f\"ur Mathematik, Technische Universit\"at, Berlin}
\date{}

\begin{document}
\begin{abstract}
We consider a stochastic optimal control problem where the controller can anticipate the evolution of the driving noise over some dynamically changing time window. The controlled state dynamics are understood as a rough differential equation. We combine the martingale optimality principle with a functional form of It\^o's formula to derive a Hamilton-Jacobi-Bellman (HJB) equation for this problem. This HJB equation is formulated in terms of Dupire's functional derivatives and involves a transport equation arising from the anticipativity of the problem. 
\end{abstract}


\maketitle
\tableofcontents

\section{Introduction}

We consider a controller who receives a terminal reward after executing a strategy $\phi$ in the  drift-controlled diffusion system
\begin{align}\label{sec:introduction:cDE}
	Y(t)=y, \quad \dd Y(s)=b(s,Y(s),\phi(s))\dd s + f(Y(s))\dd W(s).
\end{align}
 In contrast to the standard setting, we assume that the controller knows at any time $s \in [t,T)$ the future development of the Brownian motion $W$ until time $\tau(s)>s$, where $\tau$ is a time-change $\tau\colon [0,T]\to[0,T]$. From the controller's perspective the problem of maximizing the expected reward is 
\begin{align}\label{sec:introduction:value}
	v(t,y)=\esssup_\phi \expE(g(Y(T))\mid \F_{\tau(t)}^W).
\end{align} 
A key challenge in this problem is that $v(t,y)$ depends on the evolution of $(W(s), s\in[t,\tau(t)])$ of the Brownian motion over the time window $[t,\tau(t)]$. 
The aim of the present paper is to derive a Hamilton-Jacobi-Bellman equation whose solution $u$ will yield a pathdependent functional that allows us to  represent $v$ in a suitable way. The idea is to combine the martingale optimality principle 
with a pathdependent It\^o-formula that goes back to \cite{dupireFunctionalItoCalculus2009, BallyVlad2016Sibp}. 
To apply this framework we write
\begin{align}\label{sec:introduction:repvalue}
	v(t,y)=u(t,y,W(t),\Wtaunull,\Wtau),
\end{align}
where $\Wtaunull$ records the initial path segment of $W$ up to time $\tau(0)$ and  $\Wtau$ is the time-changed Brownian motion $\Wtau(t)=W(\tau(t))-W(\tau(0))$. More specifically, we  will view $\Wtaunull$ as a parameter for a family of maps $u(\cdot,\Wtaunull,\cdot)$ which are \emph{causal} in the sense that, for every $t\in[0,T]$, $u(t,y,W(t),\Wtaunull,\Wtau)$ depends only on $W(t)$ and $(\Wtau(s), s\in[0,t])$.
We show that, if sufficiently smooth, such a 
\begin{align}\label{sec:introduction:u}
    u\colon [0,T]\times\R^n\times \R^d\times C([0,\tau(0)],\R^d)\times C([0,T],\R^d)\to\R
\end{align} 
is characterized by the following (deterministic) \emph{causal Hamilton-Jacobi-Bellman equation}
    \begin{align}\label{sec:introduction:hjb}
	\begin{split}
			\begin{cases}
				0 &= Du(t,y,\w(t),\wtaunull,\wtau) + \sup_{\varphi\in\R^m}\ \la\nabla_y u(t,y,\w(t),\wtaunull,\wtau), b(t,y,\varphi)\ra\\ 
				&\quad+ \frac12 \tr(f(y)^T\nabla^2_y u(t,y,\w(t),\wtaunull,\wtau) f(y)) + \tr(f(y)^T\nabla^2_{yw} u(t,y,\w(t),\wtaunull,\wtau))\\ 
                &\quad + \frac12 \tr(\nabla_{w}^2 u(t,y,\w(t),\wtaunull,\wtau)) + \frac12\tr(\nabla_\z^2 u(t,y,\w(t),\wtaunull,\wtau))\tau'(t),\\
                0 &= \la\nabla_y u(t,y,\w(t),\wtaunull,\wtau), f(y)\ra + \nabla_w u(t,y,\w(t),\wtaunull,\wtau),\\
				g(y) &= u(T,y,\w(t),\wtaunull,\wtau),
			\end{cases}
	\end{split}
\end{align}
which is to hold for $t\in[0,T]$, $y\in\R^d$ and almost every Brownian path $\w\in C([0,T])$.
Here, $Du$ and $\partial^2_z u$ (linked to $z=\wtau$) denote the causal time and second causal space derivative introduced by B.~Dupire in \parencite{dupireFunctionalItoCalculus2009}. Indeed, the \hyperlink{theo:verification}{Verification theorem} shows that, if there is a suitably regular function $u$ that solves this HJB equation and a suitable feedback candidate for an optimal control, then $u$ yields the problem's value $v(t,y)$ given in \eqref{sec:introduction:value} via~\eqref{sec:introduction:repvalue}. Conversely, Theorem \ref{theo:regularvaluesolvesHJB} shows that, if $v(t,y)$ can be written as in~\eqref{sec:introduction:repvalue} for a suitably regular $u$, then this function satisfies our causal HJB equation. 

Clearly the HJB equation differs from the standard setting as its parabolic first equation in~\eqref{sec:introduction:hjb} uses causal derivatives and the second equation adds a transport equation. Causal derivatives also appear in the HJB equation for optimal control with pathdependent coefficients studied in \parencite{cossoOptimalControlPathdependent2022, cossoPathdependentHamiltonJacobiBellmanEquation2023}. This relates to the work on pathdependent PDEs in non-Markovian setting, that is with random coefficients, \parencite{ekrenViscositySolutionsPath2014}. Another instance of pathdependency is the case of a non-Markovian driver in \eqref{sec:introduction:cDE}, e.g.\ a fractional Brownian motion, then $Y$ does not satisfy the flow property, see \parencite{viensMartingaleApproachFractional2018}. By contrast, our coefficients are standard, the driver simply a Brownian motion and the pathdependency is due to the anticipative information $(\F^W_{\tau(s)})_{s \in [t,T]}$ in \eqref{sec:introduction:value}.

We first establish the martingale optimality principle for the time changed filtration. Then a suitable pathdependent It\^o formula gives the Doob-Meyer decomposition of the supermartingale $[t,T]\ni s\mapsto u(s, Y(s),W(s),\Wtaunull,\Wtau)$ induced by any admissible policy. It reveals that the martingale part is driven by the time-changed Brownian motion $\Wtau$. We show that the contributions from changes of $Y$ (thus of $W$) cannot be compensated by the $\dd s$-term in this It\^o formula. As a consequence the transport term in the second equation of~\eqref{sec:introduction:hjb} emerges. So that only the first two equations together establish a suitable martingale optimality principle.

Let us compare our approach and its HJB equation~\eqref{sec:introduction:hjb} with the literature, at least in the special case $\tau(t) \equiv T$ where the controller knows the full the evolution of $W$. The corresponding \emph{pathwise stochastic control problem} is studied using PDE methods in \cite{MR1647162} and by probabilistic techniques or rough analysis methods in \cite{buckdahnPathwiseStochasticControl2007,caruanaRoughPathwiseApproach2010,frizControlledRoughSDEs2024}, who even add a second Brownian motion that we are omitting here for simplicity. In our notation (and using their Stratonovich dynamics instead of~\eqref{sec:introduction:cDE}), they offer an HJB equation directly for $v$ of~\eqref{sec:introduction:value}:\begin{align}\label{eq:intro-roughHJB}
    -d_t v(t,y) = \sup_{\varphi \in \mathbb{R}} \la\nabla_y v(t,y), b(t,y,\varphi)\ra \dd t+ \la \nabla_y v(t,y),f(y)\ra  \circ \dd W(t).
\end{align}
In \cite{buckdahnPathwiseStochasticControl2007}, this equation is given a rigorous meaning by a Doss-Sussmann flow transformation. In the new coordinates the troublesome driver $\dd W(t)$ disappears and one finds a partial differential equation, albeit with rather complex random coefficients. The paper \cite{frizControlledRoughSDEs2024} views~\eqref{eq:intro-roughHJB} as a rough partial differential equation. By contrast, our HJB equation is based on the original coefficients and uses comparably easy to compute causal derivatives allowing for easy verification arguments.

The analysis of our causal HJB equation requires some technical preparations. A first issue is to rigorously address the anticipativity in the system dynamics~\eqref{sec:introduction:cDE} where we follow the approach by \cite{frizControlledRoughSDEs2024}: Since the control $\phi$ can anticipate part of the future evolution of the Brownian motion $W$, so will the system $Y$ and the volatility $f(Y)$ in general. This takes the integral of $f(Y)$ against $\dd W$ outside the scope of It\^o's theory of stochastic differential equations. With sufficient regularity assumptions on $f$ and $b$, rough analysis (see \cite{RoughBook} for an introduction)  gives~\eqref{sec:introduction:cDE} a pathwise meaning by considering the It\^o rough path lift $\mathbf{W}$ of $W$. This perspective also calls for a suitable pathdependent It\^o formula that extends  \cite{dupireFunctionalItoCalculus2009, BallyVlad2016Sibp}  to the rough paths setting; see~\cite{bielertRoughFunctionalIto2024, cuchieroFunctionalItoformulaTaylor2025}.

With the anticipative system dynamics understood, we then investigate the probabilistic properties of the value function~\eqref{sec:introduction:value}. Section~\ref{sec:dpp} establishes the \hyperlink{dpp}{dynamic programming principle}. In particular Lemma \ref{lem:updirec} shows that there exists a maximizing sequence for \eqref{sec:introduction:value}, comparable to \parencite[Theorem 4.2, (ii)]{buckdahnPathwiseStochasticControl2007}. 

Section~\ref{sec:hjb} focuses on the causal HJB equation. We start by introducing several concepts to understand causal partial differential equations. To apply a rough functional It\^o formula in our pathwise setting we have to understand the covariation between the solution $Y$, thus $W$, and the time-changed Brownian motion $\Wtau$. Since $W$ and $\Wtau$ are not be semimartingales with respect to the same filtration, their stochastic quadratic variation cannot be constructed by classical stochastic analysis. Instead we follow \parencite{contPathwiseIntegrationChange2019} for a pathwise notion of quadratic variation in Section~\ref{sec:sec:qv}. In particular Lemma \ref{lem:2varW} shows that (in this pathwise sense) almost surely $[W,\Wtau] = 0$, $[W]=\id$ and $[\Wtau]=\tau-\tau(0)$. 

Next, Section \ref{sec:func} summarizes some concepts from functional calculus for causal maps from \parencite{dupireFunctionalItoCalculus2009, oberhauserExtensionFunctionalIto2012, BallyVlad2016Sibp, bielertRoughFunctionalIto2024}. Also we apply the rough functional It\^o formula of~\cite{bielertRoughFunctionalIto2024} with It\^o dynamic corrections to causal maps of the form $F(\id,Y,\Wtau)$ where $Y$ is a solution to~\eqref{sec:introduction:cDE}. We use the rough functional It\^o formula from \parencite{ bielertRoughFunctionalIto2024}, because rough integration is an intrinsic integral notion and provides explicit estimates that simplify proofs, whereas in \parencite{dupireFunctionalItoCalculus2009, oberhauserExtensionFunctionalIto2012, BallyVlad2016Sibp} the integral notion is via an indirect approach due to F\"ollmer \parencite{Föllmer1981}. We are then in the position to prove in the \hyperlink{theo:verification}{Verification Theorem~\ref{theo:verification}} that if a suitable map $u$ of the form \eqref{sec:introduction:u} satisfies the causal HJB-equation~\eqref{sec:introduction:hjb}, then it yields a modification of the value function \eqref{sec:introduction:value} via~\eqref{sec:introduction:repvalue}. To prove that the HJB equation is also necessary for regular value functions, we have to resolve some technical difficulties, which is the purpose of Section~\ref{sec:sec:regvalueHJB}. The main challenge is that we need to establish our results for almost every Brownian path and thus need to be careful with the choice of exceptional null sets. 
We formulate the stochastic \hyperlink{dpp}{dynamic programming principle} for equivalence classes of random variables, but will need to draw conclusions for and work with functions $\vb$ evaluated at Brownian paths. To that end we carefully construct conditional distributions of $(Y,\phi,\Wtaunullb,\Wtau)$ given $\F_{\tau(0)}$, see Section \ref{sec:sec:sec:condDist}. Another technical challenge that arises from the combination of deterministic and probabilistic concepts is resolved in Lemma~\ref{lem:qvCondExp}. There, we calculate the (pathwise) quadratic variation of the rough integral (from the rough functional It\^o formula) conditioned on the time-changed filtration $\F_{\tau(t)}$. We end Section~\ref{sec:sec:regvalueHJB} with Theorem \ref{theo:regularvaluesolvesHJB} showing the necessity of the HJB equation. For that we adapt some ideas previously used in optimal control with pathdependent coefficients in the controlled diffusion from \parencite{cossoOptimalControlPathdependent2022}. 

Section~\ref{sec:exmp} illustrates our approach by some examples. We revisit \parencite{insiderbank, bankOptimalInvestmentNoisy2024} and their utility maximization problem for an investor who can slightly peek ahead in the stock price evolution. Here, our dynamic programming approach and the causal HJB equation complement the convex duality approach taken there. Example \ref{exmp:2} and \ref{exmp:3} are originally from pathwise stochastic control. 
Actually Example \ref{exmp:3} goes beyond the scope of the present paper since it involves an extra controlled diffusion $B$ (independent of $W$ for which one has anticipative information) as studied by~\cite{buckdahnPathwiseStochasticControl2007,caruanaRoughPathwiseApproach2010,frizControlledRoughSDEs2024}. It is included as an outlook about a generalization of our causal HJB approach in this direction.

\section{Notation}\label{sec:notation}
Let $(\Omega, \F, \prob)$ be a complete probability space.
Let $L^0(\prob)$ denote the set of equivalence classes of $\F$ measurable maps $X\colon \Omega\to\R\cup\{+\infty\}$. For $X,Y\in L^0(\prob)$ we write $X\le Y$, if for any two representatives the inequality holds $\prob$-\as . Clearly $(L^0(\prob), \le)$ is an ordered set. Then $X^*\in L^0(\prob)$ is the supremum of $\mathcal{X} \subset L^0(\prob)$ if it is the smallest upper bound for $\mathcal X$, i.e.
\begin{align}\label{def:esssup}
\forall X\in\mathcal X,\ X\le X^*, \quad \exists Z\in L^0(\prob)\colon \forall X\in\mathcal X, X\le Z \Rightarrow X^*\le Z.
\end{align}
For an introduction we recommend \parencite{lowtherEssentialSuprema2019}.
We emphasize that in the case of nonnegative random variables, spelling out the properties for representatives, each representative of $X^*$ is the $\prob$-\as\ defined essential supremum of the family of representatives of $\mathcal X$ studied  in \parencite[Definition A.1, Theorem A.3]{karatzasMethodsMathematicalFinance1998}.  The essential supremum exists and it is easy to show that for any two choices of representative families $\mathcal X_1$, $\mathcal X_2$, i.e. $\mathcal X_1$ contains for each $X\in\mathcal X$, one representative  $X_1\in X$ and similar for $\mathcal X_2$, it holds $\prob$-\as\ that $X_1 = X_2$. Therefore we abuse the notation and also write 
\begin{align*}
X^* = \esssup \mathcal X.
\end{align*}
Further we set for $x,y\in\R$, $x\vee y = \max\{ x,y\}$ and for $X,Y\in L^0(\prob)$, $X\vee Y$ denotes the equivalence class of the pointwise maximum of representatives.
Moreover for a sub-$\sigma$ algebra $\G\subset \F$ and $X\in L^1(\Omega, \F, \prob)$ (resp. $X\ge 0$) we always consider the conditional expectation $\expE(X\mid \G)\in L^1(\Omega, \G, \prob)$ (resp. $\expE(X\mid \G)\in L^0(\Omega, \G, \prob)$). We mention that standard properties of the conditional expectation also hold for nonnegative random variables with infinite expectation.
The set $L^0([0,T],\R^m)$ of equivalence classes of Borel measurable maps $X\colon [0,T]\to\R^m$ is equipped with the Ky-Fan metric $d_{KF}$ which metrizes convergence in Lebesgue measure. 

The Euclidean norm on $\R^d$ is denoted by $\vert\cdot\vert$.
Let $T>0$ and $D([0,T],\R^d)$ be the set of c\`adl\`ag paths $X\colon [0, T] \to \R^d$ equipped with the uniform norm $\vert \cdot \vert_\infty$. For such paths and $t\in [0,T]$ we denote by $X(t)$ the value of the path at time $t$ and by $X_t$ the stopped path $X_t = X(\cdot \wedge t)\in D([0,T],\R^d)$. 
We call $\partition = \{ [t_{k-1}, t_k]\colon k=1,\dots,n\}$ with $t_k\in[0,T]$ for all $k=0,\dots,n$, \textit{partition of $[0,T]$} if $0=t_0< t_1<\dots < t_n = T$. The \textit{mesh} of a partition $\partition$ is defined as $\vert\partition\vert = \max_{[s,t]\in\partition} \vert t-s\vert$. \index{interval@$\vert\partition\vert$} Further let $\Delta_T:= \{ (s,t)\in [0,T]\times [0,T]\colon 0\le s\le t\le T\}$.

Let $C([0,T],\R^d)$ denote the the set of continuous paths $X\colon[0,T]\to\R^d$. Equipped with the uniform norm $\vert\cdot\vert_\infty$ it is complete separable space, thus Polish (as a topological space). We denote by $\borelsup{T}{d}$ the Borel $\sigma$-algebra (generated by norm topology).

For $\alpha\in(0,1)$, a two-parameter path $\Xi\colon\Delta_T\to \R^d$ is $\alpha$-H\"older continuous if 
\begin{align*}
	\vert \Xi\vert_\alpha := \sup_{\substack{(s,t)\in\Delta_T,\\ s<t}} \frac{\vert\Xi(s,t)\vert}{\vert t-s\vert^\alpha} < \infty.
\end{align*}
Then a path $X\colon [0,T]\to \R^d$ is $\alpha$-H\"older continuous if its increments $(\delta X)(s,t):= X(t) - X(s)$ are and their collection is $C^\alpha([0,T],\R^d)$.

The bracket $\la\cdot,\cdot\ra$ denotes the dual pairing according to the context. In particular we use the notation for the Euclidean scalar product in $\R^d$  and the Frobenius scalar product in $\R^{n\times d}$. The space of bounded linear operators from $\R^{l\times d}$ to $\R^{m\times n}$ is denoted by $\Lin(\R^{l\times d},\R^{m\times n})$. As usual we identify $\Lin(\R^d, \R^n)\cong\R^{n\times d}$ and $\Lin(\R^n,\Lin(\R^n,\R))\cong\R^{n\times n}$. Moreover we write for $A\in\R^{n\times m}$, $B\in\Lin(\R^{l\times d}, \R^{n\times m})$, $\la A,B\ra\in\R^{l\times d}$ for the matrix defined via the linear operator $\la A,B\ra X:= \la A,BX\ra$ for $X\in\R^{l\times d}$. And further for $A,B\in\R^{n\times d}$, $A\otimes B\in\Lin(\R^{d\times d},\R^{n\times n})$ is defined by setting for $X\in\R^{d\times d}$, $(A\otimes B) X :=AXB^T$ (so in particular for $A,B\in\R^n$, simply $AB^T\in\R^{n\times n}$.)
The trace of a matrix $A\in\R^{n\times n}$ is denoted by $\tr(A)$.

For a function $f\colon\R^l\to\R^{n\times m}$ we denote its derivative, if it exists, by $\nabla f\colon \R^l \to \Lin(\R^l,\R^{n\times m})$ and similar for higher order derivatives. 

For two terms $x,y$ we abbreviate the existence of some constant $C>0$ such that $x\le C y$ to $x\lesssim y$ and by $\lesssim_p$ we indicate a dependency $C= C(p)$ on some parameter $p$.\index{$\lesssim$}

Let $(S,\mathcal S)$ be a measurable space and $X\colon\Omega\to S$ measurable, then $\prob^X\colon \mathcal{S}\to[0,1]$ denotes the law of $X$ under $\prob$. For a sub-$\sigma$-algebra $\G\subset \F$, $\prob_{\mid\G}$ denotes the restriction of $\prob$ from $\F$ to $\G$. Finally, $\lambda$ will be used to denote the Lebesgue measure on the real line.

\section{Optimal Control Problem With Anticipative Controls}\label{sec:condProblem}

In this section we rigorously understand the anticipative controlled state dynamics as a randomized rough differential equation. To that end we briefly recall results from deterministic rough analysis in Section~\ref{sec:condProblem:sec:stateEq}. The random setting is specified in Section~\ref{sec:condProblem:sec:StochDyn}, where the driver is a Brownian motion whose evolution is partially anticipated by the 
admissible controls. Section \ref{sec:condProblem:sec:valuePro} defines the value function of our optimizatio problem.

\subsection{State Equation}\label{sec:condProblem:sec:stateEq}
  We consider the following set up for a controlled deterministic dynamic. Let $X\colon[0,T]\to \R^d$ be a path, $b$, $f$ coefficients and $\phi$ a control of the form
 \begin{align*}
 	b\colon [0,T] \times \R^n \times \R^m \to \R^n, \quad f\colon \R^n \to \R^{n\times d}, \quad 	\phi\colon [0,T] \to \R^m.
 \end{align*} 
 For $T>0$ and initial datum $(t,y)\in[0,T]\times\R^n$, $Y^{t,y,\phi}\colon [0,T]\to\R^n$ denotes the path with controlled dynamic
 \begin{align}
 	\begin{split}\label{cDE}
 		\dd Y(s) &= b(s, Y(s), \phi(s)) \dd s + f(Y(s)) \dd X(s),\quad s\in (t, T],\\
 		Y(t) &= y.
 	\end{split}
 \end{align}

\subsubsection{Rough Differential Equations}
We briefly introduce rough paths and controlled rough paths to H\"older continuity parameters $\alpha, \beta>0$. For a thorough treatment of rough path theory we refer to  \parencite{RoughBook}, in particular Chapters 2 and 4. For simplicity they use one H\"older parameter. For a setting with several other variations of parameters see \parencite{frizRoughStochasticDifferential2024}.

\begin{defi}[$\alpha$-H\"older Rough Path]
	Let $\alpha\in(1/3,1/2]$, $X\colon[0,T]\to \R^d$, $\X\colon\Delta_T\to\R^{d\times d}$. Then $\Xb = (X, \X)$ is called $\alpha$-H\"older rough path if $X$ (resp. $\X$) is $\alpha$-(resp. $2\alpha$)-H\"older continuous and Chen's relation holds for all $(s,u), (u,t)\in\Delta_T$: 
	\begin{align*}
		\X(s,t)-\X(s,u)-\X(u,t)=(X(u)-X(s))\otimes(X(t)-X(u)).
	\end{align*}
\end{defi}

Equipped with the metric 
\begin{align}\label{defi:rpmetric}
	\rho_\alpha(\Xb, \Xbt)=\vert X-\tilde{X}\vert_\alpha + \vert\X-\tilde{\X}\vert_{2\alpha},
\end{align}
the set of $\alpha$-H\"older rough paths over $\R^d$ is a complete metric space.
Let $\mathcal C^\infty$ denote the subset of \emph{smooth rough paths}, i.e. all $\alpha$-H\"older rough paths $(X,\X)$, such that $X$ and, for every $s\in[0,T]$, $\X(s,\cdot)$ are smooth. Further write $\RP_T$ for the closure of $\mathcal C^\infty$ with respect to $\rho_\alpha$. Then $\RP_T$ is a Polish space, compare \parencite[Exercise 2.8]{RoughBook}, and we write $\mathfrak{C}_T$ for its Borel $\sigma$-algebra.\\

\begin{exmp}[It\^o Lift]\label{exmp:lift}
	Let $\alpha\in(1/3,1/2)$. 
	\begin{enumerate}[i)]
		\item It is commonly used that $\prob$-\as\ sample path of a $d$-dimensional Brownian motion $W=(W^1,\dots, W^d)$ induces an $\alpha$-H\"older continuous rough path $\Wb$, cf. \parencite[Chapter 3]{RoughBook}, using iterated It\^o integrals for $\W$. Namely one sets for $i,j=1,\dots,d$,
        \begin{align*}
            \W^{i,j}(s,t):= \lb \int_s^t (W^i(r) - W^i(s))\dd W^j(s)\rb(\omega)
        \end{align*}
        and $\W=(\W^{i,j})_{i,j=1,\dots,d}$.
		Then the Kolmogorov criterion for rough paths \parencite[Theorem 3.1, Proposition 3.4]{RoughBook} gives a null set $ N\in\F_0$, such that for every $\omega\in\Omega\setminus N$, $W$ (resp. $\W$) is $\alpha$ (resp. $2\alpha$)-H\"older continuous and for all $q<\infty$,
		\begin{align*}
			\vert W\vert_\alpha+\sqrt{\vert\W\vert_{2\alpha}}\in L^{q}(\Omega, \F, \prob).
		\end{align*}
		Setting $W$ and $\W$ on $ N$ to zero, every sample path of $\Wb := (W, \W)$ is an $\alpha$-H\"older rough path, called \textit{Brownian rough path}. Combining  \parencite[Proposition 2.8, Proposition 3.4]{RoughBook} and the fact that $(W,\W+\la W\ra)$ is a geometric rough path and $\la W\ra$ smooth, Brownian rough paths can be approximated by smooth rough paths, hence $\Wb\colon\Omega\to\RP_T$.
		In particular it holds for every $t\in[0,T]$ that the lift $\omega\mapsto \Wb_t(\omega)$ is $\F_t-\mathfrak{C}_t$ measurable.
		\item In the same way we define a lift on the path space $C([0,T],\R^d)$: 
		 $(C([0,T], \R^d), \borel_{\vert\cdot\vert_\infty}(C([0,T],\R^d)))$ equipped with the Wiener measure $\prob^W$ gives rise to the Brownian motion, namely
		\[ \begin{split}
			\Bcan\colon C([0,T], \R^d) \times [0,T] &\to \R^d\\
			(\w, t)&\mapsto \w(t).
		\end{split}\]
		Further we write $\negl$ for the class of $W$-negligible sets and $\borel_{\vert\cdot\vert_\infty}(C([0,T],\R^d))\vee \negl$ for the completion.
		Using part i), we  may define the lift as a function of $\w$. Namely we denote by $N$ the $\prob^W$-negligible set where the lift of $B^{can}$ is not defined
		and let
		\begin{equation}\label{defi:lift}
			\begin{split}
				\lift\colon C([0,T], \R^d) &\to \RP_T\\
				\w&\mapsto \begin{cases}
					\mathbf{0} =(0,0)&\ \w\in N,\\
					\wb = (\w, \wt)& \text{ else,}
				\end{cases}
			\end{split}
		\end{equation}
		where $\wb = \Wb^{can}(\w)$, i.e.\ $\wt = \W^{can}(\w)$ and $\W^{can}$ the iterated It\^{o} integrals as before. Clearly it follows that $\lift$ is $\borel_{\vert\cdot\vert_\infty}(C([0,T],\R^d))\vee\negl-\mathfrak{C}_T$ measurable.
	\end{enumerate}
\end{exmp}

We next define the class of paths that we can integrate against $\alpha$-H\"older rough paths.

\begin{defi}[$(\alpha, \beta)$-H\"older $X$-controlled Rough Path]\label{defi:cRP}
	Let $\alpha\in(1/3, 1/2]$, $\beta\in(0,1]$ and $X$ an $\alpha$-H\"older rough path. Further let $Y\colon[0,T]\to\R^{n\times l}$, $Y'\colon[0,T]\to\Lin(\R^d, \R^{n\times l})$. Then $(Y, Y')$ is called $(\alpha, \beta)$-H\"older $X$-controlled rough path, if $Y'$ is $\alpha$-H\"older continuous and the remainder $R^{Y, X}\colon\Delta_T\to\R^{n\times l}$, defined for $(s,t)\in\Delta_T$ by 
	\begin{align*}
		R^{Y,X}(s,t) = Y(t)-Y(s)-Y'(s)(X(t)-X(s)),
	\end{align*}
	is $\alpha+\beta$-H\"older continuous. If $\beta=\alpha$ we abbreviate to $\alpha$-H\"older $X$-controlled rough path.
\end{defi}

The construction of the integral of $(Y, Y')$ with respect to $\Xb$ goes back to \parencite[Proposition 1]{gubinelliControllingRoughPaths2004} and can be found in many variations, e.g.\ \parencite[Theorem 4.10]{RoughBook}. For us the following suffices:

\begin{theo}[Rough Integration]\label{theo:RI}
	Let $\Xb=(X,\X)$ be an $\alpha$-H\"older rough path over $\R^d$ and $(Y,Y')$ be an $(\alpha,\beta)$-H\"older $X$-controlled rough path over $\R^{n\times d}$ such that $\beta\in(0,\alpha]$ and $2\alpha+\beta>1$. Then 
	\begin{align*}
		\int_0^T Y\dd\Xb := \lim_{\vert\partition\vert\to 0} \sum_{[s,t]\in\partition} Y(s)(X(t)-X(s)) + Y'(s)\X(s,t)
	\end{align*}
	is a well-defined limit along partitions $\partition$ of $[0,T]$ that we call rough integral. Moreover the following estimate holds: for $(s,t)\in\Delta_T$,
	\begin{align*}
		\left\vert \int_s^t Y\dd\Xb - Y(s)(X(t)-X(s)) - Y'(s)\X(s,t)\right\vert\lesssim_{T,\alpha,\beta} (\vert X\vert_\alpha\vert R^{Y,X}\vert_{\alpha+\beta} +\vert \X\vert_{2\alpha}\vert Y'\vert_{\alpha}) \vert t-s\vert^{2\alpha +\beta},
	\end{align*}
	and $\lb \int_0^\cdot Y\dd\Xb, Y\rb$ is $(\alpha, \alpha)$-H\"older $X$-controlled with
	\begin{align*}
		\left\vert \int_s^t Y\dd\Xb\right\vert\lesssim \big(\vert X\vert_\alpha(\vert R^{Y,X}\vert_{2\alpha} + \vert Y(0)\vert + T^\alpha\vert Y\vert_\alpha) + \vert\X\vert_{2\alpha}(\vert Y'(0)\vert + (1\vee T^\alpha)\vert Y'\vert_\alpha)\big)\vert t-s\vert^\alpha.
	\end{align*}
\end{theo}

To make sense of ~\eqref{cDE}, we need to integrate $f(Y^{t,y,\phi})$ against an $\alpha$-H\"older rough path $X$. Let $f\in C^3_b(\R^n, \R^{n\times d})$. For an $\alpha$-H\"older $X$-controlled rough path $(Y,Y')$ given by Definition \ref{defi:cRP} with $l=1$ the composition of $f$ with $(Y,Y')$ by
\begin{align*}
	(f(Y), \nabla_yf(Y)Y')
\end{align*}
yields again an $\alpha$-H\"older $X$-controlled rough path, cf. \parencite[Lemma 7.3]{RoughBook}. Note $\nabla_yf(y)\in\Lin(\R^n,\R^{n\times d})$, $Y'\colon [0,T]\to\Lin(\R^d,\R^n)$, so that $\nabla_yf(Y)Y'\colon [0,T]\to\Lin(\R^d,\R^{d\times n})$. In particular the following estimates hold for $(s,t)\in\Delta_T$,
\begin{align}\label{estimate:sigma(Y)}
	\begin{split}
		\vert f(Y(t))-f(Y(s))\vert &\lesssim_{\vert Y\vert_\alpha,\vert\nabla_yf\vert_\infty} \vert t-s\vert^\alpha,\\
		\vert \nabla_yf(Y(t))\cdot f(Y(t)) - \nabla_yf(Y(s))\cdot f(Y(s))\vert&\lesssim_{\vert f\vert_{C^2_b}, \vert Y\vert_\alpha} \vert t-s\vert^\alpha,\\
		\vert R^{f(Y), X}(s,t)\vert &\lesssim_{\vert f\vert_{C^2_b}, \vert Y\vert_\alpha, \vert R^{Y,X}\vert_{2\alpha}} \vert t-s\vert^{2\alpha}.
	\end{split}
\end{align}

We now collect some  results from rough path theory which are useful in our setting.
\begin{theo}[Existence and Properties of Solutions]\label{theo:sol2cSDE} Let $\alpha\in(1/3,1/2)$ and $\Xb = (X, \X)$ be an $\alpha$-H\"older rough path. Let $b\colon [0,T]\times \R^d\times\R^m\to \R^d$ be bounded, measurable and  uniformly Lipschitz continuous in $y\in\R^n$, i.e., there exists $L>0$, such that for all $y,\tilde y\in \R^d$,
	\begin{align}\label{b}
		\sup_{\varphi\in\R^m}\sup_{t\in[0,T]} \vert b(t,y,\varphi)-b(t,\tilde y, \tilde\varphi)\vert\le L \vert y-\tilde y\vert\
	\end{align}
	and $f\in C^3_b(\R^n, \R^{n\times d})$.
\begin{enumerate}[i)]
	\item  Then for every $y\in\R^n$ and $\phi\in L^0([0,T],\R^m)$ there exists a unique $\alpha$-H\"older $X$-controlled rough path $(Y, f(Y))$ that solves \eqref{cDE} started in $t=0$, i.e.\ for all $t\in[0,T]$,
	\begin{align*}
		Y(t) = y + \int_0^t b(\id, Y, \phi) \dd\lambda + \int_0^t f(Y) \dd\Xb,
	\end{align*}
	where the last integral is a well-defined rough integral of the $\alpha$-H\"older $X$-controlled path $(f(Y), \nabla_yf(Y)\cdot f(Y))$ against the rough path $\Xb$.
	\item The solution $Y$ satisfies the a priori estimate 
	\begin{align*}
		\vert Y\vert_\alpha \lesssim_{\vert b\vert_\infty, \vert f\vert_{C^2_b}} (\vert\Xb\vert_\alpha + T^{1-2\alpha})\vee(\vert\Xb\vert_\alpha + T^{1-2\alpha})^{1/\alpha} 
	\end{align*}
	and
	\begin{align*}
		\vert R^{Y,X}\vert_{2\alpha} \lesssim_{T, \vert b\vert_\infty, \vert f\vert_{C^2_b}} (1+ \vert\Xb\vert_\alpha + T^{1-2\alpha})^2\vee (1+\vert\Xb\vert_\alpha+ T^{1-2\alpha})^{2/\alpha}
	\end{align*}
	with $\vert\Xb\vert_\alpha := \vert X\vert_\alpha + \sqrt{\vert\X\vert_{2\alpha}}$. 
	\item Write $Y^{t,y,\phi, \Xb}:= Y$ for the solution to \eqref{cSDE}. Then for fixed initial data $(t,y)\in[0,T]\times \R^n$ the solution map
	\begin{align}\label{theo:sol2cSDE:solmap}
		F\colon L^0([0,T],\R^m)\times \RP_T \to C([0,T],\R^n),\quad 
		(\varphi,\Xb)\mapsto Y^{t,y,\varphi,\Xb}
	\end{align}
	is $\borel_{\alpha}(L^0([0,T],\R^m))\otimes\mathfrak C^\alpha_T - \borel_{\vert\cdot\vert_\infty}(C[0,T],\R^n))$ measurable.
\end{enumerate}
\end{theo}

\begin{proof}
	\begin{enumerate}[i)]
		\item For each $\phi\in \mathcal L^0([0,T],\R^m)$, existence and uniqueness follows from \parencite[Theorem 8.3]{RoughBook}, where we emphasize that there is no problem to include a drift coefficient $b$ satisfying the conditions above. A proof in the general setting of rough stochastic differential equations can be found in \parencite[Definition 4.1, Theorem 4.6]{frizRoughStochasticDifferential2024}. Then note that the the solution $Y$ is independent of the representative $\phi$.
		\item The a priori estimates for $\vert Y\vert_\alpha$ follow from including the drift in \parencite[Proposition 8.2]{RoughBook} (namely adding $\vert b\vert_\infty\vert t-s\vert$ in eq. $(8.3)$ in the proof). The estimate for $\vert R^{Y,X}\vert_{2\alpha}$ follows from iterating eq. $(8.5)$ with included drift ($\vert R^{Y,X}\vert_{2\alpha, h} \lesssim_{\vert b\vert_\infty, \vert f\vert_{C^2_b}} \vert \X\vert_{2\alpha, h} + \vert Y\vert_{2\alpha, h}^2 + 2h^{1-2\alpha}\lesssim (\vert\Xb\vert_\alpha + T^{1-2\alpha})^2$) in the manner of \parencite[Exercise 4.5]{RoughBook}, as done in \parencite[Lemma 3.2]{bielertRoughFunctionalIto2024}.
		\item Follows from the measurability results in the more involved setting of controlled rough stochastic differential equations applied with the measurable parameter space $U=L^0([0,T],\R^m)\times\RP_T$, $\mathfrak U = \borel_{\alpha}(L^0([0,T],\R^m))\otimes\mathfrak C_T$. By \parencite[Theorem 4.5]{frizControlledRoughSDEs2024}, the map $(s,\phi,\Xb)\mapsto Y^{t,y,\phi,\Xb}(s)$ is $\borel_{\vert\cdot\vert}([t,T])\otimes\borel_{\alpha}(L^0([0,T],\R^m))\otimes\mathfrak C_T - \borel_{\vert\cdot\vert}(\R^n)$-measurable (noting that the optional $\sigma$-algebra reduces to $\borel_{\vert\cdot\vert}([t,T])$). The claim follows since $Y^{t,y,\phi,\Xb}$ is continuous and $\borel_{\vert\cdot\vert_\infty}(C([0,T],\R^n))$ generated by the point evaluations.
	\end{enumerate}
\end{proof}

\subsection{Stochastic Dynamics with Anticipative Controls}\label{sec:condProblem:sec:StochDyn}
Let $W\colon\Omega\times [0,T]\to \R^d$ be a Brownian motion and $(\F_t)_{t\in[0,T]}$ denote the completed canonical filtration of $W$. We now randomize the controlled differential equation \eqref{cDE} to 
\begin{align}
	\begin{split}\label{cSDE}
		\dd Y(s) &= b(s, Y(s), \phi(s)) \dd s + f(Y(s)) \dd \Wb(s),\quad s\in (t, T],\\
		Y(t) &= y,
	\end{split}
\end{align}
where $\Wb$ denotes the Brownian rough path described in Example \ref{exmp:lift} and $\phi\colon\Omega\times[0,T]\to \R^m$ is a random control. The advantage of solving the equation pathwise as a rough differential equation is that it allows to consider anticipative controls $\phi$: Assume that at any point of time $t\in[0,T]$ one has partial knowledge of the future. We model the peek into the future of $W$ by a \textit{time-change}\index{interval@$\tau$} $\tau$,  that is by a continuous, non-decreasing map \[\tau\colon [0,T]\to [0,T]\] which is bounded from below by the \textit{identity} $\id\colon [0,T]\to[0,T]$ with $\tau(T)=T$. Then the information flow is given by the
\textit{time changed filtration} \[\F_\tau := (\F_{\tau(t)})_{t\in[0,T]}\] 
and the class of \hypertarget{adm}{\textit{admissible controls}} $\adm$ consist of all $\phi\colon\Omega\times[0,T]\to \R^m$ that are progressively measurable with respect to $\F_\tau$, i.e.\ for every $t\in[0,T]$, $\Omega\times[0,t]\ni(\omega, s)\mapsto \phi(\omega, s)$ is $\F_{\tau(t)}\otimes\borel([0,t])$-measurable. 
We emphasize that if a solution $Y^{t, y,\phi}$ of \eqref{cSDE} exists, it is adapted to the time-changed filtration $\F_\tau$ due to the anticipating controls $\phi\in\adm$. Hence integration of $f(Y^{t, y,\phi})$ against the Brownian motion $W$ is \textbf{not} covered by  It\^o integration.

\begin{exmp}\label{exmp:tau}
	\begin{enumerate}[i)]
		\item \parencite{insiderbank} considers a utility maximization problem with $\tau(t)= (t+\Delta)\wedge T$ for some $\Delta>0$.
		\item  The case $\tau(t) \equiv T$ corresponds to full knowledge of $W$ from the start. This case is considered in \parencite{buckdahnPathwiseStochasticControl2007} and recently in \parencite{frizControlledRoughSDEs2024}. They both include another Brownian motion $B$ that is independent of $W$ and known up to $t$ as in the classic stochastic control setting (note that in the present paper the roles of $B$ and $W$ are exchanged).
	\end{enumerate}
\end{exmp}

\begin{rem}[Measurability and Random Initial Data]
	Recalling the construction of Brownian rough paths in Example \ref{exmp:lift}, it is clear that on $ N$, it holds that $\int_t^. f(Y^{t,y,\phi})\dd\Wb = 0$ and on $N^C$, $\int_t^sf(Y^{t,y,\phi})\dd\Wb$ is $\F_{\tau(s)}$-measurable (as the pointwise limit of $\F_{\tau(t)}$-measurable random variables by Theorem \ref{theo:RI}). Since the filtration is complete this implies that $Y^{t,y, \phi}$ is adapted to the time-changed filtration $\F_\tau$.
	Note also that since \eqref{cSDE} is solved pathwise and since the dependece on the initial condition $y\in\R^n$ is continuous by \parencite[Theorem 8.5]{RoughBook} (applicable as in the proof of Theorem \ref{theo:sol2cSDE} part i)), $y\in\R^n$ may be replaced by a random initial condition. If $y\colon\Omega\to \R^n$ is $\F_{\tau(t)}$-measurable, then $Y^{t,y,\phi}$ is still adapted to $\F_\tau$. It is further reasonable to allow $y\in L^0(\Omega,\F_{\tau(t)}, \prob; \R^n)$. Indeed for representatives $\tilde y$ and $\bar y$ of $y$, $\vert Y^{t,\tilde y,\phi}-Y^{t,\tilde y,\phi}\vert_\infty \lesssim \vert\tilde y -\bar y\vert = 0$ $\prob$-as, hence the solutions are indistinguishable. 
\end{rem}

\subsection{Value Process}\label{sec:condProblem:sec:valuePro}
For a \textit{reward function} $g\colon \R^n \to [0, \infty)$, the \textit{stochastic optimal control problem with anticipative controls} is to maximize the random \textit{reward functional}
\begin{align*}
	J\colon [0,T]\times L^0(\Omega,\F, \prob; \R^n)\times \adm &\to L^0(\prob),\\
	(t,y,\phi)&\mapsto \expE \lb g(Y^{t,y,\phi}(T)) \mid \F_{\tau(t)} \rb.
\end{align*}
over the \hyperlink{adm}{admissible controls} $\adm$.
The \textit{value function} is then given by
\begin{align}\label{valuefct}
	\begin{split}
		v\colon [0,T]\times L^0(\Omega,\F, \prob; \R^n)&\to L^0(\prob),\\
		(t, y)&\mapsto \esssup_{\phi\in \adm} J(t,y,\phi),
	\end{split}
\end{align}
where the essential supremum is defined in \eqref{def:esssup}. Note that by construction $v(t,y)$ is $\F_{\tau(t)}$-measurable.

\section{Dynamic Programming Principle}\label{sec:dpp}
We prove a dynamic programming principle for the value function \eqref{valuefct} with respect to the time changed filtration $\F_\tau$ to obtain the martingale optimality principle. 

\begin{lem}\label{lem:updirec} For $y\in L^0(\Omega,\F_{\tau(s)}, \prob; \R^n)$, $t\in[0,T]$ and $s\in[t,T]$, the families
		\begin{align*}
			\lb \expE\lb g\lb Y^{s, y, \phi}(T)\rb\mid \F_{\tau(s)} \rb \rb_{\phi\in\adm},\quad \lb \expE\lb v\lb s, Y^{t, y, \phi}(s)\rb\mid \F_{\tau(t)} \rb \rb_{\phi\in\adm}
		\end{align*}
 		are closed under pairwise maximization.
\end{lem}

\begin{proof}
	We first argue that the family
	\begin{align}\label{theo:dpp:eq:2}
		\lb \expE\lb g\lb Y^{s, y, \phi}(T)\rb\mid \F_{\tau(s)} \rb \rb_{\phi\in\adm}
	\end{align}
	is closed under pairwise maximization. Let $\phi, \psi\in\adm$, choose representatives
	\begin{align*}
		X^\phi \in \expE\lb g\lb Y^{s, y, \phi}(T)\rb\mid \F_{\tau(s)} \rb,\quad X^\psi \in \expE\lb g\lb Y^{s, y, \psi}(T)\rb\mid \F_{\tau(s)} \rb
	\end{align*} 
	and consider $A := \{X^\phi > X^\psi\}\in\F_{\tau(s)}$. Then $\xi := (\phi\one_A + \psi\one_{A^c})\one_{[s,T]}\in\adm$ and the induced
    controlled dynamics yield
	\begin{align*}
		Y^{s,y,\xi}(T) = Y^{s,y,\phi}(T)\one_A + Y^{s, y, \psi}(T) \one_{A^c}.
	\end{align*}
	Thus we obtain our first claim since
	\begin{align*}
		X^\phi\vee X^\psi\in \expE\lb g\lb Y^{s,y,\phi}(T)\rb\one_A + g\lb Y^{s,y,\psi}(T)\one_{A^C}\rb\mid \F_{\tau(s)}\rb =\expE\lb g\lb Y^{s,y,\xi}(T)\rb\mid \F_{\tau(s)}\rb.
	\end{align*}

	We next show that for $y\in\R^n$ the family
	\begin{align}\label{theo:dpp:eq:3}
		\lb \expE\lb v\lb s, Y^{t, y, \phi}(s)\rb\mid \F_{\tau(t)} \rb \rb_{\phi\in\adm}
	\end{align}
	is closed under maximization, too. Indeed using the same idea as above, we consider for representatives 
	\begin{align*}
		X^\phi \in \expE\lb v\lb s, Y^{t, y, \phi}(s)\rb\mid \F_{\tau(t)} \rb,\quad X^\psi \in \expE\lb v\lb s, Y^{t, y, \psi}(s)\rb\mid \F_{\tau(t)} \rb,
	\end{align*}  
	the set $A:=\{X^\phi > X^\psi\}\in\F_{\tau(t)}$ to get
	$\xi:= (\phi\one_A + \psi\one_{A^c})\one_{[t,T]}\in\adm$.
	On the one hand, it holds that
	\begin{align}\label{theo:dpp:eq:4}
		v(s, Y^{t,y,\xi}(s)) &= \esssup_{\nu\in\adm} \expE\lb g\lb Y^{s, Y^{t,y,\phi}(s),\nu}(T) \rb \mid \F_{\tau(s)}\rb\one_A + \expE\lb g\lb Y^{s, Y^{t,y,\psi}(s),\nu}(T) \rb \mid \F_{\tau(s)}\rb\one_{A^c}\notag\\
		&\le v(s, Y^{t,y,\phi}(s))\one_A + v(s, Y^{t,y,\psi}(s))\one_{A^c}.
	\end{align}
	On the other hand, since the family \eqref{theo:dpp:eq:2} is closed under pairwise maximization, we find by \parencite[Theorem A.3]{karatzasMethodsMathematicalFinance1998} (or combine \parencite[Lemma 4]{lowtherEssentialSuprema2019} with being closed under maximization for a version in $L^0$) sequences $(\nu^\phi_n)$, $(\nu^\psi_n)\in\adm$ such that 
	\begin{align*}
		\expE\lb g\lb Y^{s, Y^{t,y,\phi}(s), \nu^\phi_n}(T)\rb \mid \F_{\tau(s)}\rb \nearrow v(s, Y^{t,y,\phi}(s))
	\end{align*}
	as $n\to\infty$ and similar for $\psi$. Thus setting $\xi^n = (\nu^\phi_n\one_A + \nu^\psi_n\one_{A^c})\one_{[t,T]}\in \adm$ it holds that
	\begin{align*}
		&v(s, Y^{t,y,\phi}(s))\one_A + v(s, Y^{t,y,\psi}(s))\one_{A^c}\\
		&= \lim_{n\to\infty} \expE\lb g\lb Y^{s, Y^{t,y,\phi}(s), \nu^\phi_n}(T)\rb \mid \F_{\tau(s)}\rb\one_B + \lim_{n\to\infty}\expE\lb g\lb Y^{s, Y^{t,y,\psi}(s), \nu^\psi_n}(T)\rb \mid \F_{\tau(s)}\rb\one_{B^c}\\
		&= \lim_{n\to\infty} \expE\lb g\lb Y^{s,Y^{t,y,\xi}(s),\xi_n}\rb \mid \F_{\tau(s)}\rb
		\le v(s, Y^{t,y,\xi}(s)).
	\end{align*}
	Hence we showed equality in \eqref{theo:dpp:eq:4}. This proves the claim in the same manner as before since by construction  $A\in\F_{\tau(t)}$.
\end{proof}

\begin{theo}[Dynamic Programming Principle]\hypertarget{dpp}
	Under the conditions of Theorem \ref{theo:sol2cSDE} the value function satisfies a dynamic programming principle: for $y\in L^0(\Omega,\F_{\tau(t)},\prob; \R^n)$ and $t\le s\in[0,T]$ it holds that 
	\begin{align*}
		v(t,y) = \esssup_{\phi\in\adm} \expE\lb v\lb s, Y^{t, y, \phi}(s) \rb\mid \F_{\tau(t)}\rb.
	\end{align*}
\end{theo}

\begin{proof}
Note that since $g\ge 0$, $v\ge 0$, so all conditional expectations belong to $L^0(\Omega, \F_\tau, \prob)$ and may take the value $+\infty$. By Theorem \ref{theo:sol2cSDE} the solution $Y^{t,y,\phi}$ is unique, hence for $t\le s$ it holds on $[s,T]$ that 
\begin{align}\label{flowprop}
	Y^{t, y, \phi} = Y^{s, Y^{t, y, \phi}(s), \phi}.
\end{align}

We first prove ``$\le$''. For every $\phi\in\adm$ it is clear by the flow property \eqref{flowprop} that
	\begin{align*}
		v(s, Y^{t, y, \phi}(s)) \ge \expE\lb g\lb Y^{s,Y^{t, y, \phi}(s), \phi}(T)\rb \mid \F_{\tau(s)}\rb = \expE\lb g\lb Y^{t, y, \phi}(T)\rb \mid \F_{\tau(s)}\rb.
	\end{align*}
	By the properties of the essential supremum \eqref{def:esssup} and using the tower property of conditional expectations it follows that
	\begin{align*}
		\esssup_{\psi\in\adm} \expE\lb v\lb s, Y^{t, y, \psi}(s) \rb\mid \F_{\tau(t)}\rb
		\ge \expE\lb \expE\lb g\lb Y^{t, y, \phi}(T)\rb \mid \F_{\tau(s)}\rb\mid \F_{\tau(t)}\rb
		=\expE\lb g\lb Y^{t, y, \phi}(T)\rb \mid \F_{\tau(t)}\rb.
	\end{align*}
Moreover \eqref{def:esssup} with $Z = \esssup_{\psi\in\adm} \expE\lb v\lb s, Y^{t, y, \psi}(s) \rb\mid \F_{\tau(t)}\rb$ implies that
	\begin{align*}
		\esssup_{\psi\in\adm} \expE\lb v\lb s, Y^{t, y, \psi}(s) \rb\mid \F_{\tau(t)}\rb \ge v(t, y).
	\end{align*}

To show ``$\ge$'', we use Lemma \ref{lem:updirec} (see also its proof for references) to find sequences $(\phi_n), (\psi_{n,m})\subset\adm$ such that, as $n\to\infty$,
	\begin{align}\label{eq:N_1}
		\expE\lb v\lb s, Y^{t, y, \phi_n}(s) \rb \mid \F_{\tau(t)}\rb \nearrow \esssup_{\phi\in\adm} \expE\lb v\lb s, Y^{t, y, \phi}(s) \rb\mid \F_{\tau(t)}\rb
	\end{align}
	and for each $n$ as $m\to\infty$ that
	\begin{align}\label{eq:N_n}
		\expE\lb g\lb Y^{s, Y^{t, y, \phi_n}, \psi_{n,m}}(T) \rb \mid \F_{\tau(s)}\rb \nearrow v\lb s, Y^{t, y, \phi_n}(s) \rb.
	\end{align}
	Then $\gamma_{n,m} = \phi_n \one_{[0,s]} + \psi_{n,m}\one_{(s, T]}\in\adm$. The flow property \eqref{flowprop} now guarantees that \[ Y^{s,Y^{t, y, \phi_n}(s), \psi_{n,m}} = Y^{t, y, \gamma_{n,m}}.\] Together with monotone convergence for conditional expectations (e.g. version \parencite[Theorem 4.3.2]{dudley} for nonnegative random variables) and \eqref{def:esssup} we conclude that 
	\begin{align*}
		\esssup_{\phi\in\adm} \expE\lb v\lb s, Y^{t, y, \phi}(s) \rb\mid \F_{\tau(t)}\rb 
		&= \lim_{n\to\infty} \expE\lb v\lb s, Y^{t, y, \phi_n}(s) \rb \mid \F_{\tau(t)}\rb\\
		&=  \lim_{n, m\to\infty}  \expE\lb \expE\lb g\lb Y^{s, Y^{t,y,\phi_{n,m}}, \psi_m}(T) \rb \mid \F_{\tau(s)}\rb \mid \F_{\tau(t)}\rb\\
		&= \lim_{n, m\to\infty} \expE\lb g(Y^{t, y, \gamma_{n,m}}(T))\mid \F_{\tau(t)}\rb\\
		&\le v(t,y).
	\end{align*}
\end{proof}

\begin{rem}[Supermartingale Dynamic]\label{rem:valuefctsupermartingale}
	The previous theorem immediately implies a supermartingale dynamic of $[t,T]\ni s\mapsto v(s, Y^{t,y,\phi}(s))$ with respect to $\F_\tau$: For any $s_1 \le s_2 \in [t,T]$, the dynammic programming principle shows that in $L^0$,
	\begin{align*}
		v(s_1, Y^{t, y, \phi}(s_1)) 
		\ge \expE\lb  v(s_2, Y^{s_1, Y^{t, y, \phi}, \phi}(s_2)) \mid \F_{\tau(s_1)}\rb
		= \expE\lb  v(s_2, Y^{t, y, \phi}(s_2)) \mid \F_{\tau(s_1)}\rb,
	\end{align*}
	again by uniqueness of the solution $Y^{t, y, \phi}$.
\end{rem}

\section{Hamilton-Jacobi-Bellmann Equation}\label{sec:hjb}

From now on assume that the reward function $g$ is bounded and that the time change $\tau$ is continuously differentiable and satisfies $\tau(t)>t$ for $t\in[0,T)$. By definition \eqref{valuefct} the value function satisfies $v(t,y)\in\F_{\tau(t)}$ and is therefore a functional of $(W(s))_{s \in [0,\tau(t)]}$. To describe this functional, we want to use the concept of causal maps and their derivatives from Dupire, see Section \ref{sec:func}. A pathdependent functional is causal if it depends at time $t$ solely on the path until $t$. In that sense the value function is a causal functional not of $W$, but of the time changed Brownian motion $\Wtau$, defined by 
\begin{align}\label{Wtau}
	\begin{split}
		\Wtau\colon\Omega\times [0,T]&\to \R^d,\\
		(\omega, t)&\mapsto W(\omega,\tau(t)) -W(\omega,\tau(0)).
	\end{split}
\end{align} 
To capture the entire path of $W$ on $[0,\tau(0)]$, we supplement the functional with the independent initial segment $W^{\tau(0)}:=(W(s))_{s \in [0,\tau(0)]}$. Note that $W^{\tau(0)}$ act as parameter to the problem that is fixed from the start.
As the state dynamics~\eqref{cSDE} are Markovian in the sense that the coefficients only depend on the current state, it is in fact natural to expect that $v(t,y)$ will be independent of $\mathcal F_t$; see e.g.\ \parencite[Proposition 4]{claissePseudoMarkovPropertyControlled2016}, \parencite[Lemma 5.5.]{frizControlledRoughSDEs2024}. 
Writing $v(t,y)$ as a functional of Brownian increments $(W(s)-W(t), u\in(u,\tau(t)])$ would result in time-varying domains for such a functional. This is not directly compatible with Dupire's framework of causal functionals. For this reason we prefer the framework with $\Wtau$ and $\Wtaunull$, even though it keeps track of `too much' information. Nevertheless, one problem remains in this framework: While it is easy to construct the Brownian increments from the processes $\Wtaunull$ and $\Wtau$, Example \ref{exmp:funcdiff} part iii) below shows that such a reconstruction is not causally differentiable in time. The problem can be resolved by extending the state space of the controlled differential equation to track $(Y^{t,y,\phi}, W)$, compare the examples in Section \ref{sec:exmp}. 
Summing up we expect that
\begin{align*}
    v(t,y) = u(t,y,W(t),\Wtaunull,\Wtau)
\end{align*}
for some mapping
\begin{align*}
	u\colon[0,T]\times \R^n\times \R^d \times C([0,\tau(0)],\R^d) \times D([0,T],\R^d)&\to\R\\
    (t,y,w,\w,\z) \mapsto u(t,y,w,\w,\z).
\end{align*}
Assuming (!) that this map is regular enough we show that it is characterized by the following HJB  equation: For $\prob^{W}$-\as\ $\w\in C([0,T],\R^d)$, all $t\in[0,T)$, $y\in\R^n$, 
	\begin{align}
		\begin{split}
			\begin{cases}
				0 &= Du(t,y,\w(t),\wtaunull,\wtau) + \sup_{\varphi\in\R^m}\ \la\nabla_y u(t,y,\w(t),\wtaunull,\wtau), b(t,y,\varphi)\ra\\ 
				&\quad+ \frac12 \tr(f(y)^T\nabla^2_y u(t,y,\w(t),\wtaunull,\wtau) f(y)) + \tr(f(y)^T\nabla^2_{yw} u(t,y,\w(t),\wtaunull,\wtau))\\ 
                &\quad + \frac12 \tr(\nabla_{w}^2 u(t,y,\w(t),\wtaunull,\wtau)) + \frac12\tr(\nabla_\z^2 u(t,y,\w(t),\wtaunull,\wtau))\tau'(t),\\
                0 &= \la\nabla_y u(t,y,\w(t),\wtaunull,\wtau), f(y)\ra + \nabla_w u(t,y,\w(t),\wtaunull,\wtau),\\
				g(y) &= u(T,y,\w(t),\wtaunull,\wtau).
			\end{cases}
		\end{split}
	\end{align}
    Here, $ \nabla_y u, \nabla_y^2u,\nabla_wu,\nabla_w^2u, \nabla_{yw}^2u$ denote the Euclidean derivatives with respect to the second and third argument of $u$. Moreover $\wtau$ and $\wtau$ are constructed from $\w$ as in \eqref{Wtau}. Further  $Du$, $\nabla_\z u$ denote the \textit{causal time and space derivative} with respect to $\z\in C([0,T],\R^d)$  defined in the forthcoming Section
\ref{sec:func}. We recall from Section \ref{sec:notation} that
\begin{align*}
    \la\nabla_y u(t,y,\w(t),\wtaunull,\wtau), f(y)\ra = f(y)^T\nabla_y u(t,y,\w(t),\wtaunull,\wtau) \in\R^d.
\end{align*} 
Hence the second equation abbreviates in fact $d$ transport equations.
    
For notational simplicity we combine $(Y^{t,y,\phi},W)$ into a generic $Y$ and work with the following formulation: 
We say that a functional 
\begin{align*}
	u\colon[0,T]\times \R^n\times C([0,\tau(0)],\R^d) \times D([0,T],\R^d)\to\R
\end{align*}
satisfies the \textit{Hamilton-Jacobi-Bellman equation}, if for $\prob^{W}$-\as\ $\w\in C([0,T],\R^d)$, 
\begin{align*}
	u(\cdot, \wtaunull, \cdot)\colon [0,T]\times \R^n\times\ D([0,T],\R^d)\to \R
\end{align*}
is a \textit{causal map} such that for every $t\in [0,T)$, $y\in\R^n$, it holds that
\begin{align}\label{hjb}
	\begin{split}
		\begin{cases}
			0 &= Du(t,y,\wtaunull,\wtau) + \sup_{\varphi\in\R^m}\ \la\nabla_y u(t,y,\wtaunull,\wtau), b(t,y,\varphi)\ra\\ 
			&\qquad\qquad+ \frac12 \tr(f(y)^T\nabla^2_y u(t,y,\wtaunull,\wtau)f(y)) + \frac12 \tr(\nabla_\z^2 u(t,y,\wtaunull, \wtau))\tau'(t),\\
            0 &= \la\nabla_y u(t,y,\wtaunull,\wtau), f(y)\ra,\\
			g(y) &= u(T, y, \wtaunull, \wtau).
		\end{cases}
	\end{split}\tag{HJB}
\end{align}
Theorem \ref{theo:verification} and \ref{theo:regularvaluesolvesHJB} show in what sense the value function \eqref{valuefct} is characterized by \eqref{hjb}.

\subsection{Quadratic Variation}\label{sec:sec:qv} To fit the common It\^o dynamics of stochastic optimal control, 
we calculate $[W,\Wtau]$ using the notion of pathwise quadratic variation along a sequence of partitions from \parencite[Definition 4.1]{contPathwiseIntegrationChange2019}.\\

Define $\mathscr M([0,T], \Sym{d})$ as the space of (vector-valued) countably additive, finite measures on $\borel([0,T])$ with values in the space of symmetric matrices $\Sym{d}$.
We write $\mu_n\rightharpoonup\mu$, if $\mu_n$ converges weakly to $\mu$, i.e., for all $h\in C([0,T], \Sym{d})$, $\la h,\mu_n\ra\to\la h,\mu\ra$ as $n\to\infty$. 
Here the dual pairing is given by integration $\la h,\mu\ra =\int_0^T \la h,\dd\mu\ra$. 
Since $h$ is continuous, this reduces to integration of the real coordinates $h = (h^{i,j})_{i,j=1,\dots,d}$ to the respective signed measures $\mu = (\mu^{i,j})_{i, j=1,\dots, d}$:
\begin{align}
    \int_0^T \la h,\dd\mu\ra = \sum_{i,j=1,\dots,d} \int_0^T h^{i,j}\dd\mu^{i,j}.
\end{align}

\begin{defi}\label{defi:qv}
	A path $x\in C([0,T], \R^d)$ has finite quadratic variation along a sequence of partitions $\partition = (\partition_n)$ of $[0,T]$ with $\vert\partition\vert\to 0$ if the sequence 
	\begin{align*}
		\mu_n := \sum_{[s,t]\in\partition_n} \delta(\cdot-s)(x(t)-x(s))^{\otimes 2}
	\end{align*}
	converges weakly in $\mathscr M([0,T], \Sym{d})$ to a measure $\mu^x$ without atoms. 
	In that case, we say that $x$ has finite quadratic variation (along $\partition$) and set $[x]\colon [0,T]\to\Sym{d}$, $[x](t):= \mu^x([0,t])$ for its quadratic variation.
\end{defi}

We next consider the quadratic variation of the Brownian motion $W$ and the time-changed process $\Wtau$.
Thereby we use the following version of a classic result linking weak convergence of measures to pointwise convergence of their distribution functions.

\begin{theo}[Weak Convergence and Distribution Functions]\label{theo:weakconvergence}
	Let $\mu_n,\mu\colon\borel(\R)\to\R$ be finite signed measures such that the sequence $(\mu_n)$ is bounded in total variation norm and let $F_n, F$ denote their corresponding distributions functions. If $(F_n)$ converges pointwise to $F$ on a set $C$ that is dense in $\R$, then $\mu_n\rightharpoonup\mu$.
\end{theo}

\begin{proof}
	A proof for sequences of measures can be found in \parencite[Chapter 2, Theorem 1]{gnedenko1968limit}. Moreover the proof is easily adjusted to sequences of distribution functions of signed measures, if they are bounded in total variation norm. 
\end{proof}

\begin{lem}\label{lem:2varW} Let $\tau\colon [0,T]\to[0,T]$ be a time-change with $\tau>\id$ on $[0,T)$ and $W\colon\Omega\times[0,T]\to\R^d$ a Brownian motion. Let $\Wtau$ denote the time-change Brownian motion \eqref{Wtau}. Then there exists for every increasing sequence of partitions $\partition$ with $\vert\partition\vert\to 0$ a (not relabeled) subsequence such that for $\prob$-\as\ $\omega\in\Omega$, $(W(\omega), \Wtau(\omega))$ has finite quadratic variation and 
	\begin{align*}
		[(W(\omega),\Wtau(\omega))] = \diag(\id\one_d, (\tau-\tau(0))\one_d),
	\end{align*}
 where $\one_d$ denotes the vector of ones in $\R^d$.
\end{lem}

\begin{proof}	
    We first show the assertion for $d=1$. 
	For a sequence $\partition=(\partition_n)$ with $\vert\partition\vert\to 0$ (to be constructed) 
	define the random measures
	\begin{align*}
		\mu_n^{W} = \sum_{[u,v]\in\partition_n} \delta(\cdot-u)(W(v)-W(u))^2,\quad \mu_n^{\Wtau} = \sum_{[u,v]\in\partition_n} \delta(\cdot-u)(W(\tau(v))- W(\tau(u)))^2.
	\end{align*}
	as well as a random signed measure
	\begin{align}\label{lem:2varW:eq:1}
		\mu_n^{W, \Wtau} = \sum_{[u,v]\in\partition_n} \delta(\cdot-u)(W(v)-W(u))(W(\tau(v))- W(\tau(u))).
	\end{align}
	The assertion follows if these one-dimensional measures converge accordingly. 
	
	\underline{$\mu_n^{W}\rightharpoonup \lambda_{\mid [0,T]}$ \as}: It is well known from \parencite[Theorem 5]{levyMouvementBrownienPlan1940} that for every increasing sequence of partitions $(\partition_n)$ with $\vert \partition_n\vert\to 0$ as $n\to\infty$, it holds $\prob$-\as\ that $\mu_n^W([0,T]) \to T$.  Then clearly also \[\prob(\lim_{n\to\infty} \mu_n^W([0,t]) \to t, \forall t\in[0,T]\cap \Q)=1.\]
	Hence it follows from Theorem \ref{theo:weakconvergence} that $\prob$-\as, $\mu_n^W\rightharpoonup\lambda_{\mid[0,T]}$.
	By the same argument $\mu_n^{\Wtau}$ converges $\prob$-\as\ weakly to the measure induced by $\tau -\tau(0)$.
	
	\underline{$\mu_n^{W, \Wtau}\rightharpoonup 0$ \as}: Let $t\in[0,T]$. It holds that 
	\begin{align}\label{lem:2varW:eq:L2mun}
    \begin{split}
		&\Vert \mu_n^{W, \Wtau}([0,t])\Vert_{L^2(\Omega)}^2
		= \sum_{\substack{[u,v]\in\partition_n,\\ u\le t}} \expE (W(v)-W(u))^2(W(\tau(v))-W(\tau(u)))^2\\
		&\quad + 2 \sum_{\substack{[u,v], [s,r]\in\partition_n,\\ v\le s\le t}} \expE(W(v)-W(u))(W(\tau(v))- W(\tau(u)))(W(r)-W(s))(W(\tau(r))- W(\tau(s))).
        \end{split}
	\end{align}
	To utilize the fact that the Brownian motion $W$ has independent increments and its moment properties we distinguish several cases. If $\tau(u)\le v$ (meaning that $[u,v]$ and $[\tau(u), \tau(v)]$ intersect) we add and substract $W(v)$ to get non-overlapping subintervals $[u,v]$, $[\tau(u), v]$ and $[v,\tau(v)]$. Then $W(\tau(v))- W(v)$ is independent of $\F_v$ (or simply zero if $t=T$, $v=T$). Spelled out it holds that
	\begin{align*}
		&\expE(W(v)-W(u))^2(W(\tau(v))-W(\tau(u)))^2\\
		&= \expE(W(v)-W(u))^2\expE(W(\tau(v))-W(\tau(u)))^2\one_{v<\tau(u)}+ \Big[\expE(W(v)-W(u))^2\expE(W(\tau(v))- W(v))^2\\
		&+ 2\expE(W(v)-W(u))^2(W(v) - W(\tau(u)))\expE(W(\tau(v))- W(v)) + \expE(W(v)-W(u))^2(W(v)-W(\tau(u)))^2 \Big]\one_{\tau(u)\le v}\\
		&= (v-u)(\tau(v)-\tau(u))\one_{v<\tau(u)} + \Big[(v-u)(\tau(v)-v) + \expE(W(v)-W(u))^2(W(v)-W(\tau(u)))^2\Big]\one_{\tau(u)\le v}.
	\end{align*}
	Continuing the case $\tau(u)\le v$, adding and substracting $W(\tau(u))$ in the first factor yields that
	\begin{align*}
		&\expE(W(v)-W(u))^2(W(v)-W(\tau(u)))^2\\
		&= \expE(W(v)- W(\tau(u)))^4 + 2\expE(W(\tau(u))- W(u))\expE(W(v)- W(\tau(u)))^3\\
		&\quad + \expE(W(\tau(u))- W(u))^2\expE(W(v)- W(\tau(u)))^2\\
		&= 3(v-\tau(u))^2 + (\tau(u)- u)(v-\tau(u)).
	\end{align*}
	Thus 
	\begin{align*}
		&\sum_{\substack{[u,v]\in\partition_n,\\ u\le t}} \expE (W(v)-W(u))^2(W(\tau(v))-W(\tau(u)))^2\\
		&= \sum_{\substack{[u,v]\in\partition_n,\\ u\le t}}  (v-u)(\tau(v)-\tau(u))\one_{v<\tau(u)} + \Big[(v-u)(\tau(v)-v) + 3(v-\tau(u))^2 + (\tau(u)- u)(v-\tau(u))\Big]\one_{\tau(u)\le v}\\
		&\lesssim \sum_{\substack{[u,v]\in\partition_n}} (v-u)(\tau(v)-\tau(u)) \le \vert\partition_n\vert T\to 0
	\end{align*}
	as $n\to\infty$. 
    
    The other summand in \eqref{lem:2varW:eq:L2mun} decomposes similar manner: Let $[u,v], [s,r]\in\partition_n$ with $v\le s\le t$, then
	\begin{align*}
		&\expE(W(v)-W(u))(W(\tau(v))- W(\tau(u)))(W(r)-W(s))(W(\tau(r))- W(\tau(s)))\\
		& = \expE(W(v)-W(u))(W(\tau(v))- W(\tau(u)))\Big[ (W(r)-W(s))(W(\tau(r))- W(\tau(s)))\one_{r<\tau(s)}\\
		&+ \Big[(W(r)- W(s))(W(\tau(r))- W(r)) + (W(r)- W(\tau(s)))^2 + (W(\tau(s))- W(s))(W(r)- W(\tau(s)))\Big]\one_{\tau(s)\le r}\Big]\\
		&=\expE(W(v)-W(u))(W(\tau(v))- W(\tau(u)))(r-\tau(s)),
	\end{align*}
	where we used in the last line that in the case $r<\tau(s)$, $W(\tau(r))- W(\tau(s))$ is independent of $\F_{\tau(s)}$ and by monotonicity of $\tau$, $\F_v\subset \F_{\tau(v)}\subset\F_{\tau(s)}$ and similarly for $r\ge\tau(s)$. Moreover the same arguments show that 
	\begin{align*}
		\expE(W(v)-W(u))(W(\tau(v))- W(\tau(u))) =0.
	\end{align*}
	Hence $\Vert \mu_n^{W, \Wtau}([0,t])\Vert_{L^2(\Omega)}\to 0$ as $n\to\infty$. Take now any increasing sequence of partitions $\partition$ with $\vert\partition\vert\to 0$. Then we may extract for each $t\in[0,T]$ a $\prob$-\as\ convergent subsequence. In fact restricting to $\mathbb Q\cap [0,T]$,
	we find by a diagonal argument a (not relabeled) subsequence and a $\prob$-negligible set $ N$, such that for all $\omega\in\Omega\setminus N$, every $t\in\mathbb Q\cap [0,T]$, $\mu_n^{W(\omega), \Wtau(\omega)}([0,t])\to 0$ as $n\to\infty$. 
	To conclude that this implies $\mu_n^{W(\omega), \Wtau(\omega)}\rightharpoonup 0$ we apply Theorem \ref{theo:weakconvergence}: First note that for each $\omega\in\Omega$, 
	\[F_n^{W(\omega), \Wtau(\omega)} := \sum_{[u,v]\in\partition_n} (W(\omega,v)-W(\omega,u))(W(\omega,\tau(v))- W(\omega,\tau(u)))\one_{[0,v)}\]
    	is clearly right-continuous and has bounded variation (since it takes only finitely many values). Therefore $\mu_n^{W(\omega), \Wtau(\omega)}$ is a finite $\sigma$-additive measure on $\borel([0,T])$ and \parencite[Theorem 4.17]{bongiornoRieszRepresentationTheorem2001} shows that it has total variation $\vert \mu^{W(\omega), \Wtau(\omega)}_n\vert ([0,T])= \mathrm{Var}(F_n^{W(\omega), \Wtau(\omega)}, [0,T])$. It remains to show that the sequence is bounded in total variation norm. Indeed by the Cauchy-Schwarz inequality for sums it holds that
	\begin{align*}
		\mathrm{Var}(F^{W(\omega), \Wtau(\omega)}_n, [0,T]) &= \sup_{\partition\text{ of }[0,T]} \sum_{[s,t]\in\partition} \vert F_n^{W(\omega), W(\omega)\circ\tau}(t) - F_n^{W(\omega), W(\omega)\circ\tau}(s)\vert\\
		&= \sum_{[u,v]\in\partition_n} \vert W(\omega,v)-W(\omega,u)\vert \cdot \vert W(\omega,\tau(v))- W(\omega,\tau(u))\vert\\
		&\le \mu_n^{W(\omega)}([0,T])^{1/2}\mu_n^{\Wtau(\omega)}([0,T])^{1/2},
	\end{align*}
	and the last term is bounded, since $(\mu_n^W)$, $(\mu_n^{\Wtau})$ are $\prob$-\as\ weakly convergent (change $ N$ if necessary).
    
    Finally for $d>1$ let $i,j=1,\dots, d$ with  $i\not= j$. Define $\mu_n^{W^i,W^j}$ by adjusting \eqref{lem:2varW:eq:1} in the obvious way. Since the components $W^i$ and $W^j$ are independent it follows in the same manner as before that for every $t\in[0,T]$ that \[\Vert\mu_n^{W^i,W^j}([0,t])\Vert_{L^2(\Omega)}, \Vert\mu_n^{W^i, (W^j)^{\circ\tau}}([0,t])\Vert_{L^2(\Omega)}\to 0.\]
    as $n\to\infty$. 
\end{proof}

\begin{cor}\label{cor:2varXW} Let $\alpha\in(1/2,1/3)$. Assume the setting of Lemma \ref{lem:2varW} and suppose that $Y$ solves \eqref{cSDE} and additionally that the time change $\tau$ is continuously differentiable. Then there exists a sequence of partitions $\partition$ with $\vert\partition\vert\to 0$ and a negligible set $ N\in\F$, such that for all $\omega\in\Omega\setminus N$, $(W(\omega), \Wtau(\omega))$ has finite quadratic variation along $\partition$ and $\vert W(\omega)\vert_\alpha <\infty$. In particular for every $\omega\in\Omega\setminus N$ it then also follows that $(Y(\omega), \Wtau(\omega))$ has finite quadratic variation with
	\begin{align*}
		[(Y(\omega), \Wtau(\omega))] = \mathrm{diag}\lb\int_0^\cdot f(Y(\omega))^{\otimes 2}\dd\lambda, (\tau-\tau(0))\one_d\rb.
	\end{align*}
\end{cor}

\begin{proof}
	Such a sequence $\partition =(\partition_n)$ and negligible set $N$ exists by the above lemma and the Kolmogorov criterion, see for instance \parencite[Theorem 3.1]{RoughBook}. We consider
	\begin{align*}
		\mu_n^{(Y, \Wtau)} := \sum_{[s,t]\in\partition} \delta(\cdot - s)\begin{pmatrix} Y(t)-Y(s)\\ \Wtau(t) - \Wtau(s)\end{pmatrix}^{\otimes 2}.
	\end{align*}
	Further let $h = (h^{i,j})_{i,j=1,\dots,d+n}\in C([0,T], \Sym{n+d})$. Then
	\begin{align}
		&\int_0^T \la h, \dd\mu_n^{(Y, \Wtau)}\ra 
		= \sum_{[s,t]\in\partition_n} \la h(s), \begin{pmatrix} Y(t)-Y(s)\\ \Wtau(t) - \Wtau(s) \end{pmatrix}^{\otimes 2}\ra\notag\\
		&= \sum_{[s,t]\in\partition_n} \la h(s), \begin{pmatrix} R^{Y,W}(s,t)\\0\end{pmatrix}^{\otimes 2} + 2\sym{\begin{pmatrix} R^{Y,W}(s,t)\\ 0\end{pmatrix} \otimes \begin{pmatrix}f(Y(s))(W(t)- W(s)\\ \Wtau(t)-\Wtau(s)\end{pmatrix}} \ra\label{cor:2varXW:eq:1}\\
		&\quad + \sum_{[s,t]\in\partition_n} \la h(s), \begin{pmatrix} f(Y(s))(W(t)-W(s))\\ \Wtau(t) - \Wtau(s)\end{pmatrix}^{\otimes 2} \ra.\label{cor:2varXW:eq:2}
	\end{align}
	Since $Y$ is a solution to~\eqref{cSDE}, $(Y, f(Y))$ is $W$-controlled by Theorem~\ref{theo:sol2cSDE}) and thus
	\begin{align*}
		&\left\vert \la h(s), \begin{pmatrix} R^{Y,W}(s,t)\\0\end{pmatrix}^{\otimes 2} + 2\sym{\begin{pmatrix} R^{Y,W}(s,t)\\ 0\end{pmatrix} \otimes \begin{pmatrix}f(Y(s))(W(t)- W(s)\\ \Wtau(t)-\Wtau(s)\end{pmatrix}} \ra \right\vert\\
		&\lesssim_{\vert f\vert_\infty, \vert f\vert_\infty} \vert R^{Y,W}(s,t)\vert^2 + \vert R^{Y,W}(s,t)\vert(\vert W(t)-W(s)\vert+\vert \Wtau(t)-\Wtau(s)\vert)\\
        &\lesssim_{\vert \W\vert_{2\alpha}, \vert f\vert_{C^2_b}, \vert b\vert_\infty, \vert W\vert_\alpha}  \vert t-s\vert^{4\alpha} + \vert t-s\vert^{3\alpha},
	\end{align*}
    since $\vert\Wtau\vert_\alpha\lesssim_{T,\vert\tau'\vert_\infty}\vert W\vert_\alpha$.
	Now $3\alpha >1$ implies that \eqref{cor:2varXW:eq:1} converges to zero as $n\to\infty$. The other sum \eqref{cor:2varXW:eq:2} gives
	\begin{align*}
		&\sum_{[s,t]\in\partition_n} \la h(s), \begin{pmatrix} f(Y(s))(W(t)-W(s))\\ \Wtau(t) - \Wtau(s)\end{pmatrix}^{\otimes 2} \ra=\int_0^T \la \tilde h(s), \dd\mu_n^{W, \Wtau}(s)\ra,
	\end{align*}
    for $\tilde h\colon [0,T]\to \Sym{2d}$ defined by 
    \begin{align*}
        \tilde h(s) =\begin{pmatrix} f(Y(s))^T&0\\ 0 &\diag(\one_d)\end{pmatrix}h(s)\begin{pmatrix} f(Y(s))&0\\ 0 &\diag(\one_d)\end{pmatrix}.
    \end{align*}
	Since $s\mapsto \tilde h(s)$ is continuous, Lemma \ref{lem:2varW} implies that 
    \begin{align*}
        \int_0^T \la \tilde h, \dd\mu_n^{W, \Wtau}\ra\to \int_0^T \la \tilde h, \dd\mu^{W, \Wtau}\ra 
        = \int_0^T \la \tilde h,  \mathrm{diag}(\one_d, \tau'\one_d)\ra \dd\lambda 
        = \int_0^T \la h, \mathrm{diag}(f(Y(\omega))^{\otimes 2}, \tau'\one_d)\ra \dd\lambda.
    \end{align*}
    
\end{proof}

\subsection{Functional Calculus}\label{sec:func}
We briefly recall definitions and results from \parencite{bielertRoughFunctionalIto2024}.
\subsubsection{Causal Derivatives}
Derivatives for pathdependent causal functionals are introduced in \parencite{dupireFunctionalItoCalculus2009}. The following formulations are due to \parencite{oberhauserExtensionFunctionalIto2012}. A functional $F\colon\ [0,T]\times D([0,T], \R^d) \to \R$ is called \textit{causal} if for any $X \in D([0,T], \R^d)$ and $t \in [0,T]$ we have
$F(t, X) = F(t, X_t)$. 

\begin{defi}[Causal Space Derivative]\label{defi:D}
	If for all $(t,X)\in [0,T]\times D([0,T], \R^d)$ the map
	\begin{align*}
		\R^d\ni h\mapsto F(t, X_t + h \one_{[t,T]})
	\end{align*}
	is continuously differentiable at $h = 0$ we say that $F$ has a causal space derivative. We denote it by $\nabla F(t,X) = (\partial_1 F(t,X), \dots, \partial_n F(t,X))$. Similarly, we define for $n\in\N$ the $n$th causal space derivative and denote it by $\nabla^n F$.
\end{defi}
\begin{defi}[Causal Time Derivative]\label{defi:nabla}
	If for all $(t,X)\in [0,T]\times D([0,T], \R^d)$ the map 
	\begin{align*}
		[0,\infty)\ni h \mapsto F(t + h, X_t)
	\end{align*}
	is continuous and right-differentiable at $h=0$ we denote this derivative by $D F(t, X)$. If additionally $t\mapsto D F(t, X)$ is Riemann integrable, then we say that $F$ has a causal time derivative.
\end{defi}
Note that causal space and time derivatives are again causal functionals.
For $n\in\N$ we write $F\in\mathbb{C}^{1,n}_b$, if $F$ has a causal time derivative and $n$ causal space derivatives such that $F$, $DF$ and for $k=1, \dots, n$, $\nabla^k F$ are continuous in $[0,T]\times D([0,T], \R^d)$ and uniformly bounded in $(t,X)$. Further we define a norm bei $\vert F\vert_{\mathbb{C}^{1,n}_b} := \vert DF\vert_\infty+\sum_{k=0}^n \vert\nabla^kF\vert_\infty$.

\begin{exmp}\label{exmp:funcdiff} Let $g\in C(\R^d)$, $f\in C^1(\R^d)$. By \parencite[Example 5.2.5, 5.2.6]{BallyVlad2016Sibp},
	\begin{enumerate}[i)]
		\item $F(t,X) = f(X(t))$, then $\nabla F(t, X) = \nabla f(X(t))$ and $DF(t, X)=0$.
		\item $F(t,X)=\int_0^t g(X(s))\dd s$, then $\nabla F(t,X)=0$ and $DF(t,X)=g(X(t))$.
	\end{enumerate}
The discussion in the beginning of Section \ref{sec:hjb} suggest that the value $v(t,y)$ could depend on $f(W(\tau(t))- W(t))$ or $\int_t^{\tau(t)} f(W(\tau(t))- W(u))\tau'\dd u$. Let $\tau$ be a time change such that $\tau, \tau^{-1}\in C^1$. 
\begin{enumerate}[i)]\setcounter{enumi}{2}
	\item We look for a family of causal maps $F$ that realize $F(t,\Wtau,\Wtaunull)=f(W(\tau(t))-W(t))$: Since
	\begin{align*}
		W(\tau(t))-W(t) 
		=\Wtau(t)-\Wtau(\tau^{-1}(t\vee\tau(0)))+\Wtaunull(\tau(0))-\Wtaunull(t\wedge \tau(0)).
	\end{align*}
	we set for $\w\in C([0,\tau(0)],\R^d)$,
	\begin{align*}
		F(t,X,\w)=X(t)-X(\tau^{-1}(t\vee\tau(0)))+\w(\tau(0))-\w(t\wedge \tau(0)).
	\end{align*}
	Then $F(\cdot,\w)$ is a causal map with $\nabla F(t,\Wtau, \Wtaunull)= \nabla f(W(\tau(t))-W(t))$. But $F(\cdot, \w)$ has in general \textbf{no} causal time derivative: Let $h>0$. If $\tau^{-1}(t+h)\le t$, then
	\begin{align*}
		&X_t(t+h)-X_t(\tau^{-1}((t+h)\vee\tau(0))) +\w(\tau(0))-\w((t+h)\wedge \tau(0))\\
		&= X(t)-X(\tau^{-1}((t+h)\vee\tau(0))) +\w(\tau(0))-\w(t+h)\wedge \tau(0)).
	\end{align*}
	Hence for $DF(t,X, \w)$ to exist, $X$ would need to be right-differentiable in $\tau^{-1}(t)$. In particular this is false for typical Brownian sample paths $X=W(\omega)$.
	\item For the functional $F(t, X)= \int_{\tau^{-1}(t)}^t f(X(t)-X(u))\tau'(u) \dd u$, Leibniz' integral rule gives
	\begin{align*}
		\nabla F(t,X) &= \lim_{\vert h\vert\to 0}\ \frac1h \lb \int_{\tau^{-1}(t)}^t f(X(t) +h-X(u))\tau'(u)\dd u - \int_{\tau^{-1}(t)}^t f(X(t)-X(u))\tau'(u)\dd u \rb\\
		&= \int_{\tau^{-1}(t)}^t \nabla f(X(t)-X(u))\tau'(u)\dd u
	\end{align*}
	and 
	\begin{align*}
		 DF(t,X) &= \lim_{h\searrow 0}\ \frac1h \lb \int_{\tau^{-1}(t+h)}^{t+h}f(X_t(t+h)-X_t(u))\tau'(u)\dd u - \int_{\tau^{-1}(t)}^t f(X_t(t)-X_t(u))\tau'(u)\dd u \rb\\
		 &= f(0)\tau'(t) - \lim_{h\searrow 0} \frac1h \int_{\tau^{-1}(t)}^{\tau^{-1}(t+h)} f(X(t)-X(u))\tau'(u)\dd u\\
		 &= f(0)\tau'(t) - f(X(t)-X(\tau^{-1}(t)))\tau'(\tau^{-1}(t))(\tau^{-1})'(t).
	\end{align*}
\end{enumerate}
\end{exmp}

\begin{rem}
	The continuity and boundedness assumptions on $F$ and its causal derivatives are very strong. For instance, for a continuous semimartingale $X$ there is no causal functional $F\in\c12$, such that $F(\id,X)$ is indistinguishable from the stochastic quadratic variation or from the It\^o integral $\int_0^\cdot X\dd X$ (since these objects are not continuous in the topology of uniform convergence). Of course it is desirable (but outside the scope of this work) to allow for a richer class of causal functionals as considered in \parencite[Definition 10, Theorem 33, Corollary 39]{oberhauserExtensionFunctionalIto2012}.  
    Other working directions are to change the state space of $F$: E.g.\ add an explicit dependence on $\la X\ra$, see \parencite[Part II, Chapter 6.3]{BallyVlad2016Sibp}, or let $F$ depend on a full (geometric) rough path $\Xb\colon[0,T]\to G^{N}(\R^d)$, see \parencite{cuchieroFunctionalItoformulaTaylor2025} for an appropriate nonsymmetric adjustment of the concepts of causal derivatives. We point out that in the case of one-dimensional Brownian motion the current set up is enough. Indeed by Example \ref{exmp:lift} its It\^o lift is simply $(W, \frac12\delta W - \id )$. Hence a \textit{reduced rough path}, that is its second level is purely symmetric and consists solely of increments $\delta W$. Whereas a multidimensional Brownian motion leads to a second level $\W$ with a non zero antisymmetric part. That part represents complex pathdependencies that are ``not seen'' by the the symmetric causal space derivative. 
\end{rem}

In the Definition \ref{defi:nabla} of the causal space derivative, $F$ needs to be defined for all c\`adl\`ag paths $X$ (otherwise $F$ is not defined for the perturbed path $X_t +h\one_{[t,T]}$). By contrast, in our optimization problem only continuous (even Brownian) paths matter. Hence the question arises in what sense the Brownian paths determine a solution to \eqref{hjb}. It is known from \parencite[Theorem 5.4.1, Theorem 5.4.2]{BallyVlad2016Sibp}, that whenever two causal functionals $F^1, F^2\in\c12$ coincide on $[0,T]\times C([0,T],\R^d)$, also their causal space derivatives $\nabla^k F^1=\nabla^k F^2$ coincide there. In fact, even a restriction to Brownian paths as in \eqref{hjb}, still characterizes the `causal heat operator' up to indistinguishability, if combined with regularity assumptions. This is the point of following result which is essentially due to \parencite{oberhauserExtensionFunctionalIto2012}:

\begin{cor}\label{cor:unique}
Assume that $\tau$ is a continuously differentiable time change and $u\colon [0,T]\times \R^n\times C([0,\tau(0)],\R^d)\times D([0,T],\R^d)\to \R$ such that
\begin{enumerate}[i)]
    \item for all $y\in\R^n$, $\w\in C([0,\tau(0)],\R^d)$, $u(\cdot,y,\w,\cdot)\in\c12$,
    \item $(y,\w)\mapsto u(\cdot,y,\w,\cdot)\in\c12$ is continuous,
    \item for all $t\in[0,T]$, $\w\in C([0,\tau(0)],\R^d)$, $\z\in D([0,T],\R^d)$, $y\mapsto u(t,y,\w,\z)$ is two times continuously differentiable.
\end{enumerate}
Let $\tilde u$ be another map satisfying $i)-iii)$ and such that for $\prob^{\Wtaunull}$-\as\ every $\w\in C([0,\tau(0)],\R^d)$ and all $y\in\R^n$, $u(\cdot,y,\w,\cdot)$ is indistinguishable from $\tilde u(\cdot,y,\w,\cdot)$, i.e.\ for $\prob^{\Wtau}$-\as\ $\z\in D([0,T],\R^d)$ it holds $u(\id,y,\w,z)=u(\id,y,\w,z)$. Then $u$ solves \eqref{hjb} if and only if $\tilde u$ does.
\end{cor}

Let us emphasize that, without the regularity assumptions on $u,\tilde{u}$  the above result is false, see \parencite[Remark 8]{oberhauserExtensionFunctionalIto2012}.

\begin{proof}
    Let $y\in\R^n$ and $N_1$ be a $\prob^{\Wtaunull}$-neglible set, such that for all $\w\in C([0,\tau(0)],\R^d)\setminus N_1$, $u(\cdot,y,\w,\cdot)$ and $\tilde u(\cdot,y,\w,\cdot)$ are indistinguishable. By the regularity assumption ii), we may assume that the set where the processes coincide is independent of $\w$. Namely we find a $\prob^{\Wtau}$-neglible set $N_2$ such that for all $\w\in C([0,\tau(0)],\R^d)\setminus N_1$, all $\z\in D([0,T],\R^d)\setminus N_2$, it holds that
    \begin{align*}
        u(\id,y,\w,\z)=\tilde u(\id,y,\w,\z).
    \end{align*}
    Clearly outside of $N_1,N_2$ and for every $t\in[0,T]$, $k=1,2$,
    \begin{align}\label{cor:unique:eq:1}
        \nabla^k_y u(t,y,\w,\z)=\nabla^k_y \tilde u(t,y,\w,\z).
    \end{align}
     Concerning the causal derivatives, it follows from \parencite[Proposition 29]{oberhauserExtensionFunctionalIto2012} that for $\w\in C([0,\tau(0)],\R^d)\setminus N_1$, the processes $Du(\cdot,y,\w,\cdot)+\frac12\nabla_\z u(\cdot,y,\w,\cdot)\tau'$ is indistinguishable from the respective expression with $\tilde u$. Since regularity property $ii)$ holds in $\c12$, this yields again that outside of $N_1,N_2$ (adjust $N_2$ if necessary),
    \begin{align}\label{cor:unique:eq:2}
        Du(\id,y,\w,\z) +\frac12 \nabla_\z^2 u(\id,y,\w,\z)\tau'=D\tilde u(\id,y,\w,\z) +\frac12 \nabla_\z^2 \tilde u(\id,y,\w,\z)\tau'.
    \end{align}
    Now since $N:=(\Wtaunull)^{1}(N_1)\cup(\Wtau)^{-1}(N_2)$ is a $\prob$-neglible set, \eqref{cor:unique:eq:1} and \eqref{cor:unique:eq:2} imply the claim.
\end{proof}

\subsubsection{Rough Funtional It\^o Formula}
We apply the following rough functional It\^o formula \parencite[Theorem 3.1, Example 3.3]{bielertRoughFunctionalIto2024} in the regime of Brownian paths and with It\^o correction.

\begin{theo}[\hypertarget{theo:roughfuncInt}{Rough Functional It\^{o} Formula}]
	Let $\alpha\in((\sqrt{13}-1)/6, 1/2)$. Let $X$ be $\alpha$-H\"older continuous and have finite quadratic variation along some partition $\partition$ with $\vert\partition\vert\to 0$. Let further $F\in\mathbb{C}^{1,4}_b$ be a causal map such that for $k=1,2,3$, $\nabla^kF\in C^{1,1}_b$ and for $k=0,1,2$, $\nabla^kF$ and $D\nabla^kF$ are Lipschitz continuous for fixed times with bounded Lipschitz constants.  
	Then 
	\begin{align*}
		F(T, X) = F(0, X) + \int_0^T DF(\id, X) \dd \lambda + \int_0^T \nabla F(\id, X) \dd \Xb + \frac12 \int_0^T \la \nabla^2 F(\id, X), \dd[X]\ra
	\end{align*}
	where for $\X\colon \Delta_T\to \Sym{d}$, $\X(s,t):= \frac12((X(t)-X(s))^{\otimes 2} - ([X](t)-[X](s)))$ and
	\begin{align}\label{theo:roughfuncInt:eq}
		\int_0^T \nabla F(\id, X) \dd \Xb := \lim_{\vert \partition\vert \to 0} \sum_{[s,t]\in\partition} \langle\nabla F(s, X),(X(t)-X(s))\rangle +  \langle\nabla^2 F(s, X),\X(s,t)\rangle,
	\end{align}
	is a well defined limit along arbitrary partitions with mesh tending to zero.
\end{theo}

\begin{rem}\label{rem:nablaFcontrolled}
	\parencite[Lemma 3.2]{bielertRoughFunctionalIto2024} shows that  $(\nabla F(\id, X), \nabla^2 F(\id, X))$ is an $(\alpha, \alpha^2)$ H\"older $X$-controlled rough path. In particular inspecting the proof of \parencite[Corollary 2.3]{bielertRoughFunctionalIto2024} shows that the H\"older constant $\vert\nabla^2 F(\id, X)\vert_\alpha$ depends on $\vert\nabla^3 F\vert_\infty$, $\sup_{r\in[s,t]} \{\mathrm{Lip}(\nabla^2F(r, \cdot)), \mathrm{Lip}(D \nabla^2F(r, \cdot))\}$ as well as $\vert X\vert_\alpha$. Similarly, $\vert R^{\nabla F(X),X}\vert_{\alpha+\alpha^2}$ depends on $\vert\nabla F\vert_\infty$, $\vert D\nabla F\vert_\infty$, $\vert\nabla^3 F\vert_\infty$, $\sup_{r\in[s,t]} \{\mathrm{Lip}(\nabla F(r, \cdot)), \mathrm{Lip}(D \nabla F(r, \cdot))\}$ as well as $\vert X\vert_\alpha$.
\end{rem}

The next theorem is an application of the \hyperlink{theo:roughfuncInt}{rough functional It\^{o} formula} to the path $(Y, \Wtau)$, where $Y$ solves \eqref{cSDE}. Since its proof relies on standard rough path techniques combined with Corollary \ref{cor:2varXW} it is deferred to the Appendix \ref{sec:appendix}. As already pointed out in \parencites[Remark 6.2.2.]{BallyVlad2016Sibp}, the statement holds for any suitable causal map $F$ on the same set with full measure.

\begin{theo}\label{theo:funcIto4SolXW}
	Assume the setting of Corollary \ref{cor:2varXW} and let $N$ be the $\prob$-negligible set obtained there. Then we may apply for $\omega\in\Omega\setminus N$ the \hyperlink{theo:roughfuncInt}{rough functional It\^{o} formula} with $(Y(\omega), \Wtau(\omega))$ and with all causal maps $F$ and $\alpha$ as described therein. On this set we obtain that
	\begin{align*}
		F(T, Y, \Wtau) - F(0, Y, \Wtau) &= \int_0^T DF(\id, Y, \Wtau) + H(\id,\phi,Y,\Wtau) \dd\lambda\\
		&\quad + \int_0^T \la \nabla_y F(\id, Y, \Wtau), f(Y)\dd\Wb\ra + \int_0^T \la\nabla_z F(\id,Y,\Wtau),\dd\Wtau\ra,
	\end{align*}
	where $H = H^{F, b,f,\tau}\colon [0,T]\times \R^d\times \R^m\times C([0,T],\R^d) \to \R$ is defined by
	\begin{align*}
		H(t,y,\varphi,\z):= \la \nabla_y F(t, y, \z), b(t, y,\varphi)\ra + \frac12[ \tr(f(y)^T\nabla_y^2 F(t, y, \z)f(y)) + \tr(\nabla_\z^2 F(t, y, \z))\tau'(t)]
	\end{align*}
     and $\nabla_y F$ denotes the gradient of $F(t, \cdot, \z)$ with respect to $y\in\R^n$ and $DF,\nabla_\z F$ the causal derivatives of $F(\cdot,y,\cdot)$ with respect to $(t,\z)\in [0,T]\times C([0,T],\R^d)$.
    The rough integral is the well-defined since
    \begin{align}
    \begin{split}\label{theo:funcIto4SolXW:Z}
        Z &:= \la\nabla_y F(\id, Y, \Wtau), f(Y)\ra,\\ Z' &:= \la\nabla_y^2 F(\id, Y, \Wtau), f(Y)^{\otimes 2}\ra + \la\nabla_y F(\id, Y,\Wtau), \nabla_y f(Y)\cdot f(Y)\ra
    \end{split}
    \end{align}
    defines an $\alpha$-H\"older $W$-controlled rough path by a Leibniz rule
    \begin{align}
        \begin{split}\label{theo:funcIto4SolXW:eq1}
    \vert Z'\vert_\alpha \lesssim_{\vert\nabla_y^2 F\vert_\infty, \vert f\vert_\infty^2, \vert\nabla_y F\vert_\infty, \vert\nabla_yf\vert_\infty} &\vert f(Y)\vert_\alpha + \vert \nabla_y^2 F(\id,Y,\Wtau)\vert_\alpha\\
    &\quad+ \vert \nabla_yf(Y)\cdot f(Y)\vert_\alpha + \vert\nabla_yF(\id,Y, \Wtau)\vert_\alpha,
        \end{split}\\
        \begin{split}\label{theo:funcIto4SolXW:eq2}
    \vert R^{Z, W}\vert_{\alpha + \alpha^2} 
    \lesssim_{T, \alpha, \vert f\vert_\infty, \vert\nabla F\vert_\infty, \vert \nabla^2 F\vert_\infty } &\vert R^{\nabla F(Y,\Wtau), (Y,\Wtau)}\vert_{\alpha + \alpha^2} + \vert Y\vert_\alpha \vert f(Y)\vert_\alpha\\
    &\quad + \vert R^{Y,W}\vert_{2\alpha} + \vert R^{f(Y), W}\vert_{2\alpha}
        \end{split}
    \end{align}
	and the last integral is a well-defined It\^o integral. 
\end{theo}

\subsection{Verification Theorem}
Any solution $u$ to the deterministic \eqref{hjb} equation relates to the value function $v\colon [0,T]\times \mathcal \R^n \to L^0(\prob)$ of~\eqref{valuefct} in the following way. 

\begin{theo}[Verification Theorem]\label{theo:verification} 
	Let
	\begin{align*}
		u\colon[0,T]\times \R^n\times C([0,\tau(0)],\R^d)\times D([0,T],\R^d)\to\R
	\end{align*}
	be Borel measurable. Assume that for $\prob^W$-\as\ $\w\in C[0,T],\R^d)$,  $u(\cdot, \wtaunull,\cdot)\in \mathbb{C}_b^{1,4}$ and that the assumptions of the \hyperlink{theo:roughfuncInt}{rough functional It\^{o} formula} are satisfied.
	\begin{enumerate}[i)]
		\item If $u$ satisfies \eqref{hjb} then for all $t\in [0,T)$, $y\in\R^n$, every representative of $v(t, y)$ is $\prob$-\as\ dominated by  $u(t,y,\Wtaunull,\Wtau)$.
		\item If further there exists for every $t\in[0,T]$, $y\in\R^n$, a control $\phi^*\in\adm$ such that $\lambda\vert_{[t, T]} \otimes \prob$-\as it holds \begin{align*}
			&\la\nabla_y u(\id,Y^{t,y,\phi^*},\Wtaunull,\Wtau), b(\id,Y^{t,y,\phi^*}, \phi^*)\ra\\
			&= \sup_{\varphi\in\R^m}\ \la\nabla_y u(\id,Y^{t,y,\phi^*},\Wtaunull,\Wtau), b(\id,Y^{t,y,\phi^*},\varphi)\ra,
		\end{align*}
		   then for every $t\in[0,T]$, $y\in \R^n$, $u(t,y,\Wtaunull,\Wtau)$ is a representative of $v(t,y)$.
	\end{enumerate}
\end{theo}

\begin{proof}
	\begin{enumerate}[i)]
		\item For $\phi\in\adm$ we denote by $Y^\phi$ the solution to \eqref{cSDE} starting at $t\in[0,T)$ with initial condition $y\in\R^n$. The conditions on $u$ allow to apply the \hyperlink{theo:roughfuncInt}{rough functional It\^{o} formula} for $\prob$-\as\ $\omega\in\Omega$ with $u(\cdot,\Wtaunull(\omega),\cdot)$ and $(Y^\phi(\omega),\Wtau(\omega))$. Noting that 
		\begin{align}\label{theo:verification:eq:1}
			\int_t^T \la\nabla_\z u(\id,Y^\phi,\Wtaunull,\Wtau),\dd \Wtau\ra
		\end{align}
		is a well-defined It\^{o} integral, too, 
        (since $s\mapsto\partial_w u(s,Y^\phi(s),\Wtaunull,\Wtau)$ is adapted to $\F_\tau$
        and has continuous sample paths), we may apply Theorem $\ref{theo:funcIto4SolXW}$. Writing for $H^{u(\cdot,\Wtaunull(\omega),\cdot),b,f,\tau}(t,y,\varphi,\z)$ simply $H(t,y,\varphi,\Wtaunull(\omega),\z)$, it holds $\prob$-\as\ that
		\begin{align*}
			&u(T,Y^\phi(T),\Wtaunull,\Wtau) - u(t,Y^\phi(t),\Wtaunull,\Wtau)\\ 
			&= \int_t^T Du(\id,Y^\phi,\Wtaunull,\Wtau) + H(\id,Y^\phi,\phi,\Wtaunull,\Wtau) \dd \lambda\\
			&\quad + \int_t^T \la \nabla_y u(\id,Y^\phi,\Wtaunull,\Wtau), f(Y^\phi)\ra\dd\Wb + \int_t^T \la\nabla_\z u(\id,Y^\phi\id,\Wtaunull,\Wtau),\dd \Wtau\ra.
		\end{align*}
		By the initial condition for $Y^\phi$, $u(t,Y^\phi(t),\Wtaunull,\Wtau) = u(t,y,\Wtaunull,\Wtau)$ and by the terminal condition for $u$, $u(T,Y^\phi(T),\Wtaunull,\Wtau) = g(Y^\phi(T))$. The second equation in \eqref{hjb} implies for $\prob$-\as\ $\w\in C([0,T],\R^d)$ that for on $[t,T)$, $y\in\R^n$ also
		\begin{align*}
			0 &\equiv \nabla_y\la\nabla_y u(\id,\cdot,\Wtaunull,\wtau), f\ra (y)\\
			&= \nabla_y^2 u(\id,y,\Wtaunull,\wtau)f(y) + \nabla_yf(y)\nabla_yu(\id,y,\Wtaunull, \wtau).
		\end{align*}
		Consequently the rough integral equals zero $\prob$-\as. Taking the conditional expectation with respect to $\F_{\tau(t)}$, the It\^o integral vanishes, too.
		Moreover by \eqref{hjb}, it holds $\prob$-\as\ for every $s\in[t,T)$ that
		\begin{align*}
			&Du(s,Y^\phi(s),\Wtaunull,\Wtau) + H(s,Y^\phi(s),\phi,\Wtaunull,\Wtau)\\
			&= Du(s,Y^\phi(s),\Wtaunull,\Wtau) + \la b(s, Y^\phi(s), \phi), \nabla_y u(s,Y^\phi(s),\Wtaunull,\Wtau)\ra  \\
			&\quad + \frac12 \tr(f(Y^\phi(s))^T\nabla^2_y u(s,Y^\phi(s),\Wtaunull,\Wtau)f(Y^\phi(s)))  + \frac12 \tr(\nabla_\z^2 u(s,Y^\phi,\Wtaunull,\Wtau))\tau'(t) \le 0.
		\end{align*}
		So it follows with respect to \as-ordering in $L^0(\prob)$ that
		\begin{align*}
			\expE\lb g(Y^\phi(T)) \mid \F_{\tau(t)}\rb - u(t,y,\Wtaunull,\Wtau) \le 0.
		\end{align*}
		Recalling the characterizing properties \eqref{def:esssup} of the essential supremum, this yields the claim.
		\item Using $\phi=\phi^*$ in the calculations in part i) we obtain equality instead of the last inequality. Hence \eqref{def:esssup} implies that 
		\begin{align*}
			v(t,y) \ge u(t,y,\Wtaunull,\Wtau)
		\end{align*}
		in $L^0(\prob)$. Consequently, it holds for every $t\in[0,T]$, $y\in\R^n$ that
		\begin{align*}
			v(t, y) = \expE\lb g(Y^{t, y, \phi^*}(T))\mid \F_{\tau(t)}\rb.
		\end{align*}
	\end{enumerate}
\end{proof}

\subsection{Regular Value Functions Solve HJB}\label{sec:sec:regvalueHJB}
This section is devoted to the proof of Theorem \ref{theo:regularvaluesolvesHJB}, which shows that if there exists a regular modification of the value function $v$, then it solves \eqref{hjb}. Due to technical reasons explained in Remark \ref{rem:vbinsteadvblift}, it is convenient to first consider a regular modification $\vb$ parameterized instead by the rough lift $\Wtaunullb$. This is in line with \parencite{frizControlledRoughSDEs2024}, where conditioning on the full path on $[0,T]$ of a Brownian motion $W$ leads to a rough path parameter $\Xb$ in the value function. We thus can adopt some measurability results from there. In the end, using the lift from Example \ref{exmp:lift}, it is $u=\vb\circ\lift$ that solves \eqref{hjb}.\\

From now on call
\begin{align*}
	\bar{v}\colon [0,T]\times\R^n\times \RP_{\tau(0)} \times D([0,T],\R^d)\to [0,\infty)
\end{align*}
\hypertarget{goodvaluefct}{\textit{good version of $v$}} if for every $t\in[0,T]$, $y\in\R^n$, $\vb(t,y,\Wtaunullb,\Wtau)$ is a representative of the problem value $v(t,y)$ defined in \eqref{valuefct}
and if, in addition,
\begin{enumerate}[i)]
	\item for all $(t,y,\z)\in[0,T]\times\R^n\times D([0,T],\R^d)$, $\RP_{\tau(0)}\ni\wb\to \vb(t,y,\wb,\z)$ is Borel measurable,
	\item for all $\wb\in \RP_{\tau(0)}$, $\vb(\cdot,\wb,\cdot)\in\mathbb C_b^{1,4}$ bounded uniformly in $\wb$, i.e., 
	\begin{align*}
		\sup_{\wb\in \RP_{\tau(0)}} \{\vert D\vb(\cdot,\wb,\cdot)\vert_{\infty}, \vert\nabla^k\vb(\cdot,\wb,\cdot)\vert_{\infty},\ k=0,1,2,3,4\} <\infty,
	\end{align*}
	\item for $k=1,2,3$, $\nabla^k\vb(\cdot,\wb,\cdot)\in\mathbb{C}_b^{1,1}$ with $D\nabla^k\vb$ bounded uniformly in $\wb$,
	\item for  $l=0,1$ and $k=0,1,2$, $D^l\nabla^k\vb(\cdot,\wb,\cdot)$ are Lipschitz continuous in $(y,\z)$ uniformly in $(t,\wb)$, i.e.\ there exists $L>0$, such that for all $t\in[0,T]$, $\wb\in\RP_{\tau(0)}$, $y,\tilde y\in\R^n$ and $\z,\tilde{\z}\in D([0,T],\R^d)$ it holds
	\begin{align*}
		\vert D^l\nabla^k \vb(t,y,\wb,\z) - D^l\nabla^k \vb(t,\tilde y,\wb,\tilde{\z})\vert \le L(\vert y-\tilde y\vert - \vert\z-\tilde{\z}\vert_\infty).
	\end{align*}
\end{enumerate}

\subsubsection{Conditional Distributions}\label{sec:sec:sec:condDist}
We construct a conditional distribution of $(Y^{t,y,\phi},\phi,\Wtaunullb,\Wtau)$ given $\F_{\tau(t)}$ in order to have control over null sets in the main result Theorem \ref{theo:regularvaluesolvesHJB}. The goal is to pick a sufficiently regular modification of 
\begin{align*}
[t,\tau(t)]\ni s\mapsto \expE(\vb(s,Y^{t,y,\phi}(s),\Wtaunullb,\Wtau)-\vb(t,y,\Wtaunullb, \Wtau)\mid\F_{\tau(t)})
\end{align*}
after applying the \hyperlink{theo:roughfuncInt}{rough functional It\^o formula}.  
Moreover we calculate in Corollary \ref{cor:qvCondExp} the quadratic variation of (a suitable modification of)
\begin{align}\label{eq:cP4RI}
[t,\tau(t)]\ni s \mapsto \expE\lb\int_t^s\la\nabla_y\vb(\id,Y^{t,y,\phi},\Wtaunullb,\Wtau), f(Y^{t,y,\phi})\ra\dd\Wb\mid\F_{\tau(t)}\rb.
\end{align}
Recalling Definition \ref{defi:qv}, the notion of quadratic variation applies to continuous (deterministic) paths. Hence we need a continuous version of \eqref{eq:cP4RI}. Since $(\Wb(u), u\in[t,s])\in\F_{\tau(t)}$, we formally expect for $s\in[t,\tau(t)]$ that
\begin{align*}
	&\left[ \expE\lb \int_t^\cdot \la\nabla_y\vb(\id,Y^{t,y,\phi},\Wtaunullb,\Wtau),f(Y^{t,y,\phi})\ra\dd\Wb\mid \F_{\tau(t)} \rb \right]_s\\
    &= \int_t^s \big\vert\expE(f(Y^{t,y,\phi})^T\nabla_y\vb(\id,Y^{t,y,\phi},\Wtaunullb,\Wtau)\mid \F_{\tau(t)})\big\vert^2\dd\lambda.
	\end{align*}

We refer to \parencite[Section 2.3]{claissePseudoMarkovPropertyControlled2016} for a discussion of several subtle measurability issues concerning regular conditional probabilities on probability spaces $(\Omega, \F, \prob)$ where the $\sigma$-algebra $\F$ is completed. A convenient approach is to work instead with the conditional distribution of a random variable $X\colon \Omega\to T$ given a sub-$\sigma$-algebra $\G$ of $\F$. They have the advantage that $\F$ may contain all $\prob$-negligible sets and so can $\G$, only the $\sigma$-algebra for $T$ must not be completed with negligible sets of $\prob^X$, cf. \parencite[Theorem 10]{fadenExistenceRegularConditional1985}. 
Recall first the notion of conditional distributions from \parencite[Section 10.2]{dudley}.

\begin{defi}\label{defi:cP}
	Let $(\Omega, \F,\prob)$ be a probability space, $(T, \mathcal T)$ be measurable spaces. Further let $X\colon \Omega\to T$ be measurable and $\G$ a sub-$\sigma$-algebra of $\F$. Then a function $\cP{X}{\G}\colon\mathcal T\times\Omega\to[0,1]$ is called conditional distribution of $X$ given $\G$, if
	\begin{enumerate}[i)]
		\item for $\cP{}{\G}$-\as\ $\omega\in\Omega$, $\cP{X}{\G}(\cdot, \omega)$ is a probability measure on $\mathcal T$,
		\item for each $A\in\mathcal T$, $\cP{X}{\G}(A, \cdot)$ is $\G$-measurable and a representative of $\expE(\one_{X^{-1}(A)}\mid\G)$.
	\end{enumerate}
\end{defi}
We point out that, if it exists, we always find a version $\cP{X}{\G}$, such that $i)$ holds for every $\omega\in\Omega$. Moreover we recall that if $T$ is Polish and $\mathcal T$ the Borel $\sigma$-algebra, then conditional distributions for random variables always exists, cf. \parencite[10.2.2 Theorem]{dudley}. 

We first focus on solutions for $Y^\phi$ started at time $t=0$ and will explain later how to transform the general case to this one. In our setting,
\begin{align*}
Y^\phi\colon \Omega \to C([0,T],\R^n),\quad \Wtau\colon \Omega \to C([0,T],\R^d) \quad  \Wtaunullb\colon\Omega\to \RP_{\tau(0)},\quad \phi\colon \Omega\to L^0([0,T],\R^m)
\end{align*}
indeed take values in Polish spaces.
Fix $y\in\R^n$. We consider the 'solution set' $S$, which is the collection of 
of $(\y,\varphi,\wb)\in C([0,T],\R^n)\times L^0([0,T],\R^m)\times \RP_{\tau(0)}$, such that $\y$ solves the rough differential equation
	\begin{align}\label{solutionsset}
	 	\dd \y = b(\id,\phi,\y)\dd\lambda + f(\y)\dd\wb,\ \y(0)=y\text{ on }[0,\tau(0)].
	\end{align}

\begin{lem}\label{lem:solutionset} It holds that $S\in \borel_{\vert\cdot\vert_\infty}(C([0,T],\R^n))\otimes\borel_{d_{KF}}(L^0([0,T],\R^m))\otimes \mathfrak{C}_{\tau(0)}$ and $C^\alpha([0,T], \R^m)\in\borel_{\vert\cdot\vert_\infty}(C([0,T],\R^m))$.
\end{lem}

\begin{proof}
	By Theorem \ref{theo:sol2cSDE} part $iii)$, the solution map 
	\begin{align*}
		F\colon L^0([0,\tau(0)],\R^m)\times \RP_{\tau(0)} \to C([0,\tau(0)],\R^n)	
	\end{align*}
	is $\borel_{\alpha}(L^0([0,\tau(0)],\R^m)\otimes\mathfrak{C}_{\tau(0)}-\borel_{\vert\cdot\vert_\infty}(C([0,\tau(0)],\R^n))$ measurable.
	Since $\borel_{\alpha}(L^0([0,\tau(0)],\R^m)\otimes\mathfrak{C}_{\tau(0)}$ and $\borel_{\vert\cdot\vert_\infty}(C([0,\tau(0)],\R^d))$ are countable generated (Polish) it follows that the graph
	\begin{align*}
		\mathrm{graph}(F) =  \{(\y,\varphi,\wb)\in C([0,\tau(0)],\R^d)\times L^0([0,\tau(0)],\R^m)\times \RP_{\tau(0)}\colon \y = F(\varphi, \wb)\}
	\end{align*}
	belongs to $\sigma(\borel_{\vert\cdot\vert_\infty}(C([0,\tau(0)],\R^d))\times \borel_{\alpha}(L^0([0,\tau(0)],\R^m)\otimes \mathfrak{C}_{\tau(0)})$, see for instance \parencite[Proposition 2.1]{Musial1980}. Being Polish implies that the latter $\sigma$-algebra equals $\borel_{\vert\cdot\vert_\infty}(C([0,\tau(0)],\R^d))\otimes \borel_{\alpha}(L^0([0,\tau(0)],\R^m)\otimes \mathfrak{C}_{\tau(0)}$, cf.\ \parencite[4.1.7]{dudley}. Finally let $\mathrm{Res}\colon C([0,T],\R^d)\times L^0([0,T],\R^m)\times \RP_{\tau(0)}\to C([0,\tau(0)],\R^d)\times L^0([0,\tau(0)],\R^m)\times \RP_{\tau(0)}$, $(\y,\varphi,\wb)\mapsto(\y_{\mid [0,\tau(0)]}, \varphi_{\mid [0,\tau(0)]}, \wb)$. Then $S=\mathrm{Res}^{-1}(\mathrm{graph}(F))$ and the first assertion follows since the restriction $\mathrm{Res}$ is continuous.\\
	For the second claim notice that since for every $t\in[0, T]$, the coordinate projection $C([0, T], \R^m)\ni z\mapsto z(t)$ 
    are clearly measurable and that we can restrict supremum in $\vert\cdot\vert_\alpha\colon C([0,T],\R^d)\to \R\cup\{+\infty\}$  to rationals.
\end{proof}

\begin{rem}[$\vb$ instead of $\vb\circ\lift$]\label{rem:vbinsteadvblift}
In the proof we used that the solution map $F$ is Borel measurable. As mentioned above, it is crucial that $\mathfrak{C}_{\tau(0)}$ is not completed with $\prob^{\Wtaunullb}$-negligible sets. A formulation on the measurable space $(C([0,\tau(0)],\R^d), \borel{\vert\cdot\vert_{\infty}}(C([0,\tau(0)],\R^d)))$ is not possible. Indeed this amounts to considering in the proof $F\circ\lift$ instead of $F$. Recalling Example \ref{exmp:lift}, $\lift$ is $\borel{\vert\cdot\vert_{\infty}}(C([0,\tau(0)],\R^d))\vee\mathscr N-\mathfrak{C}_{\tau(0)}$ measurable, hence $S\in\borel_{\vert\cdot\vert_\infty}(C([0,T],\R^n))\otimes\borel_{d_{KF}}(L^0([0,T],\R^m))\otimes \borel{\vert\cdot\vert_{\infty}}(C([0,\tau(0)],\R^d))\vee\mathscr N$. 
\end{rem}

To simplify notation a bit we write $C^\alpha:= C^\alpha([0,T], \R^d)$  and for the Polish state space
\begin{align}\label{Ctau}
	U_{\tau(0)}:=C([0,T],\R^n)\times L^0([0,T],\R^m)\times \RP_{\tau(0)}\times C([0,T],\R^d).
\end{align}
The proof of the next technical lemma can be found in the Appendix \ref{sec:appendixB}.

\begin{lem}\label{lem:condDist} 
Let $W\colon\Omega\times[0,T]\to\R^d$ be a Brownian motion and $\tau$ a continuously differentiable time change. Let further $y\in\R^d$ and $\phi\in\adm$ and let $Y^\phi$ denote the solution to \eqref{cSDE} with $Y^\phi(0)=y$. Then there exist a conditional distribution
\begin{align*}
\cP{Y^\phi,\phi,\Wtaunullb, \Wtau}{\F_{\tau(0)}}\colon \Omega\times \borel(U_{\tau(0)}) \to [0,1]
\end{align*}
of $(Y^\phi,\phi,\Wtaunullb, \Wtau)$ given $\F_{\tau(0)}$ and similar $\cP{Y^\phi, \phi,\Wtau}{\F_{\tau(0)}}$ such that 
\begin{enumerate}[i)]
\item for every $\omega\in\Omega$, $\cP{Y^\phi,\phi,\Wtaunull, \Wtau}{\F_{\tau(0)}}(S\times C^\alpha, \omega) =1$, where $S$ is the solutions set defined in \eqref{solutionsset},
\item for every nonnegative Borel measurable function $f\colon U_{\tau(0)} \to [0,\infty]$ and every $\omega\in\Omega$ it holds that
\begin{align*}
&\int f(\y,\varphi, \wb, \z)\cP{Y^\phi,\phi,\Wtaunullb, \Wtau}{\F_{\tau(0)}}(\dd(\y,\varphi, \wb, \z), \omega)\\ 
&\quad= \int f(\y,\varphi, \wb, \z)\cP{Y^\phi,\phi, \Wtau}{\F_{\tau(0)}}(\dd(\y,\varphi, \z), \omega)_{\mid \wb = \Wtaunullb(\omega)},
\end{align*}
\item for every nonnegative Borel measurable function $f\colon \RP_{\tau(0)}\times C([0,T],\R^d) \to [0,\infty]$ and every $\omega\in\Omega$, it holds that
\begin{align*}
\int f(\wb, \z)\cP{Y^\phi,\phi,\Wtaunullb, \Wtau}{\F_{\tau(0)}}(\dd(\y,\varphi, \wb, \z), \omega) = \int f(\wb, \z) \prob^{\Wtau}(\dd\z)_{\mid \wb = \Wtaunullb(\omega)}.
\end{align*}
\end{enumerate}
\end{lem}

\subsubsection{Quadratic Variation of Conditional Expectation}
We first calculate the quadratic variation of \eqref{eq:cP4RI} in the special case $t=0$, then use a transformation to obtain the result for arbitrary $t\in[0,T]$ in Corollary \ref{cor:qvCondExp}.

\begin{lem}\label{lem:qvCondExp} 
	Assume there exists a \hyperlink{goodvaluefct}{good version} $\vb$ of $v$. Let $\phi\in\adm$, denote by $Y^\phi$ the solution to \eqref{cSDE} with $Y^\phi(0)=y\in\R^n$ and consider the null set $N$ and  the sequence of partitions $(\partition_n)$ from Corollary \ref{cor:2varXW}. Then it holds for every $t\in[0,\tau(0)]$ that
	\begin{align}\label{lem:qvCondExp:eq}
		\begin{split}
		&\left[ \int_{S\times C^\alpha} \int_0^\cdot \la\nabla_y \vb(\id,\y,\wb,\z), f(\y)\ra\dd\wb\ \cP{Y^\phi,\phi,\Wtaunullb, \Wtau}{\F_{\tau(0)}}(\dd(\y,\varphi,\wb,\z), \omega)\right](t)\\
		&\quad = \int_0^t \left\vert \int \la\nabla_y \vb(\id,\y,\wb,\z), f(\y)\ra \cP{Y^\phi,\phi,\Wtau}{\F_{\tau(0)}}(\dd(\y,\varphi,\z), \omega)\right\vert^2 \dd\lambda_{\big\vert \wb = \Wtaunullb(\omega)},
		\end{split}
	\end{align}
	where the solution set $S$ is defined in \eqref{solutionsset} and  $\cP{Y^\phi,\phi,\Wtaunullb, \Wtau}{\F_{\tau(0)}}$ (resp. $\cP{Y^\phi,\phi,\Wtau}{\F_{\tau(0)}}$) are the conditional distributions of $(Y^\phi,\Wtaunullb,\Wtau)$ (resp. $(Y,\phi,\Wtau)$) given $\F_{\tau(0)}$ constructed in Lemma \ref{lem:condDist}.
\end{lem}

\begin{proof}
	We first show that all integrals in \eqref{lem:qvCondExp:eq} are well defined and then prove equality \eqref{lem:qvCondExp:eq}. For convenience we denote the good version $\bar v$, by $v$, too and write $Y = Y^\phi$. Additionally we now trust the reader to identify the Borel $\sigma$-algebras and simply call a map Borel measurable.
	
	Let $(r,s)\in\Delta_{\tau(0)}$. By the regularity assumptions on the \hyperlink{goodvaluefct}{good version $v$}, $\wb\mapsto v(s,y,\wb,\z)$ is Borel measurable for every $(s,y,\z)\in [0,T]\times\R^n\times D([0,T],\R^d)$ and $\nabla_y v(\cdot,\wb,\cdot)$ is continuous. Since $\nabla_y v$ is a pathwise limit by Definition \ref{defi:nabla}, $\wb\mapsto \nabla_y v(s,y,\wb,\z)$ is measurable in the same way. Thus $\nabla_yv$ is a Carath\'eodory function. Moreover $\nabla_y v$ restricted to continuous paths is continuous, too, and $[0,T]\times\R^n\times C([0,T],\R^d)$ is a separable, metrizable space. Therefore $\nabla_y v\colon [0,T]\times\R^n\times\RP_{\tau(0)}\times C([0,T],\R^d)\to \R^n$ is Borel measurable by \parencite[Lemma III.14]{castaingConvexAnalysisMeasurable1977} and so is $\nabla_y v$ considered as a causal map in $\y\in C([0,T],\R^n)$ depending only on the current value $\nabla_y v(s, \y(s),\wb,\z)$. The same argument holds for $\nabla^2_y v$. Considering also $f$ as a causal map depending only on the current value $\y\mapsto f(\y(s))$, the same follows for
	\begin{align}\label{lem:qvCondExp:eq:3}
		\y\mapsto f(\y(s)), \quad y\mapsto \nabla_yf(\y(s))\cdot f(\y(s)).
	\end{align}
    Together with the uniform bounds of the \hyperlink{goodvaluefct}{good version} and $f$, we deduced that the RHS of \eqref{lem:qvCondExp:eq} is well-defined. Further it follows that
	\begin{align}\label{lem:qvCondExp:eq:4}
		\begin{split}
		(\y,\wb,\z)\mapsto \sum_{[r,s]\in\tilde{\partition}}& \la \nabla_y v(r,\y(r),\wb,\z), f(\y(r))(\w(s)-\w(r))\ra  + \la \nabla^2_y v(r,\y(r),\wb,\z), f(\y(r))^{\otimes 2}\wt(r,s)\ra \\
		&\quad + \la\nabla_y v(r,\y(r),\wb,\z), \nabla f(\y(r))\otimes f(\y(r))\wt(r,s)\ra 
		\end{split}
	\end{align}
	is Borel measurable for any partition $\tilde{\partition}$ of $[0,\tau(0)]$. By the regularity assumptions on $v$, the last expression converges for $(\y,\wb,\varphi)\in S$, (so $y$ RDE solution on $[0,\tau(0)]$) to a rough integral by Theorem \ref{theo:funcIto4SolXW}. This implies that 
	\begin{align*}
		\mathcal I\colon [0,\tau(0)]\times  U_{\tau(0)} &\to \R\\
		(s,\y,\varphi,\wb,\z)&\mapsto\int_0^s\la\nabla_yv(\id,\y,\wb,\z), f(\y)\ra\dd\wb \ \one_{S\times C^\alpha}
	\end{align*} 
	is the well-defined pointwise limit of \eqref{lem:qvCondExp:eq:4} (summing over $\tilde{\partition}\cap[0,s]$) as $\vert \tilde{\partition}\vert\to 0$. Therefore $\mathcal I(t,\cdot)$ is Borel measurable for every $s\in[0,\tau(0)]$, too.
	It remains to show that $\mathcal I(t,\cdot)$ is integrable with respect to $\cP{Y^\phi,\phi,\Wtaunullb, \Wtau}{\F_{\tau(0)}}$. In order to do so we make use of the strong estimates rough path theory provides. Thereby we disregard deterministic constants that depend on the coefficients $b,f$ from \eqref{cSDE} and the regularity of $v$. Also we disregard dependencies of the form $t^\alpha\le \tau(0)^\alpha\le T^\alpha$. By Theorem \ref{theo:RI} combined with Leibniz type estimates \eqref{theo:funcIto4SolXW:eq1}, \eqref{theo:funcIto4SolXW:eq2} for the $\w$-controlled product $\la\nabla_yv(\id,y, \wb,\z), f(y)\ra$, it holds for $(\y,\varphi,\wb,\z)\in S\times C^\alpha$ that
	\begin{align}
		&\vert\mathcal I(s,\y,\varphi,\wb,\z)\vert\notag\\
		\begin{split}\label{lem:qvCondExp:eq:5}
		&\lesssim \vert\w\vert_\alpha \lb 1 + \vert  R^{\la\nabla_yv(\id,\y, \wb,\z), f(\y)\ra,\w}\vert_{\alpha+\alpha^2} + \vert\la\nabla_yv(\id,\y,\wb,\z), f(\y)\ra\vert_\alpha \rb\\
		&\quad + \vert\wt\vert_{2\alpha} \lb 1 + \vert\la\nabla^2_y v(\id,\y,\wb,\z), f^{\otimes 2}(\y)\ra + \la\nabla_y v(\id,\y,\wb,\z), \nabla_yf(\y)\cdot f(\y)\ra\vert_\alpha\rb
		\end{split}\\
		&\lesssim \vert\w\vert_\alpha\Big( 1 + \vert R^{\nabla_{(y,\z)}v(\id,\y,\wb,\z),(\y,\z)}\vert_{\alpha + \alpha^2} + \vert \y\vert_\alpha\vert f(\y)\vert_\alpha + \vert R^{\y,\w}\vert_{2\alpha} + \vert R^{f(\y), \w}\vert_{2\alpha}\notag\\
		&\qquad\qquad\ + \vert\nabla_{(y,\z)} v(\id,\y,\wb,\z)\vert_\alpha + \vert f(\y)\vert_\alpha\Big)\notag\\
		&\quad + \vert\wt\vert_{2\alpha} \lb 1+ \vert f(\y)\vert_\alpha + \vert \nabla_{(y,\z)}^2 v(\id,\y,\wb,\z)\vert_\alpha  + \vert\nabla_yf(\y)f(\y)\vert_\alpha  + \vert\nabla_{(y,\z)} v(\id,\y,\wb,\z)\vert_\alpha\rb.\notag
	\end{align}
	Applying Remark \ref{rem:nablaFcontrolled} for $R^{\nabla_{(y,\z)}v(\id,\y,\wb,\z),(\y,\z)}$, $\nabla_{(y,\z)} v(\id,\y,\wb,\z)$ and $\nabla_{(y,\z)}^2 v(\id,\y,\wb,\z)$; as well as the estimates \eqref{estimate:sigma(Y)} for $f(\y)$, $\nabla_yf(\y)\cdot f(\y)$ and $R^{f(\y), \w}$, the estimate simplifies to 
	\begin{align*}
		&\vert\mathcal I(t,\y,\varphi,\wb,\z)\vert\\
		&\lesssim \vert\w\vert_\alpha \lb 1+\vert (\y,\z)\vert_\alpha + \vert \y\vert_\alpha^2 + \vert R^{\y,\w}\vert_{2\alpha} + \vert \y\vert_\alpha\rb +\vert\wt\vert_{2\alpha}\lb 1+\vert \y\vert_\alpha + \vert (\y,\z)\vert_\alpha \rb\\
		&\lesssim (\vert\w\vert_\alpha + \vert\wt\vert_{2\alpha})\lb 1+ (\vert \y\vert_\alpha \vee \vert \y\vert_\alpha^2) + \vert R^{\y,\w}\vert_{2\alpha} +\vert \z\vert_\alpha\rb.
	\end{align*}
	It holds $\vert\w\vert_\alpha + \vert\wt\vert_{2\alpha} \le 2(1+\vert\wb\vert_\alpha)^2$. Moreover since $(\y,\varphi,\wb)\in S$ the apriori bounds from Theorem \ref{theo:sol2cSDE} part $i)$ for solutions to rough differential equations apply. 
	\begin{align}\label{lem:qvCondExp:eq:aprorix}
		\vert \y\vert_\alpha \lesssim (\vert\wb\vert_\alpha+ T^{1-2\alpha})\vee(\vert\wb\vert_\alpha + T^{1-2\alpha})^{1/\alpha}
	\end{align}
	and
	\begin{align*}
		\vert R^{\y,\w}\vert_{2\alpha}\lesssim (1+ \vert\wb\vert_\alpha + T^{1-2\alpha})^2\vee (1+\vert\wb\vert_\alpha+ T^{1-2\alpha})^{2/\alpha}.
	\end{align*}
	Hence it follows for $t\in[0, \tau(0)]$ that
	\begin{align*}
		\int_{S\times C^\alpha} &\left\vert \mathcal I(t,\y,\varphi,\wb,\z)\right\vert \cP{Y,\phi,\Wtaunullb,\Wtau}{\F_{\tau(0)}}(\dd(\y,\varphi,\wb,\z),\omega)\\
		&\le\int_{S\times C^\alpha}  C(\vert\wb\vert_\alpha)(C(\vert\wb\vert_\alpha) + \vert \z\vert_\alpha) \cP{Y,\phi,\Wtaunullb,\Wtau)}{\F_{\tau(0)}}(\dd(\y,\varphi,\wb,\z),\omega),
	\end{align*}
	where $C(\vert\wb\vert_\alpha)\varpropto (1+\vert\wb\vert_\alpha)^q$ for some $q\in[2,\infty)$. By Lemma \ref{lem:condDist} part $iii)$, it holds that
	\begin{align*}
		&\int  C(\vert\wb\vert_\alpha)(C(\vert\wb\vert_\alpha) + \vert \z\vert_\alpha) \cP{Y,\phi,\Wtaunullb,\Wtau}{\F_{\tau(0)}}(\dd(\y,\varphi,\wb,\z),\omega)\\
		&= C(\vert\Wtaunullb(\omega)\vert_\alpha)^2 + C(\vert\Wtaunullb(\omega)\vert_\alpha)\expE \vert \Wtau\vert_\alpha.
	\end{align*}
	The last constant is finite on $\Omega\setminus N$. Indeed $\vert \Wtaunull(\omega)\vert_\alpha<\infty$ by definition of $ N$ in Corollary \ref{cor:2varXW}, so also $C(\vert\Wtaunullb(\omega)\vert_\alpha)<\infty$. It follows from the discussion in Example \ref{exmp:lift} together with $\tau$ continuously differentiable that $\vert\Wtau\vert_\alpha$ is integrable.
	
 	We next show the identity \eqref{lem:qvCondExp:eq}. For $\omega\in\Omega$ and $(u,v)\in\Delta_{\tau(0)}$ Lemma \ref{lem:condDist} part $ii)$, yields that
 	\begin{align*}
 		&\int \la\nabla_yv(u,\y,\wb,\z), f(\y)(\w(v)-\w(u))\ra \cP{Y,\phi,\Wtaunullb,\Wtau}{\F_{\tau(0)}}(\dd(\y,\varphi,\wb,\z), \omega)\\
 		&=\int \la\nabla_yv(u,\y,\Wtaunullb(\omega),\z), f(\y)\ra \cP{Y,\phi, \Wtau}{\F_{\tau(0)}}(\dd(\y,\varphi,\z),\omega)(W(\omega)(v)-W(\omega)(u)).
 	\end{align*}
	We abbreviate $\la\nabla_y v, f\ra(u)(\y,\wb,\z):= \la\nabla_yv(u, \y,\wb, \z), f(\y(u))\ra$ etc., and restrict to $\Omega\setminus N$ and $\vert v-u\vert\le 1$. It holds that
	\begin{align*}
		&\lb \int \mathcal I(v,\cdot)-\mathcal I(u,\cdot)\cP{Y,\phi,\Wtaunullb,\Wtau}{\F_{\tau(0)}}\rb^{\otimes 2} - \la\lb\int \la\nabla_y v, f\ra(u) \cP{Y,\phi, \Wtau}{\F_{\tau(0)}}\rb^{\otimes 2},(W(v)-W(u))^{\otimes 2}\ra\\
		&= \lb \int_{S\times C^\alpha}R^{\mathcal I}(u,v) \cP{Y,\phi,\Wtaunullb,\Wtau}{\F_{\tau(0)}}\rb^2\\
		&\quad  + 2 \lb \int_{S\times C^\alpha}R^{\mathcal I}(u,v) \cP{Y,\phi,\Wtaunullb,\Wtau}{\F_{\tau(0)}}\rb\lb\int \la\nabla_y v, f\ra(u)  \cP{Y,\phi, \Wtau}{\F_{\tau(0)}}\rb(W(v)-W(u)).
	\end{align*}
	By Theorem \ref{theo:RI}, it holds for $(\y,\varphi,\wb,\z)\in S\times C^\alpha$ that
	\begin{align*}
		&\left\vert R^{\mathcal I(\cdot, \y,\varphi,\wb,\z)}(u,v)\right\vert\\
		&\lesssim (\vert\w\vert_\alpha\vert  R^{\la\nabla_yv, f\ra, \w}\vert_{\alpha+\alpha^2} + \vert\wt\vert_{2\alpha}\vert\la \nabla^2_y v, f^{\otimes 2}\ra + \la\nabla_y v, \nabla_yf\cdot f\ra\vert_\alpha)\vert u-v\vert^{2\alpha+\alpha^2}\\
		&\quad + \vert\la \nabla^2_y v, f^{\otimes 2}\ra + \la\nabla_y v, \nabla_yf\cdot f\ra\vert_\infty\vert\wt\vert_{2\alpha}\vert v-u\vert^{2\alpha}\\
		&\lesssim (\vert\w\vert_\alpha + \vert \wt\vert_{2\alpha})(\vert  R^{\la\nabla_yv, f\ra, \w}\vert_{\alpha+\alpha^2} +\vert\la \nabla^2_y v, f^{\otimes 2}\ra + \la\nabla_y v, \nabla_yf\cdot f\ra\vert_\alpha) \vert v-u\vert^{2\alpha}.
	\end{align*} 
	Now the last constant can be estimated in the same manner as \eqref{lem:qvCondExp:eq:5}. Consequently we get that
	\begin{align*}
		\left\vert\int_{S\times C^\alpha}R^{\mathcal I}(u,v) \cP{Y,\phi,\Wtaunullb,\Wtau}{\F_{\tau(0)}}\right\vert \lesssim C(\Wtaunullb,\expE \vert\Wtau\vert_\alpha)\vert v-u\vert^{2\alpha}
	\end{align*}
	for a constant $C(\Wtaunullb,\expE \vert\Wtau\vert_\alpha)>0$ that is finite on $\Omega\setminus N$. Therefore it follows on $\Omega\setminus N$ as $n\to\infty$ that
	\begin{align*}
	\sum_{[u,v]\in\partition_n} \lb \int_{S\times C^\alpha}R^{\mathcal I}(u,v) \cP{Y,\phi,\Wtaunullb,\Wtau}{\F_{\tau(0)}}\rb^2
	\lesssim C(\Wtaunullb,\expE \vert\Wtau\vert_\alpha)^2 \sum_{[u,v]\in\partition_n} \vert v-u\vert^{4\alpha}\to 0
	\end{align*}
	since in the regime of the \hyperlink{theo:roughfuncInt}{rough functional It\^o formula} $4\alpha\ge 3\alpha>1$. Similarly,
	\begin{align*}
		&\sum_{[u,v]\in\partition_n} \left\vert \int_{S\times C^\alpha}R^{\mathcal I}(u,v) \cP{Y,\phi,\Wtaunullb,\Wtau}{\F_{\tau(0)}}\right\vert\left\vert\int \la\nabla_y v, f\ra(u)  \cP{Y,\phi, \Wtau}{\F_{\tau(0)}}\right\vert \vert W(v)-W(u)\vert\\
		&\lesssim C(\Wtaunullb,\expE \vert\Wtau\vert_\alpha)\vert W\vert_\alpha \sum_{[u,v]\in\partition_n} \vert v-u\vert^{3\alpha}\to 0.
	\end{align*}
	Finally since $W(\omega)$ has for $\omega\in\Omega\setminus N$ finite quadratic variation along $\partition$ with $[W(\omega)]=\id\ \diag(\one_d)$ and
	\begin{align*}
		[0,\tau(0)]\ni u\mapsto \int \la\nabla_yv(u,\y,\Wtaunullb(\omega),\z), f(\y)\ra \cP{Y,\phi, \Wtau}{\F_{\tau(0)}}(\dd(\y,\varphi,\z), \omega)\in\R^d
	\end{align*}
	is continuous, it follows on $\Omega\setminus N$ from Definition \ref{defi:qv} that 
	\begin{align*}
 		&\lim_{n\to\infty}\sum_{\substack{[u,v]\in\partition_n,\\ u\le\tau(0)}}\la\lb\int \la\nabla_y v, f\ra(u) \cP{Y,\phi, \Wtau}{\F_{\tau(0)}}\rb^{\otimes 2},(W(v)-W(u))^{\otimes 2}\ra \\
        &= \int_0^{\tau(0)} \la\lb\int \la\nabla_y v, f\ra \cP{Y,\phi, \Wtau}{\F_{\tau(0)}}\rb^{\otimes 2}, \diag(\one_d)\ra\dd\lambda = \int_0^{\tau(0)} \left\vert \int \la\nabla_y v,f\ra \cP{Y^\phi,\phi,\Wtau}{\F_{\tau(0)}}\right\vert^2 \dd\lambda.
	\end{align*}
	The same arguments holds when integrating only up to $t\in[0,\tau(0)]$. Together we showed that on $\Omega\setminus N$
	\begin{align*}
	 \lim_{n\to\infty}\sum_{\substack{[u,v]\in\partition_n,\\ u\le t}} \lb \int \mathcal I(v,\cdot)-\mathcal I(u,\cdot)\cP{Y,\phi,\Wtaunullb,\Wtau}{\F_{\tau(0)}}\rb^2 \to \int_0^s \left\vert \int \la\nabla_y v,f\ra \cP{Y^\phi,\phi,\Wtau}{\F_{\tau(0)}}\right\vert^2 \dd\lambda.
	\end{align*}
 	By Theorem \ref{theo:weakconvergence} this proves \eqref{lem:qvCondExp:eq}.
\end{proof}

To deduce that Lemma \ref{lem:qvCondExp} also holds true for a solution $Y^{t,\phi}\in C([t,T],\R^n)$ of \eqref{cSDE} started at $t\in[0,T)$ and windows $[t,\tau(t)]$ we consider the following transformation: 
Let $S^t$ be as defined in \eqref{solutionsset} except that we consider  the rough differential equation over $[t,\tau(t)]$, with some fixed initial condition at time $t$; $\Wtaut$ denote $W\circ\tau - W(\tau(t))\in C([t,T],\R^d)$, compare \eqref{Wtau} and similar to \eqref{Ctau}, $U_{\tau(t)}=C([t,T],\R^d)\times L^0([t,T],\R^m)\times\RP_{\tau(t)}\times C([t,T],\R^d)$. Then it is clear that analogues of Lemma \ref{lem:solutionset} and Lemma \ref{lem:condDist} hold, so that we find conditional distributions of $(Y^{t,\phi},\Wtaunulltb,\Wtaut)$ (resp. $(Y^{t,\phi},\phi,\Wtaut)$) given $\F_{\tau(t)}$ that concentrate surely on $S^t\in\borel_{\vert\cdot\vert_\infty}(C([t,T],\R^d))\otimes\borel_{d_{KF}}(L^0([t,T],\R^m))\otimes \mathfrak{C}_{\tau(t)}$.
Set further
\begin{align}
	\begin{split}\label{cor:qvCondExp:trafo:1}
		\vb^t\colon [t,T]\times C([t,T],\R^d)\times \RP_{\tau(t)}\times D([t,T],\R^d)&\to \R,\\
		(s,\y,\wb,\z)&\mapsto \vb(s,\y(s),\wb_{\mid [0,\tau(0)]},\wb\sqcup_t\z),
	\end{split}
\end{align}
where the concatenation $\sqcup_t$ at $t\in[0,T]$ is defined by
\begin{align}\label{cor:qvCondExp:concatenration}
	\begin{split}
	\sqcup_t\colon \RP_{\tau(t)}\times D([t,T],\R^d)&\to  D([0,T],\R^d)\\
	(\w, \z) &\mapsto \wtau(\cdot\wedge t) + \z(\cdot\vee t).
	\end{split}
\end{align}
In particular it holds for $s\in[0,T]$ that
\begin{align}\label{cor:qvCondExp:trafo:2}
	(\Wtaunulltb\sqcup_t\Wtaut)(s) = \Wtaunullt(\tau(s\wedge t)) - \Wtaunullt(\tau(0)) + \Wtaut(s\vee t) = \Wtau(s).
\end{align}

\begin{cor}\label{cor:qvCondExp} 
	Under the assumptions of Lemma \ref{lem:qvCondExp} and writing $Y^{t,\phi}$ for the solution started instead at $t$, there exists for every $t\in[0,T)$, conditional distributions $\cP{Y^{t,\phi},\Wtaunulltb,\Wtaut}{\F_{\tau(t)}}$ (resp. $\cP{Y^{t,\phi},\phi,\Wtaut}{\F_{\tau(t)}}$) of $(Y^{t,\phi},\Wtaunulltb,\Wtaut)$ (resp. $(Y^{t,\phi},\phi,\Wtaut)$) given $\F_{\tau(t)}$, such that for every $\omega\in\Omega\setminus N$, $s\in[t,\tau(t)]$  it holds that
	\begin{align}\label{cor:qvCondExp:eq}
		\begin{split}
			&\left[ \int_{S^t\times C^\alpha} \int_t^\cdot \la\nabla_y \vb^t(\id,\y,\wb,\z), f(\y)\ra \dd\wb\ \cP{Y^{t,\phi},\Wtaunulltb,\Wtaut}{\F_{\tau(t)}}(\dd(\y,\varphi,\wb,\z), \omega)\right](s)\\
			&\quad = \int_t^s \left\vert \int \la\nabla_y \vb^t(\id,\y,\wb,\z), f(\y)\ra \cP{Y^{t,\phi},\phi,\Wtaut}{\F_{\tau(0)}}(\dd(\y,\varphi,\z), \omega)\right\vert^2 \dd\lambda_{\big\vert \wb = \Wtaunulltb(\omega)}.
		\end{split}
	\end{align}
\end{cor}

\begin{proof}
	Clearly $\sqcup_t$ and the restriction $\w\mapsto\w_{\mid [0,\tau(0)]}$ are continuous with respect to $\vert\cdot\vert_\infty$. Recalling the properties of \hyperlink{goodvaluefct}{good version of $v$}, $\vb^t$ is Borel measurable, too. Moreover the differentiablity assumptions carry over to $\vb^t$, noting that for $s\in[t,T]$
	\begin{align*}
		\wb\sqcup_t(\z_s + h\one_{[s,T]}) = \wb\sqcup_t\z + h\one_{[s,T]},\quad \wb\sqcup_t \z_s = (\wb\sqcup_t \z)_s.
	\end{align*}
	Consequently it follows on $[t,T]$ for the causal derivatives that $\nabla_y\vb^t(s,\y,\wb,\z) = \nabla_y\vb(s,\y(s),\wb_{\mid [0,\tau(0)]}, \wb\sqcup_t \z)$ and $D\vb^t(s,\y,\wb,\z)=Dv(s,\y(s),\wb_{\mid [0,\tau(0)]}, \wb\sqcup_t\z)$ etc. Noting that $\vert\wb\sqcup_t \z-\wb\sqcup_t\tilde{\z}\vert_\infty = \vert z-\tilde{\z}\vert_\infty$, also the Lipschitz conditions on the causal derivatives carry over from $\vb$ to $\vb^t$. Applying Lemma \ref{lem:qvCondExp} with $\vb^t$ and $Y^{t,\phi}$ proves the claim.
\end{proof}

\subsubsection{Good Version of the Value Function Solves HJB}

\begin{theo}\label{theo:regularvaluesolvesHJB}
	Assume that the value function $v$ defined in \eqref{valuefct} has a \hyperlink{goodvaluefct}{good version} $\vb$. Then $u:=\vb\circ\lift$ solves \eqref{hjb} and $u(\cdot, \Wtaunull, \Wtau)$ is a modification of $v$.
\end{theo}

\begin{proof} With some abuse of notation let $v$ denote also the good version. 
	We first prove the second equation in \eqref{hjb}. 
	Let $y\in\R^n$, $\phi\in\adm$, $Y$ denote the solution to \eqref{cSDE} started in $t\in[0,T)$.
	We may apply Theorem \ref{theo:funcIto4SolXW} to $v$ in the exact same manner as in the beginning of the proof of the \hyperlink{theo:verification}{Verification Theorem~\ref{theo:verification}}. Taking the conditional expectation with respect to $\F_{\tau(t)}$, it follows for $s\in[t,T)$ that
	\begin{align}
	&\expE(v(s,Y(s),\Wtaunullb,\Wtau) - v(t,y,\Wtaunullb,\Wtau)\mid\F_{\tau(t)})\notag\\ 
	&=\expE\lb \int_t^s Dv(\id,Y,\Wtaunullb,\Wtau) + H(\id,Y,\phi,\Wtaunull,\Wtau) \dd\lambda\ \mid\F_{\tau(t)}\rb\label{theo:regularvaluesolvesHJB:eq:1}\\
	&\quad + \expE\lb\int_t^s \la \nabla_y v(s,Y,\Wtaunullb,\Wtau), f(Y)\ra \dd\Wb\mid\F_{\tau(t)}\rb\notag\\
	&\quad + \expE\lb\int_t^s \la\nabla_z v(\id,Y,\Wtaunullb,\Wtau),\dd \Wtau\ra\mid\F_{\tau(t)}\rb.\label{theo:regularvaluesolvesHJB:eq:2}\\
	\begin{split}\label{theo:regularvaluesolvesHJB:eq:3}
	&=\int_t^s \expE(Dv(\id,Y,\Wtaunullb,\Wtau) + H(\id,Y,\phi,\Wtaunullb,\Wtau)\mid\F_{\tau(t)}) \dd\lambda\\
	&\quad +  \expE\lb\int_t^s \la \nabla_y v(\id, Y, \Wtaunull, \Wtau), f(Y)\ra \dd\Wb\mid\F_{\tau(t)}\rb.
	\end{split}
	\end{align}
	Here we used in the second equality that \eqref{theo:regularvaluesolvesHJB:eq:2} vanishes since the It\^o integral is a martingale (as explained above for \eqref{theo:verification:eq:1}) and that we may exchange the Lebesgue integral with taking the conditional expectation in \eqref{theo:regularvaluesolvesHJB:eq:1} (at least after picking a $\borel([t,s])$- measurable modification of $r\mapsto \expE(Dv(r,Y,\Wtaunullb,\Wtau) + H(r,Y,\phi,\Wtaunullb,\Wtau)\mid\F_{\tau(t)})$). 
	By the supermartingale property discussed in Remark \ref{rem:valuefctsupermartingale} it holds for $(s_1,s_2)\in\Delta_{[t,T]}$,
	\begin{align*}
		v(s_1,Y(s_1),\Wtaunullb,\Wtau) \ge \expE(v(s_2,Y(s_2),\Wtaunullb,\Wtau)\mid\F_{\tau(s_1)})
	\end{align*}
	with respect to \as-ordering in $L^0$. Subtracting $v(t,y,\Wtaunullb,\Wtau)$ on both sides and taking the conditional expectation with respect to $\F_{\tau(t)}$, the first equation implies that
	\begin{align*}
		&\int_t^{s_1} \expE(Dv(\id,Y,\Wtaunullb,\Wtau) + H(\id,Y,\phi,\Wtaunullb,\Wtau)\mid\F_{\tau(t)}) \dd\lambda\notag\\
		&\quad +  \expE\lb\int_t^{s_1} \la \nabla_y v(\id,Y,\Wtaunullb,\Wtau), f(Y)\ra \dd\Wb\mid\F_{\tau(t)}\rb\\
		&\ge \int_t^{s_2} \expE(Dv(\id,Y,\Wtaunullb,\Wtau) + H(\id, Y,\phi,\Wtaunullb,\Wtau)\mid\F_{\tau(t)}) \dd\lambda\notag\\
		&\qquad +  \expE\lb\int_t^{s_2} \la \nabla_y v(\id,Y,\Wtaunullb,\Wtau), f(Y)\ra \dd\Wb\mid\F_{\tau(t)}\rb
	\end{align*}
	Hence $s\mapsto$ \eqref{theo:regularvaluesolvesHJB:eq:3} is non-decreasing with respect to \as-ordering. We next choose the conditional distributions of $(Y,\phi,\Wtaunulltb,\Wtaut)$ and $(Y,\phi,\Wtaut)$ given $\F_{\tau(t)}$ from Corollary \ref{cor:qvCondExp} and write $\vb$ for the transformation defined in \eqref{cor:qvCondExp:trafo:1}. Recall that $v(s,Y(s),\Wtaunullb,\Wtau) = \vb(s, Y(s),\Wtaunulltb,\Wtaut)$ for $s\in[t,T]$ by \eqref{cor:qvCondExp:trafo:2}. Using the same transformation in the arguments to go over from $H$ to $\bar{H}$, the process
	\begin{align*}
		[t,\tau(t)]\ni s\mapsto &\int_t^s \int D\vb(\id,\y,\Wtaunulltb,\z) + \bar{H}(\id,\y,\varphi,\Wtaunulltb,\z) \cP{Y,\phi,\Wtaut}{\F_{\tau(t)}}(\dd(\y,\varphi,\z),\cdot)\\
		&\quad +\int_{S^T\times C^\alpha} \int_t^s\la \nabla_y\vb(\id,\y,\wb,\z),f(\y)\ra \dd\wb\ \cP{Y,\phi,\Wtaunulltb,\Wtaut}{\F_{\tau(t)}}(\dd(\y,\varphi,\wb,\z), \cdot)
	\end{align*}
	is well-defined and a modification of $s\mapsto$ \eqref{theo:regularvaluesolvesHJB:eq:3}. Since our specific modification is continuous in $s$, there exists $ N^{t,y}\in\F$ with $\prob( N^{t,y}) = 0$ and such that on $\Omega\setminus N^{t,y}$ every sample path of the modification is non-decreasing.  Therefore they have bounded variation, so their quadratic variation is zero. Together with Corollary \ref{cor:qvCondExp} (with its particular null set $N$) we conclude that for $\omega\in\Omega\setminus( N\cup N^{t,y})$ and $s\in[t, \tau(t)]$, it holds that
	\begin{align*}
		0&= \left[ \int_{S\times C^\alpha} \int_t^\cdot \la\nabla_y \vb(\id,\y,\wb,\z), f(\y)\ra \dd\wb\ \cP{(Y,\phi,\Wtaunulltb, \Wtaut)}{\F_{\tau(t)}}(\dd(\y,\varphi,\wb,\z), \omega)\right](s)\\
		&= \int_t^s \left\vert \int_{C^\alpha} \la\nabla_y \vb(\id,\y,\wb,\z), f(\y)\ra \cP{Y,\phi,\Wtaut}{\F_{\tau(t)}}(\dd(\y,\varphi,\z), \omega)\right\vert^2 \dd s_{\big\vert \wb = \Wtaunulltb(\omega)}.
	\end{align*}
	Taking the right-derivative at $s=t$, the fundamental theorem of calculus, the construction of $S^t$ and Lemma \ref{lem:condDist} imply that in $\R^d$,
	\begin{align*}
		0&= \int_{C^\alpha} \la\nabla_y \vb(t, \y(t),\wb,\z), f(\y(t))\ra \cP{Y,\phi, \Wtaut}{\F_{\tau(t)}}(\dd(\y,\varphi,\z), \omega)_{\big\vert \wb = \Wtaunulltb(\omega)}\\
		&=\int_{S^t\times C^\alpha} \la\nabla_y \vb(t, \y(t),\wb,\z), f(\y(t))\ra  \cP{(Y,\phi,\Wtaunulltb, \Wtaut)}{\F_{\tau(t)}}(\dd(\y,\varphi,\wb,\z), \omega)\\
		&= \int\la\nabla_y \vb(t,y,\wb,\z), f(y)\ra \prob^{\Wtaut}(\dd \z)_{\big\vert \wb = \Wtaunulltb(\omega)}.
	\end{align*}
	By construction of $\vb$, the concatenation $\sqcup_t$ in \eqref{cor:qvCondExp:concatenration} and the fact that $v$ is causal in the last component,
	\begin{align*}
		\vb(t,y,\wb,\z) = v(t,y,\wb_{\mid [0,\tau(0)]},\wb\sqcup_t\z) = v(t,y,\wb_{\mid [0,\tau(0)]},\wtau).
	\end{align*}
	Consequently on $\Omega\setminus( N\cup N^{t,y})$,
	\begin{align}\label{theo:regularvaluesolvesHJB:eq:4}
		0 = \la \nabla_y v(t,y,\Wtaunullb,\Wtau), f(y)\ra.
	\end{align}
	Finally since for all $\omega\in\Omega$, $(t,y)\mapsto\la \nabla_y v(t,y,\Wtaunullb(\omega),\Wtau(\omega)), f(y)\ra$ is continuous by the regularity assumptions on the good version $v$ and continuity of $f$, we find a null set $ N^\star$, such that the claim holds on $\Omega\setminus N^\star$ for all $(t,y)$. 
	
	We next show \eqref{hjb} by adjusting the proof of \parencite[Theorem 5.3]{cossoOptimalControlPathdependent2022} to the random setting. Again by the \hyperlink{theo:roughfuncInt}{rough functional It\^o formula}, it holds for $s\in[t,T)$ on $\Omega\setminus N^\star$ that 
	\begin{align*}
		&v(s,Y^\phi(s),\Wtaunullb,\Wtau) - v(t,y,\Wtaunullb,\Wtau)\\
		&= \int_t^s Dv(\id,Y^\phi,\Wtaunullb,\Wtau) + H(\id,Y^\phi, \phi,\Wtaunullb,\Wtau)\dd\lambda + \int_t^s \la\nabla_\z v(\id,Y^\phi,\Wtaunullb,\Wtau),\dd\Wtau\ra.
	\end{align*}
	Inserting that into the \hyperlink{dpp}{dynamic programming principle}, it follows exactly as for \eqref{theo:regularvaluesolvesHJB:eq:3} that 
	\begin{align*}
		0 &= \esssup_{\phi\in \adm} \expE(v(s,Y^\phi(s),\Wtaunullb,\Wtau) - v(t,y^0, \Wtaunullb,\Wtau)\mid\F_{\tau(t)})\\
		&= \esssup_{\phi\in \adm} \int_t^s \expE(Dv(\id,Y^\phi,\Wtaunullb,\Wtau) + H(\id,Y^\phi,\phi,\Wtaunullb,\Wtau)\mid\F_{\tau(t)}) \dd\lambda,
	\end{align*}
	where the equalities hold in $L^0$.
	Fix now $\psi^0\in\R^m$ and set $\psi := \psi^0\one_{[t,T]}\in\adm$. Then for every $h>0$,
	\begin{align*}
		0 \ge \frac1h \int_t^{t+h} \expE(Dv(\id,Y^{\psi},\Wtaunullb,\Wtau) + H(\id,Y^\psi,\psi,\Wtaunullb,\Wtau)\mid\F_{\tau(t)}) \dd\lambda,
	\end{align*}
	with respect to \as-ordering in $L^0$. Using again the representative of the conditional distribution of $(Y^\psi,\psi,\Wtaunulltb,\Wtaut)$ and $(Y^\psi,\psi,\Wtaut)$ given $\F_{\tau(0)}$ from above and the transformed maps $\vb$, $\bar{H}$, we consider
	\begin{align*}
		(\omega, h) \mapsto \frac1h\int_t^{t+h} \int D\vb(\id,\y,\wb, \z) + \bar{H}(\id,\y,\varphi,\wb,\z)\cP{Y^\psi,\psi,\Wtaunulltb, \Wtaut}{\F_{\tau(t)}}(\dd(\y,\varphi,\wb,\z), \omega) \dd\lambda.
	\end{align*}
	Then this modification has continuous sample paths and for every $h>0$ the RHS is for $\prob$-\as\ $\omega\in\Omega$ non-positive, hence we find $ N^{t,y, \psi^0}\in\F$ with $\prob( N^{t,y,\psi^0}) = 0$ and such that on $\Omega\setminus N^{t,y,\psi^0}$ every sample path is non-positive. Letting $h\searrow0$, the fundamental theorem of calculus shows on $\Omega\setminus N^{t,y,\psi^0}$ that
	\begin{align*}
	 0 &\ge \int D\vb(t,\y(t),\wb,\z) + \bar{H}(t,\y,\varphi(t),\wb,\z)\cP{Y^\psi,\psi,\Wtaunulltb, \Wtaut}{\F_{\tau(t)}}(\dd(\y,\varphi,\wb,\z), \omega)\\
	 & = Dv(t,y,\Wtaunullb,\Wtau) + H(t,y,\psi^0,\Wtaunullb,\Wtau).
	\end{align*}
	Here we used Lemma \ref{lem:condDist} in the same manner as for \eqref{theo:regularvaluesolvesHJB:eq:4}. Moreover as for $N^\star$ in \eqref{theo:regularvaluesolvesHJB:eq:4}, we find a null set $ N^{\ge}$, such that the last inequality holds on $\Omega\setminus N^{\ge}$ for all $(t,y,\psi^0)$. Hence on $\Omega\setminus N^\ge$,
	\begin{align*}
		0 \ge Dv(t,y,\Wtaunullb,\Wtau) + \sup_{\varphi\in\R^m} H(t,y,\varphi,\Wtaunullb,\Wtau).
	\end{align*}
	For the reverse inequality let $\eps>0$ and $s_n\searrow t$. By Lemma \ref{lem:updirec} and the \hyperlink{dpp}{dynamic programming principle} we find $(\psi_{n,m})\subset\adm$ such that for every $n\in\N$,
	\begin{align*}
		\expE(v(s_n,Y^{\psi_{n,m}},\Wtaunullb,\Wtau)\mid\F_{\tau(t)}) \nearrow v(t,y,\Wtaunullb,\Wtau)
	\end{align*}
	as $m\to\infty$ in $L^0$. The  \hyperlink{theo:roughfuncInt}{rough functional It\^o formula} then yields that
	\begin{align*}
		 \int_t^{s_n} \expE(Dv(\id,Y^{\psi_{n,m}},\Wtaunullb,\Wtau) + H(\id,Y^{\psi_{n,m}},\psi_{n,m}, \Wtaunullb,\Wtau)\mid\F_{\tau(t)}) \dd\lambda \nearrow 0
	\end{align*}
		as $m\to\infty$ in $L^0$. Repeating arguments for our sample continuous modification, there exist for every $n\in\N$ a null set $ N(n)$, such that for $\omega\in\Omega\setminus N(n)$,
		\begin{align*}
			\int_t^{s_n} \int D\vb(\id,\y,\wb,\z) + \bar{H}(\id,\y,\varphi,\wb,\z)\cP{Y^{\psi_{n,m}},\psi_{n,m},\Wtaunulltb, \Wtaut}{\F_{\tau(t)}}(\dd(\y,\varphi,\wb,\z), \omega) \dd\lambda \nearrow 0.
		\end{align*}
		Setting $ N^{\le} := \bigcup_{n\in\N}  N(n)$ the monotone convergence as $m\to\infty$ holds on $\Omega\setminus N^{\le}$. Let $\omega\in\Omega\setminus N^{\le}$ and $\eps>0$. Then we find $M(n,\omega)\in\N$, such that for all $m\ge M(n,\omega)$,
		\begin{align}
			\begin{split}\label{theo:regularvaluesolvesHJB:eq:5}
			-\eps(s_n-t)&\le \int_t^{s_n} \int D\vb(\id,\y,\wb,\z) + \bar{H}(\id,\y,\varphi,\wb,\z)\cP{Y^{\psi_{n,m}},\psi_{n,m},\Wtaunulltb,\Wtaut}{\F_{\tau(t)}}(\dd(\y,\varphi,\wb,\z), \omega)\dd\lambda\\
			&= \int_t^{s_n} \int D\vb(u,y,\wb,\z) + \bar{H}(u,y,\varphi(u),\wb,\z)\ \cP{Y^{\psi_{n,m}},\psi_{n,m},\Wtaunulltb,\Wtaut}{\F_{\tau(t)}}(\dd(\y,\varphi,\wb,\z), \omega)\dd u\\ 
			&\quad + \rho(n,m)(\omega)
			\end{split}
		\end{align}
		where 
		\begin{align}
			\begin{split}\label{theo:regularvaluesolvesHJB:eq:6}
			\rho(n,m)(\omega) := \int_t^{s_n} \int D\vb(u,\y(u),\wb,\z) - &D\vb(u,y,\wb,\z)+ \bar{H}(u,\y(u),\varphi(u),\wb,\z) - \bar{H}(u,y,\varphi(u),\wb,\z)\\
			&\qquad \cP{Y^{\psi_{n,m}},\psi_{n,m},\Wtaunulltb,\Wtaut}{\F_{\tau(t)}}(\dd(\y,\varphi,\wb,\z), \omega).
			\end{split}
		\end{align}
		Assume for a moment that for all $\omega\in\Omega$, $\frac{1}{s_n-t} \sup_{m\ge M(n,\omega)}\rho(n,m)(\omega)\to 0$ as $n\to\infty$. By Lemma \ref{lem:qvCondExp} it holds that	
		\begin{align*}
		&\int_t^{s_n} \int D\vb(u,y,\wb,\z) + \sup_{\varphi\in \R^m}\bar{H}(u,y,\varphi,\wb,\z)\ \cP{Y^{\psi_{n,m}},\psi_{n,m},\Wtaunulltb,\Wtaut}{\F_{\tau(t)}}(\dd(\y,\varphi,\w,\z), \omega)\dd u\\ 
		& =\int_t^{s_n} \int D\vb(u,y,\Wtaunulltb(\omega),\z) + \sup_{\varphi\in \R^m}\bar{H}(u,y,\varphi,\Wtaunulltb(\omega),\z)\ \prob^{\Wtaut}(\dd\z) \dd u
		\end{align*}
		Together with \eqref{theo:regularvaluesolvesHJB:eq:5} this shows that 
		\begin{align*}
		-\eps &\le \frac{1}{s_n-t}\left[\int_t^{s_n} \int D\vb(u,y, \Wtaunulltb(\omega),\z) + \sup_{\varphi\in \R^m}\bar{H}(u,y,\varphi,\Wtaunulltb(\omega),\z)\ \prob^{\Wtaut}(\dd\z) \dd u + \rho(n,m)(\omega)\right].
		\end{align*}
		Then the RHS converges as $n\to\infty$ to
		\begin{align*}
		&\int D\vb(t,y,\Wtaunulltb(\omega),\z) + \sup_{\varphi\in \R^m}\bar{H}(t,y,\varphi,\Wtaunulltb(\omega),\z)\ \prob^{\Wtaut}(\dd\z)\\
		&\qquad = Dv(t,y,\Wtaunullb(\omega),\Wtau(\omega)) + \sup_{\varphi\in \R^m}H(t,y,\varphi,\Wtaunullb(\omega),\Wtau(\omega)).
		\end{align*}
		Letting $\eps\searrow0$, it follows that
		\begin{align*}
			0 \le  Dv(t,y,\Wtaunullb(\omega),\Wtau(\omega)) + \sup_{\varphi\in \R^m}H(t,y,\varphi,\Wtaunullb(\omega),\Wtau(\omega)).
		\end{align*}
		It is left to show $\frac{1}{s_n-t} \sup_{m\ge M(n,\omega)}\rho(n,m)(\omega)\to 0$. Since $\vb$ is the transformation of the \hyperlink{goodvaluefct}{good version}, its causal derivatives are Lipschitz continuous in $(y,\z)$ uniformly in $t\in[0,T]$ and $\wb\in \RP_{\tau(0)}$. Thus for every $u\in[t,s_n]$,
		\begin{align*}
			\vert D\vb(u,\y(u),\wb,\z) - D\vb(u,y,\wb,\z)\vert \lesssim \sup_{r\in[t,s_n]} \vert \y(r)-y\vert.
		\end{align*}
		Since for all $\omega\in\Omega$, $\cP{Y^{\psi_{n,m}},\psi_{n,m},\Wtaunulltb, \Wtaut}{\F_{\tau(t)}}$ is concentrated on the solution set $S^t\times C^\alpha$ by Lemma \ref{lem:qvCondExp} part $i)$, we may assume that $\y$ solves \eqref{solutionsset} on $[t,\tau(t)]$ with $\y(t)=y$. Thus for $n$ large enough such that $s_n\in[t,\tau(t)]$, the a priori bounds on $\y$ from Theorem \ref{theo:sol2cSDE} part $i)$, 
		\begin{align*}
			\sup_{r\in[t,s_n]} \vert \y(r)-y\vert \le \vert \y\vert_{\alpha,[t,\tau(t)]}\vert s_n-t\vert^\alpha \lesssim_{\vert b\vert_\infty, \vert f\vert_{C_b^2}} (1+\vert\wb\vert_{\alpha,[t,\tau(t)]})\vee(1+\vert\wb\vert_{\alpha,[t,\tau(t)]})^{1/\alpha})\vert s_n-t\vert^\alpha.
		\end{align*}
		Since also $b$ and $f$ are Lipschitz continuous in $\y(\cdot)\in\R^n$ we get a corresponding estimate for $\bar{H}(u,\y(u),\varphi,\wb,\z) - \bar{H}(u,\y,\varphi,\wb,\z)$. Recalling \eqref{theo:regularvaluesolvesHJB:eq:6}, it follows from Lemma \ref{lem:condDist} part $iii)$ that 
		\begin{align*}
			 \frac{1}{s_n-t} \sup_{m\ge M(n,\omega)} \vert\rho(n,m)(\omega)\vert &\lesssim \sup_{m\ge M(n,\omega)} \int_t^{s_n} \int C(\wb)\ \cP{Y^{\psi_{n,m}},\psi_{n,m},\Wtaunulltb, \Wtaut}{\F_{\tau(t)}}(\dd(\y,\varphi,\wb,\z), \omega) \dd\lambda\\
			 &= C(\Wtaunulltb(\omega)) \vert s_n-t\vert^\alpha \to 0, 
		\end{align*}
		as $n\to\infty$. Note that $C(\Wtaunulltb(\omega))$ depends on powers of $\vert\Wtaunullt\vert_\alpha$, which is finite (possibly adjusting $ N^\le$).
		
		Finally since the \hyperlink{goodvaluefct}{good version} is a modification of the value function, 
		\begin{align*}
			v(T,y,\Wtaunullb,\Wtau)\in\esssup_{\phi\in\adm} \expE(g(Y^{T,y,\phi}(T))\mid \F_{\tau(T)})=g(y).\\
		\end{align*}
		We summarize that for all $t\in[0,T)$, $y\in\R^n$ it holds $\prob$-\as\ that 
	\begin{align*}
			\begin{cases}
				0 &= Dv(t,y,\Wtaunullb,\Wtau) + \sup_{\varphi\in\R^m}\ \la\nabla_y v(t,y,\Wtaunullb,\wtau), b(t,y,\varphi)\ra\\ 
				&\qquad\qquad+ \frac12 \tr( f(y)^T\nabla^2_y v(t,y,\Wtaunullb,\Wtau) f(y)) + \frac12 \tr(\nabla_\z^2 v(t,y,\Wtaunullb, \Wtau))\tau'(t),\\
				0 &= \la\nabla_y v(t,y,\Wtaunullb,\Wtau), f(y)\ra,\\
				g(y) &= v(T, y, \Wtaunullb, \Wtau).
			\end{cases}
	\end{align*}
	Define $u\colon [0,T]\times \R^n\times C([0,\tau(0)],\R^d)\times D([0,T],\R^d)\to[0,\infty)$ by $u(t,y,\w,\z):= v(t,y,\lift(\w),\z)$, where lift is given in \eqref{defi:lift}. Then $u$ inherits the regularity properties of $v$ and is measurable (with $\borel_{\vert\cdot\vert_\infty}(C([0,T],\R^d))\vee\negl$), so that the assertion follows.
\end{proof}

\section{Examples}\label{sec:exmp}
\begin{exmp}[Optimal Investment for a Frontrunner with Price Impact]
We apply the results to \parencite{insiderbank, bankOptimalInvestmentNoisy2024}. Let us briefly describe their setting and then introduce the corresponding conditional optimal control problem.

Let $W$ be a standard Brownian motion on some probability space $(\Omega, \F, \prob)$ satisfying the usual conditions. Assume that the unimpacted price development of some risky asset is given by $S=W$ (i.e.\ in Bachelier dynamics). On top of the information flow $\F^W$ of this stock price, a ``frontrunning'' investor with investment horizon $T>0$ can peek $\Delta>0$ time units into the future and thus bases his investment decision at time $t \in [0,T]$ on 
\[ \G_t = \F_t^W \vee \sigma(W(u),\ u\in[t, (t+\Delta) \wedge T]), \quad t\in[0,T]. \]
Let $\Phi^0\in\R$ be the number shares of the risky asset held by some investor at the start of trading. The admissible trading strategies of the investor in the risky asset are then processes $\Phi\colon [0,T] \times \Omega \to \R$ such that there exists some $\G$-optional $\phi\colon [0,T] \times \Omega \to \R$ with $\phi\in\Lt([0,T], \lambda)$ $\prob$-a.s., satisfying
\begin{equation}\label{adm}
	\Phi(t) = \Phi^0 + \int_0^t \phi(s) \dd s, \quad t\in[0,T].
\end{equation}
To construct a suitable value function of the portfolio one defines for $\Lambda > 0$ the temporarily impacted price process $S^\phi = S + \frac\Lambda 2 \phi$, writes $\Phi = \Phi(\Phi^0, \phi)$ and sets 
\begin{equation}\label{portfoliovalue}
	X^{\Phi^0, \phi}(t) = \Phi(t) S(t) - \Phi(0) S(0) - \int_0^t S^\phi(s) \dd \Phi(s).
\end{equation}
In \parencite[p.10]{bankOptimalInvestmentNoisy2024} they describe the following transformation to formulate the maximization of exponential utility under the new information flow, i.e.\ of maximizing$-\expE(\exp(-\alpha X(T)^\phi)\vert \G_0)$.
Let $W^{\Delta}(t) = W((t + \Delta) \wedge T) - W(\Delta)$ and 
\[\G_t = \sigma(W(u), u\in[0,\Delta]) \vee \sigma (W^\Delta(u), u\in [0, t \wedge T - \Delta]) \vee  \mathcal{N}, \]
where $\mathcal{N}$ denotes the collection of $\prob$ null sets and 	$S^W(t) = W(t\wedge \Delta) + W^\Delta((t - \Delta) \vee 0)$.
Then $W^\Delta$ is a Brownian motion stopped at $T-\Delta$ that is independent of $\G_0$. 
Using the duality theorem \parencite[Proposition A.2]{insiderbank}, they calculate for a fixed $\w\in C([0, \Delta])$ the optimal strategy for
\begin{align}\label{condcontrolproblemexsol}
	v(0,0,\Phi^0, \w) = \max_{\phi\in\adm^{\Delta}} -\expE \exp(- X^{\Phi^0, \phi, \w}(T)),
\end{align}
where the set of admissible controls $\adm^{\Delta}$ consists of all $\G$-optional processes with $\phi\in\mathrm{L}^2([0,T], \lambda)$  $\prob$-\as\ and $X^{\Phi^0, \phi, \w}(T)$ is the terminal portfolio value \eqref{portfoliovalue} for $S = S^\w$. They also compute the expected problem value $\expE(v(0,0,\Phi^0,\Wtaunull))$, but not the quantity $v(0,0,\Phi^0,\w)$ for a given initial price path segment $\w$. Revisiting the calculations in \parencite{bankOptimalInvestmentNoisy2024}, we find that
\begin{align}
	\begin{split}\label{exmp:Delta:vnull}
	v(0,0,\Phi^0, \w) &= -\exp\Bigg(-  \Phi^0(\w(\Delta)- \w(0)) - \frac{1}{2\Lambda} \int_0^\Delta (\w(\Delta) - \w(s))^2 \dd s\\
	&\quad \quad
	+ \frac12 \frac{\lb \Phi^0 + \frac{1}{\Lambda} \int_0^\Delta \w(\Delta) - \w(s) \dd s\rb^2}{\frac{\Delta}{\Lambda} + \frac{1}{\sqrt{\Lambda}} \coth\lb \frac{1}{\sqrt{\Lambda}} (T-\Delta) \rb} - \frac{\Delta}{2\Lambda} \int_{\Delta}^T \frac{1}{1 + \frac{\Delta}{\sqrt{\Lambda}} \tanh\lb \frac{1}{\sqrt{\Lambda}} (T-s) \rb} \dd s\Bigg).
	\end{split}
\end{align}

Let us show how these findings fit into our setting and formulate the problem as a conditional optimal control problem.
The driving coefficients are simply set to be
\begin{align*}
	b(s, y, \varphi) := \lb - \frac{\Lambda}{2} \varphi^2, \varphi, 0\rb, \quad f(s, y, \varphi) := (y_2,0,1) 
\end{align*}
and the initial condition is $y^t := (x^t, \Phi^t, s^t)\in\R^3$. Then the solution $Y=(Y^1,Y^2,Y^3)$ to 
\begin{align*}
	\dd Y(s) &= b(s, Y, \phi(s)) \dd s + f(Y(s)) \dd \Wb(s),\quad s\in (t, T],\\
	Y(t) &= y.
\end{align*}
can be read off: For $s\in[t,T]$, $Y^2(s) = \Phi^t + \int_t^s \phi(r) \dd r=\Phi^{\Phi^t,\phi}(s)$ describes the number of shares held at time $s$,
$Y^3(s) = s^t + W(s)-W(t)=S^{t,s^t}(s)$ gives the stock price dynamics and the investor's net worth when marking to market the shares held is
\begin{align*}
	Y^1(s) &= x - \frac{\Lambda}{2} \int_t^s \phi(r) \dd r + \int_t^s Y^2(r) \dd \Wb(r).
\end{align*}
Since $\Phi$ has finite variation, the rough integral in $Y^1$ is simply a Young integral and integration by parts shows that $Y^1$ describes the portfolio value $ X^{t,s^t,\Phi^t}$ from \eqref{portfoliovalue} started at $t$. Hence $Y=( X^{t,s^t,\Phi^t}, \Phi^{\Phi^t,\phi}, S^{t,s^t})$. Further set the terminal reward to be $g(y) := -\exp(-x)$ and the time change to 
\begin{align*}
	\tau(t):= (t+\Delta)\wedge T.	
\end{align*}
Then value function $v(t,x,\Phi,s)$ is characterized by a deterministic function $u(t,x,\Phi,s, \wb, \z)$ which should satisfy the HJB equation
\begin{align}\label{hjb:expm1}
	\begin{split}
		\begin{cases}
			0 &= Du + \sup_{\varphi\in\R} \{-\frac{\Lambda}{2}\varphi^2\partial_x u + \varphi\partial_\Phi u\} + \frac12\lb \Phi^2\partial_x^2 u + 2\Phi\partial_{xs}^2u + \partial_s^2u + \nabla_\z^2 u\rb,\\
            0 &= \Phi\partial_x u + \partial_s u,\\
			g(x) &= u(T,x,\cdot).
		\end{cases}
	\end{split}
\end{align}
We construct a candidate for $u$. Thereby $h$ denotes a help functions that may vary from line to line and we start with general arguments $(t,x,\Phi,s,\wb,\z)\in [0,T]\times \R^3\times C([0,\tau(0),\R^d])\times C([0,T],\R^d)$: The terminal condition suggests 
\begin{align*}
	u(t,  x,\Phi, s, \w, \z) = -\exp(-(x + h(t,\Phi, s, \w, \z))) \text{ with } h(T,\Phi,s,\w,\z)=0.
\end{align*}
The second equation in \eqref{hjb:expm1} yields
\begin{align*}
	0 = -u \Phi - u\partial_sh,\ \text{ i.e. }\ \partial_s h = -\Phi.
\end{align*}
Thus 
\begin{align*}
	u(t,x,\Phi,s,\w,\z)= -\exp(-(x-\Phi s+h(t,\Phi,\w,\z))).
\end{align*}
We plug this Ansatz into \eqref{hjb:expm1} and calculate 
\begin{align*}
	&Du = -u Dh,\ \partial_xu = -u,\ \partial_\Phi u = -u(-s + \partial_\Phi h),\ \partial_su = -u \Phi,\ \nabla_\z u = -u\nabla_\z h,\\ 
	&\partial_x^2u = u,\ \partial_{xs}^2u = u\Phi,\ \partial_s^2 u = u\Phi^2,\ \partial_z^2u = u((\partial_z h)^2 - \nabla_\z^2 h).
\end{align*}
The supremum in \eqref{hjb:expm1} is realized by $\varphi^* = \frac{1}{\Lambda}\frac{\partial_\Phi u}{\partial_xu} = \frac{1}{\Lambda}(-s + \partial_\Phi h)$ and so the HJB equation reduces to a functional partial differential equation for $h(t,\Phi,\w_\Delta,\w^\Delta)$: 
\begin{align}\label{exmp1:DEh}
	0 = Dh + \frac{1}{2\Lambda}(-s +\partial_\Phi h)^2 + \frac12((\partial_z h)^2 - \nabla_\z^2 h)).
\end{align}
Recalling \eqref{exmp:Delta:vnull}, we make the following educated guess:
\begin{align*}
	h(t,\Phi,\w,\z) &= \Phi(\w(\Delta)+\z(t))\\
	&\quad + \frac{1}{2\Lambda}\left( \int_{t\wedge\Delta}^\Delta (\w(\Delta) +\z(t)-\w(u))^2 \dd u + \int_{(t-\Delta)\vee 0}^t (\z(t)- \z(u))^2 \dd u\right)\\
	&\quad - \frac12\Upsilon(t)C(t, \Phi, \w, \z)^2+ \frac{\Delta}{2\Lambda}\Omega(t)
\end{align*}
with
\begin{align*}
	C(t, \Phi, \w, \z) =  \Phi + \frac{1}{\Lambda} \left(\int_{t\wedge\Delta}^\Delta (\w(\Delta) +\z(t)-\w(u)) \dd u + \int_{(t-\Delta)\vee 0}^t (\z(t) - \z(u)) \dd u\right)
\end{align*}
and
\begin{align}\label{alpha}
	\begin{split}
		\Upsilon(t) = \frac{\frac{\Lambda}{\Delta}}{1+ \frac{\sqrt{\Lambda}}{\Delta} \coth\lb \frac{1}{\sqrt{\Lambda}} (T-((t+\Delta) \wedge T)) \rb},\quad		\Omega(t) = \int_{(t+\Delta) \wedge T}^T \frac{1}{1 + \frac{\Delta}{\sqrt{\Lambda}}\tanh\lb \frac{1}{\sqrt{\Lambda}}(T-u)\rb} \dd u.
	\end{split}
\end{align}
Then the causal time derivative of $h$ is
	\begin{align*} 
		D h(t,\Phi,\w,\z) &= - \frac{1}{2\Lambda} (\z(t) - \w(\Delta) - \w(t\wedge\Delta) -\z(t-\Delta \vee 0))^2 -  \frac{1}{2} \Upsilon'(t)C(t, \Phi, \w, \z)^2 \\
		&\quad\quad +\frac{1}{\Lambda} \Upsilon(t) C(t, \Phi, \w, \z) (\z(t) - \w(\Delta) - \w(t\wedge\Delta) -\z(t-\Delta \vee 0)) + \frac{\Delta}{2\Lambda}\Omega'(t)
	\end{align*}
	and its causal space derivatives are given by
	\begin{align*}
		\partial_\Phi h(t,\Phi,\w,\z) &= \w(\Delta)+\z(t) - \Upsilon(t)C(t, \Phi, \w, \z),\\ 
		\nabla_\z h(t,\Phi,\w,\z) &= -\lb \frac{\Delta}{\Lambda}\Upsilon(t) - 1 \rb C(t, \Phi, \w, \z),\\ 
		\nabla_\z^2 h(t,\Phi,\w,\z) &= -\frac{\Delta}{\Lambda} \lb \frac{\Delta}{\Lambda}\Upsilon(t) - 1 \rb.
	\end{align*}
This yields the optimal trading speed also found in \parencite[Theorem 2.2]{insiderbank} given via the feedback relation:
\begin{align*}
	\begin{split}
		&\phi^*(t) = \frac{1}{\Lambda}(-W(t) + \partial_\Phi h(t, \Phi(t), W_{\Delta}, W^\Delta))\\
		&= \frac{1}{\Lambda}\lb W(t+\Delta) - W(t) - \Upsilon(t)\lb \Phi(t) + \frac{1}{\Lambda} \int_t^{(t+\Delta)\wedge T} W((t+\Delta)\wedge T) - W(u) \dd u \rb\rb.
	\end{split}
\end{align*}
Finally we show that for $s=W(t)$, $h(t, \Phi, W_\Delta, W^\Delta)$ solves \eqref{exmp1:DEh}, where $W_\Delta$ is the restriction of $W$ to $[0,\Delta]$ and $W^\Delta(r) := W(r+\Delta\wedge T)-W(\Delta)$. Note that 
\begin{align*}
	W^\Delta(t) - W_\Delta(\Delta) - W_\Delta(t\wedge\Delta) -W^\Delta(t-\Delta \vee 0) = W(t+\Delta)-W(t)=:\delta^\Delta W(t).
\end{align*}
Then 
\begin{align}
	&Dh + \frac{1}{2\Lambda}(-s +\partial_\Phi h)^2 + \frac12((\partial_z h)^2 - \nabla_\z^2 h))\notag\\
	\begin{split}\label{exmp1:eq:2}
	&= - \frac{1}{2\Lambda} (\delta^\Delta W)^2 -  \frac{1}{2} \Upsilon' C^2  +\frac{1}{\Lambda} \Upsilon C\delta^\Delta W + \frac{\Delta}{2\Lambda}\Omega' + \frac{1}{2\Lambda}(\delta^\Delta W - \Upsilon C)^2\\
	&\quad + \frac12 \lb  \frac{\Delta}{\Lambda}\Upsilon - 1 \rb^2 C^2 +\frac12 \frac{\Delta}{\Lambda} \lb \frac{\Delta}{\Lambda}\Upsilon - 1 \rb.
	\end{split}
\end{align}
Noting that $\frac{\Delta}{\sqrt{\Lambda}}\tanh\lb \frac{1}{\sqrt{\Lambda}}(T-u)\rb = 1 - \frac{\Delta}{\Lambda} \Upsilon(u)$, we have the relation 
\begin{align*}
	\Omega(r)  = T - (r + \Delta)\wedge T - \frac{\Delta}{\Lambda} \int_{r\wedge T}^ {T-\Delta} \Upsilon(u) \dd u
\end{align*}
and one checks that
\begin{align*}
	\Upsilon'(r) = - \lb\frac{\Delta}{\Lambda} \Upsilon(r) - 1 \rb^2 + \frac{1}{\Lambda} \Upsilon(r)^2, \quad \Omega'(r) = - \lb\frac{\Delta}{\Lambda} \Upsilon(r) - 1\rb.
\end{align*}
This implies that \eqref{exmp1:eq:2} equals zero. Consequently $u(t,x,\Phi,s,\w,\z)= -\exp(-(x-\Phi s+h(t,\Phi,\w,\z)))$ solves \eqref{hjb:expm1}.
\end{exmp}

\begin{exmp}\label{exmp:2} We reconsider the (counter)example from \parencite[Appendix]{buckdahnPathwiseStochasticControl2007}, see also \parencite{frizControlledRoughSDEs2024} for a multidimensional version. 
Assume the dynamics
\begin{align*}
	\dd Y(s) = \begin{pmatrix}\phi(s)\\ 0\end{pmatrix}\dd s + \begin{pmatrix} Y^1(s)\\ 1 \end{pmatrix} \dd W(s)
\end{align*}	
controlled by $\phi$ which we restrict to values in $U=[0,1]$. Write $Y=(X,W)$ and for the initial conditions $y =(x,w)\in\R^2$. For $\tau(t)\equiv T$, i.e. full knowledge of $W$ from the start and terminal cost $\vert x -1\vert$ we consider the problem
\begin{align*}
	\underset{\substack{\phi\colon [0,T]\times\Omega\to [0,1],\\ \phi\in\A}} {\esssinf} \expE(\vert Y^1(T) -1\vert\mid \F_T).
\end{align*} 
 Since $\tau(t)\equiv T$, $W_{\tau(0)} = W$ and $W^\tau = 0$. Therefore a deterministic good version $u$ of the value function $v(t,x,w)$ depends only on $\w\in C([0,T],\R^d)$ and we drop the argument $\z$. Thus for $\prob^W$-\as\ $\w\in C([0,T],\R^d)$, $u(t,x,\w(t),\w)$ should solve
\begin{align*}
	\begin{cases}
	0 &= Du + \inf_{\varphi\in[0,1]} \varphi\partial_x u\\
	0 &= x\partial_x u + \partial_w u,\\
	\vert x-1\vert &= u(T,x,\cdot).
	\end{cases}
\end{align*}
We next adjust the optimal control and the value function for initial datum $(0,0)$ found in \parencite[2255]{buckdahnPathwiseStochasticControl2007}. The solution to the controlled differential equation started in $(t,x)$ is given by \parencite[Lemma 3.2]{buckdahnPathwiseStochasticControl2007} as
\begin{align}\label{exmp:2:eq:1}
	X^{t,x,\phi}(s) =x\exp(W(s)-W(t))+\int_t^s\exp(W(s)-W(r))\phi(r)\dd r. 
\end{align}
On $\{1 \leq x\exp(W(T)-W(t))+\int_t^T \exp(W(T)-W(r))\dd r\}$, the control
\begin{align*}
	\phi^*(s)= \frac{1 -x\exp(W(T)-W(t))}{\int_t^T \exp(W(T)-W(r))\dd r}
\end{align*}
yields $X^{t,x,\phi^*}(T)=1$, so the value of the problem is zero.
On the complement the terminal condition and the explicit solution \eqref{exmp:2:eq:1} suggest that 
\begin{align*}
	u(t,x,w,\w)=\vert xh(t,w,\w)+g(t,w,\w)-1 \vert
\end{align*}
for auxiliary functions $h,g$ with $h(T,w(T),\w)=1$, $g(T,w(T),\w)=0$. Let us consider a state where
\begin{align*}
	u(t,x,w,\w) = 1-xh(t,w,\w)-g(t,w,\w).
\end{align*}
The second equation in \eqref{hjb} reads
\begin{align*}
	0 = -x h(t,w,\w)- x\partial_w h(t,w,\w)-\partial_w g(t,w,\w), 
\end{align*}
and is solved by $h(t,w,\w)=c(\w)\exp(-w)$ and any $g$ with $\partial_wg=0$. The terminal condition $h(T,w(T),\w)=1$ implies that $c(\w)=\exp(w(T))$. The first partial differential equation is then 
\begin{align*}
	0 = Dg(t,\w) + \inf_{\varphi\in[0,1]} \varphi(-\exp(W(T)-w)) = Dg(t,\w) -\exp(W(T)-w)
\end{align*}
for optimal $\varphi^*=1$. Recalling Examples \ref{exmp:funcdiff} together with $g(T,\w)=0$, it follows in the case $w=\w(t)$, $g(t,\w) =\int_t^T \exp(\w(T)-\w(r))\dd r$. Consequently 
\begin{align*}
	u(t,x,w,\w)= 1-x\exp(w(T)-w)-\int_t^T\exp(\w(T)-\w(r))\dd r. 
\end{align*}
Together
\begin{align*}
	u(t,x,W(t),W) = \left[1-\lb x\exp(W(T)-W(t)) + \int_t^T\exp(W(T)-W(r))\dd r\rb\right]^+,
\end{align*}
which is consistent with the literature that deals with the case $(t,x)=(0,0)$.
\end{exmp}

\begin{rem}
    The above example allows us to illustrate that the transport part of our HJB equation is indispensable for our verification argument: Making the same Ansatz, we find that $h\equiv 1$ and $g(t,w,\w)=T-t$ yield that $u(t,x,w,\w)=1-x-(T-t)$ only solves the first equation,
    \begin{align*}
        0&=Du(t,x,w,\w) +\inf_{\varphi\in[0,1]} \varphi\partial_xu(t,x,w,\w)= -Dg(t,w,\w)-1=0,
    \end{align*}
    but yields the wrong value $u(0,0,W(t),W)=1-T\neq 1- \int_t^T \exp(W(T)-W(r))\dd r$ for the optimization problem.
\end{rem}

\begin{exmp}\label{exmp:3}
	Let us consider an insider trading problem from \parencite[Example 1.1, 3.17]{allanPathwiseStochasticControl2019}, made rigorous in \parencite[6.3]{frizControlledRoughSDEs2024}. Let $B$ and $W$ be independent Brownian motions, $\sigma_0>0$ and assume the dynamics 
	\begin{align*}
		\dd Y(s)= \begin{pmatrix} 0\\ \phi(s)\\ 0\end{pmatrix}\dd s + \begin{pmatrix} Y^2(s)\\ 0\\ 0\end{pmatrix} \dd B(s) + \begin{pmatrix} Y^2(s)\\ 0\\ 1 \end{pmatrix} \dd \Wb(s)
	\end{align*}
	controlled by $\phi$. The goal is to maximize the expected terminal wealth under transaction costs $\int_0^T\eps\phi(s)^2 \dd s$ for some $\eps>0$ and with full knowledge of $W$. That is we have $\tau(t) \equiv T$ and the problem value at time $t=0$ is
	\begin{align*}
		\esssup_{\phi\in\adm} \expE\lb Y^1(T) - \int_0^T \eps \phi(s)^2\dd s\mid \F_T^W\rb,
	\end{align*} 
	where the set of admissible controls $\adm$ consists of all $\phi\colon [0,T]\times\Omega\to\R$ that are progressively measurable with respect to $\G_t =\F_T^W\vee\F^B_t$. We write again $(x,\Phi,w)$ for the initial condition of the portfolio value, number of shares and to track the current value of the Brownian motion $W$.
	From the literature we know that
	\begin{align*}
		v(t,x,\Phi)=x+(W(T)-W(t))\Phi + \frac{3}{4\eps} \int_t^T(W(T)-W(s))^2\dd s 
	\end{align*}
	is a pathwise solution to the rough HJB equation
	\begin{align*}
		-(v(T, x, \Phi)-v(t,x,\Phi)) = x + \int_t^T \inf_{\varphi} \{\varphi \partial_\Phi v + \frac12\sigma_0^2\Phi^2\partial^2_x v + \eps\varphi^2\} \dd \lambda + \int_t^T\Phi\partial_y v\dd \Wb^{\mathrm{strat}},
	\end{align*}
	where $\Wb^{\mathrm{strat}}$ is the lift using Stratonovich integration, i.e.\ to a geometric rough path.
	We add $\w\in C([0,T],\R^d)$ as a parameter (and no functional dependence since $\tau(t)=T$) and consider
	\begin{align*}
		u(t,x,\Phi,w,\w):= x+(\w(T)-w)\Phi + \frac{3}{4\eps} \int_t^T(\w(T)-\w(s))^2\dd s.
	\end{align*}
	Then for fixed $x,\Phi\in\R^2$, $\w\in C([0,T],\R^d)$ it holds that $u(\cdot, x,\Phi,\cdot, \w)\in C^{1,3}$. Hence by the rough It\^o formula, cf. \parencite[Proposition 5.8]{RoughBook},
	\begin{align*}
		&v(T, x, \Phi)-v(t,x,\Phi)
		=u(T,x,\Phi, W(T), W) -u(t, x,\Phi,W(t),W)\\
		&= \int_t^T \partial_su(s,x,\Phi,W(s),W) \dd s + \int_t^T \partial_wu(s,x,\Phi,W(s),W) \dd\Wb^{\mathrm{strat}}(s).
	\end{align*}
	Note that $\partial_xu=\partial_x v=1$, $\partial_x^2u = \partial_x^2v =0$, $\partial_\Phi u(t,x,\Phi,\w(t), \w)=\partial_\Phi v(t,x,\Phi)=-(W(T)-W(s))$ and 
	\begin{align*}
		\partial_tu(t,x,\Phi,W(t), W)= -\frac{3}{4\eps}(W(T)-W(t))^2,\quad \partial_wu(t,x,\Phi,W(t), W) = -\Phi.
	\end{align*}
	It follows that $u$ solves for $\prob^W$-\as\ $\w\in C([0,T],\R^d)$,
	\begin{align*}
		\begin{cases}
		0&=\partial_tu(t,x,\Phi,\w(t), \w) + \inf_{\varphi} \{\varphi \partial_\Phi u(t,x,\Phi,\w(t), \w) + \frac12\sigma_0^2\Phi^2\partial^2_x u(t,x,\Phi,\w(t), \w) + \eps\varphi^2\},\\
		0&= \Phi\partial_x u(t,x,\Phi,\w(t), \w)+ \partial_w u(t,x,\Phi,\w(t), \w),\\
		x&=u(T,x,\Phi,\w(T), \w).
	\end{cases} 
	\end{align*}
	We point out that we could also write causal derivatives here. Indeed, we have $Du = \partial_tu$ here and similarly $\partial_wu$ coincides with the causal space derivative of $u$ when considering $u$ as a pathdependent functional that depends on its fourth argument $\w$  only on through the current value $\w(t)$. Similar considerations apply for derivatives with respect to $x,\Phi$. Recall also that $\tau(t)\equiv T$, so that the fifth argument is in fact the parameter $\wtaunull= \w$ and $\wtau\equiv 0$, hence there is no other causal space derivative. 
\end{exmp}

\begin{rem}
	The last example suggests that the results extend to controlled rough stochastic differential equations introduced in \parencite{frizRoughStochasticDifferential2024, frizControlledRoughSDEs2024} in the following way. Consider a controlled rough stochastic differential equation,
	\begin{align*}
		\dd Y(t)= b(t,Y^{\Xb}(t),\phi(t))\dd t + \sigma(t,Y^{\Xb}(t),\phi(t))\dd B(t) + (f, f')(t, Y^{\Xb}(t))\dd \Xb(s),
	\end{align*}
	a randomization of $X\rightsquigarrow W(\omega)$ to a Brownian motion $W$ that is independent of the Brownian motion $B$ and an information flow modeled by a time-change $\tau$. For terminal reward function $g\colon \R^n\to [0,\infty)$, the optimal control problem is
	\begin{align*}
		\esssup_{\phi\in\adm} \expE\lb g(Y^{t,y,\phi}(T))\mid\F^W_{\tau(t)}\rb
	\end{align*}
	where the set of admissible controls $\adm$ consists of all $\phi\colon [0,T]\times\Omega\to\R$ that are progressively measurable with respect to $\F_\tau =(\F_{\tau(t)}^W\vee\F^B_t)$. If the corresponding value function has a regular representative $u(t,y,\Wtaunullb,W(t),\Wtau)$, we expect it to be characterized by the causal HJB: For every $(t,y)\in[0,T]\times\R^n$ and $\prob^W$-\as\ $\w\in C([0,T],\R^d)$,
	\begin{align*}
		\begin{cases}
			0 &=  Du(t, y,\wtaunull,\w(t),\wtau) + \sup_{\varphi\in\R} \la\nabla_y u(t, y,\wtaunull,\w(t),\wtau), b(t,y,\varphi)\ra +\\ 
			&\qquad+ \frac12 [\tr(\sigma(t,y,\varphi)^T\nabla^2_y u(t,y,\wtaunull,\w(t),\wtau)\sigma(t,y,\varphi))\\
            &\qquad\ + \tr(f(t,y)^T\nabla^2_y u(t,y,\wtaunull,\w(t),\wtau) f(t,y)) + \tr(\nabla_\z^2 u(t,y,\wtaunull,\w(t), \wtau))\tau'(t)],\\
            0 &= \la\nabla_y u(t,y,\wtaunull,\w(t),\wtau), f(t,y)\ra + \nabla_w u(t,y,\wtaunull,\w(t),\wtau),\\
			g(y) &= u(T,y,\wtaunull,\w(t),\wtau).
		\end{cases}
    \end{align*}
\end{rem}

\appendix 

\section{}\label{sec:appendix}
\begin{proof}[Proof of Theorem \ref{theo:funcIto4SolXW}] The proof basically adapts \parencite[Theorem 7.7]{RoughBook}. We give it only because we later rely on explicit estimates.\\
Fix $\alpha$ as in the \hyperlink{theo:roughfuncInt}{rough functional It\^{o} formula}.
Since $Y$ solves \eqref{cSDE} $\prob$-\as\ pathwise and $\tau$ is continuously differentiable, the tuple $(Y, \Wtau)$ is again $\prob$-\as\ $\alpha$-H\"older continuous. By Corollary \ref{cor:2varXW} there exists a sequence of partitions $\partition$ with $\vert\partition\vert\to0$ such that $(Y,\Wtau)$ has $\prob$-\as\ finite quadratic variation along $\partition$. Hence by \hyperlink{theo:roughfuncInt}{rough functional It\^{o} formula} it holds $\prob$-\as\ that 
\begin{align*}
		F(T, Y, \Wtau) &= F(0, Y, \Wtau) + \int_0^T DF(\id, Y, \Wtau) \dd \lambda + \int_0^T \nabla F(\id, Y, \Wtau) \dd (\Yb, \Wb^{\circ\tau})\\
		&\quad + \frac12 \int_0^T \la \nabla^2 F(\id, Y, \Wtau), \dd[(Y, \Wtau)]\ra,
\end{align*}
where with some abuse of notation $\prob$-\as,
\begin{align*}
	&\int_0^T \nabla F(\id, Y, \Wtau) \dd (\Yb, \Wb^{\circ\tau})\\
	&= \lim_{\vert \partition\vert \to 0} \sum_{[s,t]\in\partition} \langle\nabla F(s, Y, \Wtau),\begin{pmatrix} Y(t)-Y(s)\\ \Wtau(t) - \Wtau(s)\end{pmatrix}\rangle\\
	&\qquad\qquad\qquad + \frac12 \langle\nabla^2 F(s, Y, \Wtau),\begin{pmatrix} Y(t)-Y(s)\\ \Wtau(t) - \Wtau(s)\end{pmatrix}^{\otimes 2} - \begin{pmatrix} \int_s^t f(Y)^{\otimes 2}\dd\lambda& 0\\0& \tau(t)-\tau(s)\end{pmatrix}\rangle.
\end{align*}
(Not necessarily along $\partition$ for $[\cdot]$.) The terms in the approximating sequence decompose to
\begin{align}
	&\langle\nabla F(s, Y, \Wtau),\begin{pmatrix} Y(t)-Y(s)\notag\\ \Wtau(t) - \Wtau(s)\end{pmatrix}\rangle\\
	&\quad + \frac12 \langle\nabla^2 F(s, Y, \Wtau),\begin{pmatrix} Y(t)-Y(s)\notag\\ \Wtau(t) - \Wtau(s)\end{pmatrix}^{\otimes 2} - \begin{pmatrix} \int_s^t f(Y)^{\otimes 2}\dd\lambda& 0\\0& \tau(t)-\tau(s)\end{pmatrix}\rangle\\
	&= \langle\nabla_y F(s, Y, \Wtau), \int_s^t b(\id, Y) \dd\lambda \ra\label{theo:funcIto4SolXW:eq:1}\\
	&\quad + \langle\nabla_y F(s, Y, \Wtau), \int_s^t f(Y) \dd\Wb \rangle + \frac12 \la \nabla^2_y F(s, Y,\Wtau), \lb \int_s^t f(Y) \dd\Wb \rb^{\otimes 2} - \int_s^t f(Y)^{\otimes 2} \dd\lambda\ra\label{theo:funcIto4SolXW:eq:2}\\
	&\quad + \frac12 \la \nabla^2_y F(s, Y,\Wtau), 2\int_s^t b(\id, Y) \dd\lambda \otimes \int_s^t f(Y) \dd\Wb + \lb \int_s^t b(\id, Y) \dd\lambda \rb^{\otimes 2}\ra\label{theo:funcIto4SolXW:eq:3}\\
	&\quad + \frac12 \la \nabla_{y\z} F(s, Y, \Wtau), (Y(t)-Y(s))\otimes(\Wtau(t) - \Wtau(s))) \ra\label{theo:funcIto4SolXW:eq:4}\\
	&\quad + \la\nabla_\z F(s, Y, \Wtau),(\Wtau(t)- \Wtau(s))\ra + \frac12 \la \nabla_\z^2 F(s, Y, \Wtau)\lb (\Wtau(t)-\Wtau(s))^{\otimes 2} - (\tau(t)-\tau(s))\diag(\one_d)\rb\ra.\label{theo:funcIto4SolXW:eq:5}
\end{align}
Summing over \eqref{theo:funcIto4SolXW:eq:1} and letting $\vert\partition\vert\to 0$ yields $\int_0^T \la \nabla_y F(\id, Y, \Wtau), b(\id, Y)\ra \dd\lambda$. In the same sense \eqref{theo:funcIto4SolXW:eq:3} converges to zero, since the causal space derivatives of $F$ are bounded and the rest is at least $1+\alpha>1$ H\"older continuous. Moreover since $s\mapsto \nabla_{y\z} F(s, Y, \Wtau)$ is continuous and $\mu_n^{Y, \Wtau}\rightharpoonup 0$, summing over \eqref{theo:funcIto4SolXW:eq:4} along the sequence of partitions $\partition$ from Corollary \ref{cor:2varXW} it holds that
\begin{align*}
 \lim_{n\to\infty} \sum_{[s,t]\in\partition_n} \la \nabla_{y\z} F(s, Y, \Wtau), (Y(t)-Y(s))\otimes(\Wtau(t) - \Wtau(s))) \ra  = 0.
\end{align*}
Considering the last summand \eqref{theo:funcIto4SolXW:eq:5}, it is well known that rough integration against a continuous semimartingales with It\^o lift coincides $\prob$-\as\ with It\^o integration whenever both are well-defined, compare \parencite[Lemma 4.35]{chevyrevCanonicalRDEsGeneral2019}. To apply this we first note that $\prob$-\as, $[\Wtau] = \tau - \tau(0) = \la \Wtau\ra$ and so by \parencite[Lemma 5.4, Remark 5.7]{RoughBook}, 
\begin{align}
    \sym{\W^{\circ\tau}}= \frac12 \lb (\Wtau(t)-\Wtau(s))^{\otimes 2} - (\tau(t)-\tau(s))\diag(\one_d)\rb.
\end{align}
Recalling $\nabla_\z^2 F(s, Y, \Wtau)\in\Sym{d}$, the last summand in \eqref{theo:funcIto4SolXW:eq:5} equals $\la \nabla_\z^2 F(s, Y, \Wtau),\W^{\circ\tau}(s,t)\ra$.
Thus \eqref{theo:funcIto4SolXW:eq:5} is exactly the approximation for the rough integral of $(\nabla_\z F(\id, Y,\Wtau), \nabla_\z^2 F(\id, Y,\Wtau))$ against the It\^o lift of $\Wtau$. Moreover since $\Wtau$ is a martingale with respect to  the time-changed filtration $(\F_{\tau(t)})_{t\in[0,T]}$ and $\nabla_\z F(\id, Y, \Wtau)$ is again a continuous causal functional (hence adapted to $(\F_{\tau(t)})_{t\in[0,T]}$), it follows that $\int_0^T \la\nabla_\z F(\id, Y, \Wtau),\dd\Wtau\ra$ is a well-defined It\^{o} integral.\\
It is left to consider \eqref{theo:funcIto4SolXW:eq:2}. Note again that since $\nabla^2_yF(s, Y, \Wtau)\in\Sym{n}$ and by associativity of the symmetric tensor product, it holds that
\begin{align*}
    \la\nabla^2_yF(s, Y, \Wtau), f(Y(s))^{\otimes 2} \W(s,t)\ra &= \la\nabla^2_yF(s, Y, \Wtau), \sym{f(Y(s))^{\otimes 2} \W(s,t)}\ra\\
    &= \la\nabla^2_yF(s, Y, \Wtau), \sym{f(Y(s))^{\otimes 2} \sym{\W(s,t)}}\ra.
\end{align*}
Recalling the notation from Section \ref{sec:condProblem:sec:stateEq} and using that $\sym{\W(s,t)}=\frac12((W(t)-W(s))^{\otimes 2}-(t-s)\diag(\one_d))$,
\begin{align}
	&\langle\nabla_y F(s, Y, \Wtau), \int_s^t f(Y) \dd\Wb \rangle + \frac12 \la \nabla^2_y F(s, Y,\Wtau), \lb \int_s^t f(Y) \dd\Wb \rb^{\otimes 2} - \int_s^t f(Y)^{\otimes 2} \dd\lambda\ra\notag\\
	\begin{split}\label{theo:funcIto4SolXW:eq:6}
	&= \Xi_{s,t} + \la \nabla_y F(s, Y, \Wtau), \int_s^t f(Y)\dd\Wb- f(Y(s))(W(t)-W(s))- \nabla_y f(Y(s))\otimes f(Y(s))\W(s,t)\ra\\
	&\quad + \frac12\la\nabla^2_y F(s, Y, \Wtau), \lb R^{\int f(Y)\dd\Wb, W}(s,t)\rb^{\otimes 2} +  R^{\int f(Y)\dd\Wb, W}(s,t) \otimes f(Y(s))(W(t)-W(s))\ra\\
	&\quad - \frac12 \la\nabla^2_y F(s, Y, \Wtau), \int_s^t f(Y)^{\otimes 2}\dd\lambda - f(Y(s))^{\otimes 2}(t-s)\diag(\one_d)\ra, 
	\end{split}
\end{align}
with $\Xi(s,t)=Z(s)(W(t)-W(s)) + Z'(s)\W(s,t)$ writing $(Z,Z')$ as in the theorem formulation \eqref{theo:funcIto4SolXW:Z}.
$\Xi$ is the approximation sequence for integrating the product of the $W$-controlled rough paths $\nabla_y F(\id, Y, \Wtau)$ and $f(Y)$ against $W$. As stated in \parencite[Corollary 7.4]{RoughBook} a Leibniz rule holds for controlled rough paths. 
Because we rely on estimates for $\int \la \nabla_yF, f\ra\dd\Wb$ we give next give a proof in our setting.
Recalling the notation and identifications from Section \ref{sec:notation}, it holds $Z\colon[0,T]\to \R^d$ and $Z'\colon[0,T]\to\Lin(\R^d,\R^d)$.
Then it follows from Remark \ref{rem:nablaFcontrolled} (clearly $\vert \nabla_y F(\id, Y,\Wtau)\vert \le \vert \nabla F(\id, Y,\Wtau)\vert$ etc.) and estimates \eqref{estimate:sigma(Y)} that
\begin{align*}
\begin{split}
\vert Z'(t)-Z'(s)\vert 
&\le \vert\nabla_y^2 F\vert_\infty \vert f(Y(t))-f(Y(s))\vert +  \vert f\vert_\infty^2 \vert \nabla_y^2 F(t,Y,\Wtau)-\nabla^2_y F(s,Y,\Wtau) \vert\\
&\quad +\vert\nabla_y F\vert_\infty \vert \nabla_yf(Y(t))f(Y(t))-\nabla_yf(Y(s))f(Y(s))\vert\\
&\quad + \vert\nabla_yf\vert_\infty \vert\nabla_yF(t,Y, \Wtau) - \nabla_yF(s,Y,\Wtau)\vert\lesssim \vert t-s\vert^\alpha,
\end{split}
\end{align*}
and
\begin{align*}
\begin{split}
\vert R^{Z, W}(s,t)\vert 
&\le \vert f\vert_\infty \vert R^{\nabla F(Y,\Wtau), (Y,\Wtau)}(s,t)\vert + \vert\nabla^2 F\vert_\infty\vert Y(t)-Y(s)\vert \vert f(Y(t))-f(Y(s))\vert\\
&\quad + \vert\nabla^2 F\vert_\infty\vert f\vert_\infty \vert R^{Y,W}(s,t)\vert + \vert \nabla F\vert_\infty \vert R^{f(Y), W}(s,t)\vert \lesssim \vert t-s\vert^{\alpha+\alpha^2},
\end{split}
\end{align*}
for a constant depending on properties of $F$, $f$, the initial condition $x$ as well as the H\"older norms $\vert (Y,\Wtau)\vert_\alpha$, $(1+\vert f(Y(0))\vert + T^\alpha\vert \nabla_yf(Y)\cdot f(Y)'\vert_\alpha + \vert R^{Y,W}\vert_{2\alpha})^2$ and $\vert W\vert_\alpha$.
Since $2\alpha+\alpha^2>1$, Theorem \ref{theo:RI} shows $\prob$-\as\ that
\begin{align*}
	\lim_{\vert\partition\vert\to 0} \sum_{[s,t]\in\partition} \Xi(s,t) \to \int_0^T \la \nabla_y F(\id, Y, \Wtau), f(Y)\ra \dd\Wb.
\end{align*}
The sum over all the other terms in \eqref{theo:funcIto4SolXW:eq:6} converge $\prob$-\as\ to zero, using the regularity assumptions on $F$ and the estimates from Theorem \ref{theo:RI} and \ref{theo:sol2cSDE}. 
\end{proof}

\section{}\label{sec:appendixB}

\begin{proof}[Proof of Lemma \ref{lem:condDist}]
	We construct a conditional distribution $\prob_{\omega,\z}{}$ for $\prob$
	and $(\omega,\z)\in\Omega\times C([0,T],\R^d)$, such that for all $(\omega, \z)$, $\prob_{\omega,\z}$ is a probability measure on $C([0,T],\R^n)\times L^0([0,T],\R^m)\times \RP_{\tau(0)}$; for every $E\in\borelinfty(C([0,T],\R^n)\otimes L^0([0,T],\R^m)\otimes \mathfrak C_{\tau(0)}$, $(\omega, \z)\mapsto\prob_{\omega,\z}(E)$ is $\F_{\tau(0)}\otimes\borelsup{T}{d}$ measurable and $\prob_{\omega, \z}(S)=1$. Note that the last expression is well defined since $S$ is a measurable set by Lemma \ref{lem:solutionset}. For $D\in\borelsup{T}{d}$, we then set
	\begin{align}\label{lem:condDist:eq:1}
		\cP{Y^\phi,\phi,\Wtaunullb, \Wtau}{\F_{\tau(0)}}(E\times D, \omega):= \int_D \prob_{\omega, \z}(E) \prob^{\Wtau}(\dd \z).
	\end{align}	
	Then for every $\omega\in\Omega$, $\cP{Y^\phi,\phi,\Wtaunullb, \Wtau}{\F_{\tau(0)}}(\cdot, \omega)$ extends to a probability measure on $\borel(U_{\tau(0)})$ with 
	\begin{align}\label{lem:condDist:eq:2}
		\cP{Y^\phi,\phi,\Wtaunullb, \Wtau}{\F_{\tau(0)}}(S\times C^\alpha, \omega)= \int_{C^\alpha}\prob_{\omega,\z}(S)\prob^{\Wtau}(\dd\z)= \prob^{\Wtau}(C^\alpha)=1
	\end{align}
	since $W$ has $\prob$-\as\ $\alpha$-H\"older continuous sample paths and $\tau$ is $C^1$. Hence this will show part $i)$ of the lemma. 
	
	For the construction we consider the extended probability space $\hat{\Omega}= \Omega\times C([0,T],\R^d)$ with $\sigma$-algebra $\hat{\F}=\F\otimes\borelsup{T}{d}$ and probability measure $\hat{\prob}$ via setting for rectangular sets $\hat{\prob}(G\times D) = \prob(G\cap \Wtau\in D)$, where $G\in\F$, $D\in\borelsup{T}{d}$. Then
	\begin{align*}
		(Y^\phi, \phi)(\omega, \z):= (Y^\phi(\omega),\phi(\omega)),\quad \pi_{\z}(\omega, \z):= \z,\quad \Wtaunullb(\omega, t):= \Wtaunullb(\omega)
	\end{align*}
	are random variables from the extended measurable space to a Polish space. Hence there exists a conditional distribution $\hat{\prob}_{Y^\phi,\phi\mid\pi_{\z},\Wtaunullb}\colon \hat{\Omega}\times \borelsup{T}{n}\otimes \borel_{d_{KF}}(L^0([0,T],\R^m))\to[0,1]$. In particular by Definition \ref{defi:cP}, we know for every $A\in\borelsup{T}{n}$, $B\in\borel_{d_{KF}}(L^0([0,T],\R^m))$ that $\hat{\prob}_{Y^\phi,\phi\mid\pi_{\z},\Wtaunullb}(A\times B, \cdot)$ is a representative of 
	\begin{align}\label{lem:condDist:eq:3}
		\hat{\expE}(\one_{(Y^\phi, \phi)^{-1}(A\times B)}\mid \pi_{\z}, \Wtaunullb).
	\end{align}
	We next define 
	\begin{align*}
		\prob_{\omega, \z}:= \hat{\prob}_{Y^\phi,\phi\mid\pi_{\z},\Wtaunullb}(\cdot, (\omega, \z))\otimes \delta_{\Wtaunullb(\omega)}
	\end{align*} 
	Then it is clear that for every $(\omega,\z)$, $\prob_{\omega,\z}$ is a probability measure on $\borelinfty(C([0,T],\R^n)\otimes \borel_{d_{KF}}(L^0([0,T],\R^m)\otimes \mathfrak{C}_{\tau(0)}$. Moreover for every $C\in\mathfrak{C}_{\tau(0)}$, 
	\begin{align*}
		(\omega, \z)\mapsto \hat{\prob}_{Y^\phi,\phi\mid\pi_{\z},\Wtaunullb}(A\times B, (\omega, \z)) \one_{C}(\Wtaunullb(\omega))
	\end{align*}
	is $\F_{\tau(0)}\otimes\borelsup{T}{d}$ measurable (clearly $\sigma(\pi_z,\Wtaunullb)\subseteq\F_{\tau(0)}\otimes\borelsup{T}{d}$) and, recalling \eqref{lem:condDist:eq:3}, $\prob_{\omega, \z}(A\times B\times C)$ is a representative of 
	\begin{align}\label{lem:condDist:eq:4}
		\hat{\expE}(\one_{(Y^\phi, \phi)^{-1}(A\times B)}\mid \pi_{\z}, \Wtaunullb)\one_{C}(\Wtaunullb
        ) = \hat{\expE}(\one_{(Y^\phi, \phi, \Wtaunullb)^{-1}(A\times B\times C)}\mid \pi_{\z}, \Wtaunullb).
	\end{align}
	A monotone class argument shows part $ii)$ of Definition \ref{defi:cP} for general events $E\in\borelinfty(C([0,T],\R^n)\otimes \borel_{d_{KF}}(L^0([0,T],\R^m))\otimes\mathfrak{C}_{\tau(0)}$.  
	In particular we deduced for the solution set $S$, that $\prob_{\omega, \z}(S)$ is a representative of 
	\begin{align*}
		\hat{\expE}(\one_{(Y^\phi, \phi, \Wtaunullb)^{-1}(S)}\mid \pi_{\z}, \Wtaunullb) = 1,
	\end{align*}
	since $\hat{\prob}((Y^\phi, \phi, \Wtaunullb)^{-1}(S))=\prob((Y^\phi, \phi, \Wtaunullb)^{-1}(S))=1$ by construction \eqref{solutionsset}.
	Let $ N:= \{(\omega, \z)\colon \prob_{\omega, \z}(S)<1\}$, then $ N\in\F_{\tau(0)}\otimes\borelsup{\tau(0)}{d}$ and $\hat{\prob}( N)=0$. Choose $\varphi\in L^0([0,T],\R^m)$ and redefine for $(\omega,\z)\in N$,
	\begin{align}\label{lem:condDist:eq:5}
		\prob_{\omega,\z}(A\times B\times C):= \delta_{(F(\varphi, \Wtaunullb(\omega)), \varphi, \Wtaunullb(\omega)))}(A\times B\times C).
	\end{align}
	Meaning that we fix any admissible strategy $\varphi$ and consider the Dirac measure in $(Y^{\varphi}(\omega), \varphi,\Wtaunullb(\omega))$. Then $(\omega,\z)\mapsto (Y^\varphi, \varphi, \Wtaunullb)$ is $\F_{\tau(0)}\otimes\borelsup{T}{d}$ measurable, too. 
	It is easy to check that the redefined $\prob_{\omega,\z}$ has the properties claimed above.
	
	We next show that \eqref{lem:condDist:eq:1} indeed defines a conditional distribution of $(Y^\phi,\phi,\Wtaunullb, \Wtau)$ given $\F_{\tau(0)}$. Note first that since $(\omega,\z)\mapsto \prob_{\omega,\z}(E)$ is $\F_{\tau(0)}\otimes\borelsup{T}{d}$ measurable, $\omega\mapsto \cP{Y^\phi,\phi,\Wtaunullb, \Wtau}{\F_{\tau(0)}}(E\times D, \omega)$ is $\F_{\tau(0)}$ measurable. Moreover for fixed $\omega\in\Omega$, $\cP{Y^\phi,\phi,\Wtaunullb, \Wtau}{\F_{\tau(0)}}(\cdot, \omega)$ is probability measure. To show that $\cP{Y^\phi,\phi,\Wtaunullb, \Wtau}{\F_{\tau(0)}}(E\times D, \cdot)$ is a representative of 
	\begin{align*}
		\expE(\one_{(Y^\phi,\phi,\Wtaunullb,\Wtau)^{-1}(E\times D)}\mid \F_{\tau(0)}),
	\end{align*}
	let $G\in\F_{\tau(0)}$ and consider
	\begin{align*}
		\expE(\cP{Y^\phi,\phi,\Wtaunullb, \Wtau}{\F_{\tau(0)}}(E\times D, \cdot)\one_G)= \int_\Omega\int_{C([0,T],\R^d)} \prob_{\omega,\z}(E)\one_{G\times D}(\omega, \z) \prob^{\Wtau}(\dd \z)\prob(\dd\omega).
	\end{align*}
	By construction \eqref{Wtau}, $\Wtau$ is independent of $\F_{\tau(0)}$. Therefore $\hat{\prob}_{\mid \F_{\tau(0)}\otimes\borelsup{T}{d}} = \prob_{\mid\F_{\tau(0)}}\otimes\prob^{\Wtau}$. 
	Recalling \eqref{lem:condDist:eq:4}, it follows that
	\begin{align*}
		&\int_\Omega\int_{C([0,T],\R^d)} \prob_{\omega,\z}(E)\one_{G\times D}(\omega, \z) \prob^{\Wtau}(\dd \z)\prob(\dd\omega)
		= 
		\int_{\hat{\Omega}} \hat{\expE}(\one_{(Y^\phi,\phi,\Wtaunullb)^{-1}(E)}\mid \pi_\z, \Wtaunullb) \one_{G\times D} \dd\hat{\prob}\\
		&= \hat{\prob}((Y^\phi,\phi,\Wtaunullb)^{-1}(E)\times C[0,T])) \cap (G\times D)) 
		= \prob((Y^\phi,\phi,\Wtaunullb)^{-1}(E)\cap G\cap (\Wtau)^{-1}(D)),
	\end{align*}
	where we used that $G\times D\in \F_{\tau(0)}\otimes\borelsup{T}{d}$ and the definition of $\hat{\prob}$. A monotone class argument shows part $ii)$ of Definition \ref{defi:cP} for general elements in $\borel(U_{\tau(0)})$. Finally we define the conditional distribution of $(Y^\phi,\phi,\Wtau)$ given $\F_{\tau(0)}$ via setting  
	\begin{align*}
		\cP{Y^\phi,\phi,\Wtau}{\F_{\tau(0)}}(A\times B\times D,\omega):= \cP{Y^\phi,\phi,\Wtaunullb,\Wtau}{\F_{\tau(0)}}(A\times B\times \RP_{\tau(0)}\times D,\omega)
	\end{align*}
	
	It is left to show assertions $ii)$ and $iii)$. Recalling \eqref{lem:condDist:eq:1} and \eqref{lem:condDist:eq:5}, it holds for $f=\one_{A\times B\times C\times D}$ that
	\begin{align*}
		&\int f(\y,\varphi,\wb,\z) \cP{Y^\phi,\phi,\Wtaunullb,\Wtau}{\F_{\tau(0)}}(\dd(\y,\varphi,\wb,\z), \omega) \\
		&= \int_D (\hat{\prob}_{Y^\phi,\phi\mid\pi_{\z},\Wtaunullb}(A\times B, (\omega, \z))
		\one_{\hat{\Omega}\setminus  N} + \delta_{F(\varphi,\Wtaunullb(\omega))}(A\times B)\one_{ N})\one_{C}(\Wtaunullb(\omega))\ \prob^{\Wtau}(\dd\z)\\
		&= \int \one_{A\times B\times D}(\y,\varphi, \z)\ \cP{Y^\phi,\phi,\Wtau}{\F_{\tau(0)}}(\dd(\y,\varphi,\z), \omega)\cdot \one_{C}(\Wtaunullb(\omega))\\
		&= \int f(\y,\varphi,\wb,\z)\cP{Y^\phi,\phi,\Wtau}{\F_{\tau(0)}}(\dd(\y,\varphi,\z), \omega)\mid_{\wb = \Wtaunullb(\omega)}.
	\end{align*}
	Then a monotone class argument allows to first consider simple $f = \one_{H}$ for $H\in\borel(U_{\tau(0)})$ and then part $ii)$ of the lemma.
	Similarly, assertion $iii)$ follows from the fact that for $f=\one_{C\times D}$ it holds,
	\begin{align*}
		\int f(\wb,\z) \cP{Y^\phi,\phi,\Wtaunullb, \Wtau}{\F_{\tau(0)}}(\dd (\y,\varphi,\wb,\z), \omega)
		&= \cP{Y^\phi,\phi,\Wtaunullb, \Wtau}{\F_{\tau(0)}}(C([0,T],\R^2)\times C\times D, \omega)\\
		&= \delta_{\Wtaunullb(\omega)}\otimes\prob^{\Wtau}(C\times D).
	\end{align*}
\end{proof}

\subsection*{Acknowledgement} Peter Bank gratefully acknowledges funding by the Deutsche Forschungsgemeinschaft (DFG, German Research Foundation) – CRC/TRR 388 ``Rough Analysis, Stochastic Dynamics and Related Fields'' – Project ID 516748464.
\printbibliography
\printindex

\end{document}